\documentclass[12pt]{amsart}
\newtheorem{thm}{Theorem}[section]
\newtheorem{lem}[thm]{Lemma}
\newtheorem{cor}[thm]{Corollary}
\newtheorem{prop}[thm]{Proposition}
\newtheorem{definition}[thm]{Definition}

\newtheorem*{ack*}{Acknowledgment}
\usepackage[bookmarks]{hyperref}
\usepackage{graphicx,amscd,amsthm,amsfonts,amsopn,amssymb,verbatim}

\def\half{ \frac{1}{2}}
\def\D{\partial}

\def\R{{\mathbb R}}

\def\N{{\mathbb N}}
\def\C{{\mathbb C}}
\def\nint{\mathop{\diagup\kern-13.0pt\int}}
\def\Z{{\mathbb Z}}

\def\supp{{\operatorname{supp}}}

\def\bas{\begin{align*}}
\def\eas{\end{align*}}
\def\bi{\begin{itemize}}
\def\ei{\end{itemize}}

\def\emph#1{{\it #1}}

\def\FF{{\mathcal F}}
\def\MM{{\mathcal M}}
\def\eps{{\epsilon}}
\def\OO{{\mathcal O}}
\theoremstyle{definition}
\newtheorem{rem}[thm]{Remark}
\numberwithin{equation}{section}

\begin{document}

\begin{abstract}

We prove local well-posedness for the gravity water waves equations without surface tension, with initial velocity field in $H^s$, $s > \frac{d}{2} + 1 - \mu$, where $\mu = \frac{1}{10}$ in the case $d = 1$ and $\mu = \frac{1}{5}$ in the case $d \geq 2$, extending previous results of Alazard-Burq-Zuily. The improvement primarily arises in two areas. First, we perform an improved analysis of the regularity of the change of variables from Eulerian to Lagrangian coordinates. Second, we perform a time-interval length optimization of the localized Strichartz estimates.
\end{abstract}

\title{Low regularity solutions for gravity water waves}

\author{Albert Ai}
\email{aai@math.berkeley.edu}
\thanks{The author was supported by the Department of Defense (DoD) through the National Defense Science and Engineering Graduate Fellowship (NDSEG) Program, as well as by the Simons Foundation.}
\date{July 11, 2018}

\maketitle

\tableofcontents

\section{Introduction}

We consider the Cauchy problem for the gravity water waves equations in arbitrary dimension, without surface tension, over an incompressible, irrotational fluid flow.

Consider a time dependent fluid domain $\Omega$ contained in a fixed domain $\OO$, located between a free surface and a fixed bottom. We assume that the bottom is separated from the free surface by a strip of fixed length. More precisely, let 
$$\Omega = \{(t, x, y) \in [0, 1] \times \OO \ ; \ y < \eta(t, x)\}$$
where $\OO \subseteq \R^d \times \R$, $d \geq 1$, is a given connected open set, with $x \in \R^d$ representing the horizontal spatial coordinates and $y \in \R$ representing the vertical spatial coordinate. The free surface
$$\Sigma = \{(t, x, y) \in [0, 1] \times \R^d \times \R: y = \eta(t, x)\}$$
is separated from the bottom $\Gamma = \D\Omega \backslash \Sigma$ by a curved strip of width $h > 0$:
\begin{equation}\label{width}
\{(x, y) \in \R^d \times \R: \eta(t, x) - h < y < \eta(t, x) \} \subseteq \OO.
\end{equation}

We consider an incompressible, irrotational fluid flow influenced by gravity but not by surface tension. In this setting the fluid velocity field $v$ may be given by $\nabla_{x, y} \phi$ where the velocity potential $\phi: \Omega \rightarrow \R$ is harmonic. The water waves system is given by
\begin{equation}\label{euler}
\begin{cases}
\D_t \phi + \dfrac{1}{2} |\nabla_{x, y}\phi|^2 + P + gy = 0 \qquad \text{in } \Omega, \\
\D_t \eta = \D_y \phi - \nabla_x \eta \cdot \nabla_x \phi \qquad \text{on } \Sigma, \\
P = 0 \qquad \text{on } \Sigma, \\
\D_\nu \phi = 0 \qquad \text{on } \Gamma,
\end{cases}
\end{equation}
where $g > 0$ is acceleration due to gravity, $\nu$ is the normal to $\Gamma$, and $P$ is the pressure, recoverable from the other unknowns by solving an elliptic equation. Here the first equation is the Euler equation in the presence of gravity, the second is the kinematic condition ensuring fluid particles at the interface remain at the interface, the third indicates no surface tension, and the fourth indicates a solid bottom.

Many well-posedness results have been obtained for the system (\ref{euler}). We refer the reader to \cite{alazard2014cauchy}, \cite{alazard2014strichartz}, \cite{lannes2013water} for a more complete history and references. In our direction, well-posedness of (\ref{euler}) in Sobolev spaces was first established by Wu \cite{wu1997well}, \cite{wu1999well}. This well-posedness result was improved by Alazard-Burq-Zuily to a lower regularity at the threshold of a Lipschitz velocity field using energy estimates \cite{alazard2014cauchy}. In the $d = 1$ setting, this was sharpened to velocity fields with only a BMO derivative by Hunter-Ifrim-Tataru \cite{hunter2016two}.

It has been known that by taking advantage of dispersive effects, one can go below the regularity threshold attained by energy estimates. This was first approached in the context of low regularity Strichartz estimates for the wave equation in the works of Bahouri-Chemin \cite{bahouri1999equations}, Tataru \cite{tataru2000strichartz}, \cite{tataru2001strichartz}, \cite{tataru2002strichartz}, and Smith-Tataru \cite{smith2005sharp}. 

The approach of low regularity Strichartz estimates was applied by Alazard-Burq-Zuily in \cite{alazard2014strichartz} to prove well-posedness for the gravity water waves equations at a regularity threshold below Lipschitz velocity field. However, the Strichartz estimates were proven with a loss of derivative relative to the Strichartz estimates that hold for the corresponding linearized equations. The goal of this article is to improve the Strichartz estimates with respect to this derivative loss, thereby improving the well-posedness threshold.

We remark that Strichartz estimates have also been studied for the water waves equations with surface tension. See \cite{christianson2010strichartz}, \cite{alazard2011strichartz}, \cite{de2015paradifferential}, \cite{de2016strichartz}, \cite{nguyen2017sharp}.

\subsection{Reformulation of the Equations}

To state our main result we reformulate (\ref{euler}) in terms of a system on the free surface following Zakharov \cite{zakharov1968stability} and Craig-Sulem \cite{craig1993numerical}, with unknowns $(\eta, \psi)$ (henceforth, $\nabla = \nabla_x$):
\begin{equation}\label{zak}
\begin{cases}
\D_t \eta - G(\eta) \psi = 0 \\
\displaystyle \D_t \psi + g\eta + \half |\nabla \psi|^2 - \half \frac{(\nabla \eta \cdot \nabla \psi + G(\eta) \psi)^2}{1 + |\nabla \eta|^2} = 0.
\end{cases}
\end{equation}
Here, $\eta$ is the vertical position of the fluid surface as before, 
$$\psi(t, x) = \phi(t, x, \eta(t, x))$$
is the velocity potential $\phi$ restricted to the surface, and $G(\eta)$ is the Dirichlet to Neumann map with boundary $\eta$:
$$(G(\eta)\psi)(t, x) = \sqrt{1 + |\nabla \eta|^2} \D_n \phi |_{y = \eta(t, x)}.$$
See \cite{alazard2011water}, \cite{alazard2014cauchy} for a precise construction of $G(\eta)$ in a domain with a general bottom. In addition, it was shown in \cite{alazard2013water} that if a solution $(\eta, \psi)$ of (\ref{zak}) belongs to $C^0([0, T]; H^{s + \half}(\R^d))$ for $T > 0$ and $s > \frac{d}{2} + \half$, then one can define a velocity potential $\phi$ and a pressure $P$ satisfying (\ref{euler}). We will consider the case $s > \frac{d}{2} + \half$ throughout.

\subsection{Well-posedness Result}

We will state our well-posedness result in terms of the horizontal and vertical components of the velocity field restricted to the surface $\eta$:
$$V(t, x) = (\nabla_x \phi)|_{y = \eta(t, x)}, \quad B(t, x) = (\nabla_y \phi)|_{y = \eta(t, x)}.$$
The traces $(V, B)$ can be expressed directly in terms of $(\eta, \psi)$:
$$B = \frac{\nabla \eta \cdot \nabla \psi + G(\eta) \psi}{1 + |\nabla \eta|^2}, \quad \nabla \psi = V + B \nabla \eta.$$
We also define the Taylor coefficient 
$$a(t, x) = -(\D_y P)|_{y = \eta(t, x)}$$
which similarly may be defined in terms of $\psi, \eta$ (see \cite{alazard2014cauchy}). 

Our main result is a local well-posedness result for the system (\ref{zak}), with Strichartz estimates: 

\begin{thm}\label{lwp}
Let $d \geq 1$,  
\begin{equation*}
\begin{cases}
\mu = \frac{1}{10}, \ p = 4 \qquad \text{if } d = 1 \\
\mu = \frac{1}{5}, \ p = 2 \qquad \text{if } d \geq 2,
\end{cases}
\end{equation*}
and
$$s > \frac{d}{2} + 1 - \mu, \quad 1 < r < s - \left( \frac{d}{2} - \mu \right).$$
Consider initial data $(\eta_0, \psi_0) \in H^{s + \half}(\R^d)$ satisfying
\begin{enumerate}
\item $(V_0, B_0) \in H^s(\R^d),$
\item the constant width condition
$$\{(x, y) \in \R^d \times \R: \eta_0(x) - h < y < \eta_0(x) \} \subseteq \OO,$$
\item the Taylor sign condition
$$a_0(x) \geq c > 0$$
for some $c > 0$. 
\end{enumerate}

Then there exists $T > 0$ such that the system (\ref{zak}) with initial data $(\eta_0, \psi_0)$ has a unique solution $(\eta, \psi) \in C([0, T]; H^{s + \half}(\R^d))$ such that
\begin{enumerate}
\item $(\eta, \psi) \in L^p([0, T]; W^{r + \half, \infty}(\R^d))$,
\item $(V, B) \in C([0, T]; H^s(\R^d)) \cap L^p([0, T]; W^{r, \infty}(\R^d))$,
\item the constant width condition (\ref{width}) holds on $t \in [0, T]$ with $h/2$ in place of $h$,
\item the Taylor sign condition $a(t, x) \geq c/2$ holds on $t \in [0, T]$.
\end{enumerate}
\end{thm}

\begin{rem}
The now-classical Taylor sign condition expresses the fact that the pressure increases going from the air to the fluid domain. It is satisfied in the case of infinite bottom \cite{wu1999well} or small perturbations of flat bottoms \cite{lannes2005well}. In the case $d = 1$, also see \cite{hunter2016two} and \cite{harrop2017finite} for alternative proofs of this fact with infinite bottom and flat bottom, respectively. The water waves system is known to be ill-posed when the condition is not satisfied \cite{ebin1987equations}.
\end{rem}

Theorem \ref{lwp} is a consequence of the Strichartz estimates stated below in Theorem \ref{p1}, through standard contraction and limiting arguments. The well-posedness and corresponding Strichartz estimates were previously proven with $\mu = \frac{1}{24}$ in the case $d = 1$ and $\mu = \frac{1}{12}$ in the case $d \geq 2$ \cite{alazard2014strichartz}. On the other hand, the linearized equations satisfy Strichartz estimates with $\mu = \frac{1}{8}$ in the case $d = 1$ and $\mu = \frac{1}{4}$ in the case $d \geq 2$, indicating the furthest this line of inquiry can reach.

The details of the reduction from well-posedness to Strichartz estimates are discussed in \cite[Chapter 3]{alazard2014strichartz}. The argument there applies equally well for values of $\mu$ up to those associated with the Strichartz estimates for the linearized equations. In particular, the argument may be applied for our values of $\mu$, so henceforth we will focus on the statement and proof of Theorem \ref{p1}.

\subsection{Paradifferential Reduction}

We will use the paradifferential reduction of the water waves system developed in \cite{alazard2014cauchy}, \cite{alazard2014strichartz} to state and prove Strichartz estimates. We recall the reduction in this section.

Denote the principal symbol of the Dirichlet to Neumann map,
$$\Lambda(t, x, \xi) = \sqrt{(1 + |\nabla \eta|^2)|\xi|^2 - (\nabla \eta \cdot \xi)^2},$$
and the good unknown of Alinhac,
$$U_s = \langle D_x \rangle^s V + T_{\nabla \eta} \langle D_x \rangle^s B.$$
Symmetrization and complexification of the system will result in the symbol
$$\gamma = \sqrt{a \Lambda}$$
and the unknown
\begin{equation}\label{complexu}
u = \langle D_x \rangle^{-s} (U_s - iT_{\sqrt{a/\Lambda}} \langle D_x \rangle^s \nabla \eta).
\end{equation}
Throughout, $\FF$ will denote a non-decreasing positive function which may change from line to line, and denote
$$\FF(s, r)(t) = \FF(\|(\eta, \psi)(t)\|_{H^{s + \half}}, \|(V, B)(t)\|_{H^s}) (1 + \|\eta(t)\|_{W^{r + \half, \infty}} + \|(V, B)(t)\|_{W^{r, \infty}}).$$

Then we have the following paradifferential reduction of (\ref{zak}) (see Appendix \ref{paracalcnotation} for the definition and notation of the paradifferential calculus):
\begin{prop}[{\cite[Corollary 2.7]{alazard2014strichartz}}]\label{paralinearization}
Let 
$$0 < T \leq 1,\quad s > \frac{d}{2} + \frac{3}{4}, \quad 1 < r < 1 + s - \left(\frac{d}{2} + \frac{3}{4}\right).$$ 
Consider a smooth solution $(\eta, \psi) \in C^1([0, T]; H^{s + \half}(\R^d))$ to (\ref{zak}) satisfying, uniformly on $t \in [0, T]$, (\ref{width}) and the Taylor sign condition $a(t, \cdot) \geq c > 0$.

Then $u$ given by (\ref{complexu}) satisfies
\begin{equation}\label{paralinearized}
\D_t u + T_V \cdot \nabla u + i T_\gamma u = f
\end{equation}
where for each $t \in [0, T]$,
$$\|f(t)\|_{H^s(\R^d)} \leq \FF(s, r)(t).$$
\end{prop}

Observe that the principal term in (\ref{paralinearized}) is the transport term $T_V \cdot \nabla u$ of order 1, and the dispersive term $T_\gamma u$ is only of order $\half$. To account for the fact that the highest order term is not dispersive, we will perform a change of variables to Lagrangian coordinates.

\subsection{Strichartz Estimates}\label{secstrichintro}

In this section we state our main result, a Strichartz estimate for solutions $u$ to the equation (\ref{paralinearized}). Let 
\begin{equation}\label{constf}
\FF(s, r, T) = \FF(M_s(T) + Z_r(T))
\end{equation}
where denoting $I = [0, T]$, and setting $p = 4$ if $d = 1$ and $p = 2$ if $d \geq 2$,
\begin{align*}
M_s(T) &:= \|(\psi, \eta, B, V)\|_{L^\infty(I; H^{s + \half} \times H^{s + \half} \times H^s \times H^s)} \\
Z_r(T) &:= \|\eta\|_{L^p(I; W^{r + \half, \infty})} + \|(B, V)\|_{L^p(I;W^{r, \infty} \times W^{r, \infty})}.
\end{align*}

We have the following Strichartz estimate:
\begin{thm} \label{p1}
Remain in the setting of Proposition \ref{paralinearization}. Consider a smooth solution $u$ to (\ref{paralinearized}). Let $\eps > 0$, $s_0 = \frac{1}{10}$, $s_1 \in \R$, and
\begin{equation*}
\begin{cases}
\mu = \frac{1}{10}, \ p = 4 \qquad \text{if } d = 1 \\
\mu = \frac{1}{5}, \ p = 2 \qquad \text{if } d \geq 2. 
\end{cases}
\end{equation*}
Then
$$\|u\|_{L^p(I; W^{s_1 - \frac{d}{2} + \mu - \eps, \infty}(\R^d))} \leq \FF(s, r, T) \left( \|f\|_{L^1(I; H^{s_1 - s_0}(\R^d))} + \|u\|_{L^\infty(I;H^{s_1}(\R^d))}\right).$$
\end{thm}

\begin{rem}
We make the following observations concerning the Strichartz estimates:
\begin{itemize}
\item For comparison, solutions to the constant coefficient linearized equation satisfy
$$\|e^{-it|D_x|^\half} u_0\|_{L^2([0, 1]; L^\infty(\R^2))} \lesssim \|u_0\|_{H^{\frac{3}{4}}(\R^2)}$$
or a gain of $\mu = \frac{1}{4}$ over Sobolev embedding. In the case $d = 1$, the analogous gain on the $L^4L^\infty$ estimate is $\mu = \frac{1}{8}$. 

\item Such an estimate was established in \cite{alazard2014strichartz} with $\mu = \frac{1}{24}$ in the case $d = 1$ and $\mu = \frac{1}{12}$ in the case $d \geq 2$.

\item Because we measure $f$ on the right hand side in $L_t^1$ rather than $L_t^p$, the tame estimates in \cite{alazard2014strichartz} depending linearly on H\"older norms may be relaxed to estimates depending quadratically on the H\"older norms. However, this does not affect the analysis, as the Strichartz estimates limit us to the regime $s > \frac{d}{2} + \frac{3}{4}$, on which the linear tame estimates already hold.
\end{itemize}
\end{rem}

As Theorem \ref{lwp} follows from Theorem \ref{p1} and the approach in \cite[Chapter 3]{alazard2014strichartz}, the remainder of the paper will be occupied by the proof of Theorem \ref{p1}. Here we outline the argument, also following \cite{alazard2014strichartz}, highlighting the differences which we adopt here.

The first step is to reduce the Strichartz estimate to a frequency localized form. This step is standard, and performed in Section \ref{freqloc}.

The second step is to straighten the vector field $\D_t + V \cdot \nabla$ to remove the non-dispersive principal term in (\ref{paralinearized}). This is done via a change of variables in Section \ref{symbolregsec}, obtained by solving the flow $\dot{X} = V(t, X)$, essentially passing to Lagrangian coordinates. This method was already applied in \cite{alazard2014strichartz}.

We will observe some structure of the vector field $V$ to discover extra regularity in the flow $X$, which was a limiting factor in the Strichartz estimates proved in \cite{alazard2014strichartz}. We discuss this structure in Section \ref{COVsec}. As part of this analysis, we will be required to estimate our errors, and in particular the paralinearization error of the Dirichlet to Neumann map, in H\"older norm. These estimates are established in Appendices \ref{ellipticsection} and \ref{holdersec}.

The third step is to construct a parametrix and prove Strichartz estimates for the equation after the change of variables. Here, we apply a wave packet parametrix, which has the advantage that it only requires control of the Hamilton flow on the $\lambda$-frequency wave packet scale associated to our dispersive operator, $\Delta x \approx \lambda^{-3/4}$. This is discussed in Section \ref{strichsec}.

As in \cite{alazard2014strichartz}, the low regularity of our symbol will limit our Strichartz estimate, without derivative loss, to short time intervals. The fourth step is to partition the unit time interval into these short time intervals, implying a Strichartz estimate on a unit time interval with a loss depending on the number of intervals. We obtain a gain by performing this partition in a way that balances the the length of the time interval with the regularity of our symbol on that interval (also see \cite{tataru2002strichartz}). This step is performed in Section \ref{timeintervaldecompsection}.

\subsection{Acknowledgments}

The author would like to thank his advisor, Daniel Tataru, for introducing him to this research area and for many helpful discussions.

\section{Frequency Localization}\label{freqloc}

In this section we reduce Theorem \ref{p1} to the corresponding frequency localized form. 

\subsection{Dyadic Decomposition}

We recall the standard Littlewood-Paley decomposition. Fix $\varphi(\xi) \in C_0^\infty(\R^d)$ with support in $\{|\xi| \leq 2\}$ such that $\varphi \equiv 1$ on $\{|\xi| \leq 1\}$. Then for $\lambda \in 2^\Z$, define 
$$\widehat{S_\lambda u}(\xi) := (\varphi(\xi/\lambda) - \varphi(2\xi/\lambda))\widehat{u}(\xi) =: \psi(\xi/\lambda) \widehat{u} =: \psi_\lambda (\xi) \widehat{u}$$
which has support $\{\lambda/2 \leq |\xi| \leq 2\lambda\}$. Also allow $S_{< \lambda}$, etc., in the natural way, and denote $S_{0} = S_{<1}$. 

We reduce Theorem \ref{p1} to the corresponding frequency dyadic estimates:

\begin{prop} \label{p2}
Remain in the setting of Theorem \ref{p1}. Consider a smooth solution $u_\lambda$ to (\ref{paralinearized}) where $u = u_\lambda(t, \cdot)$ and $f = f_\lambda$ have frequency support $\{|\xi| \approx \lambda\}$. Then
\begin{align*}
\|u_\lambda\|_{L^p(I; W^{s_1 - \frac{d}{2} + \mu - \eps, \infty}(\R^d))} \leq  \FF(s, r, T)\left( \|f_\lambda\|_{L^1(I; H^{s_1 - s_0}(\R^d))} + \|u_\lambda \|_{L^\infty(I;H^{s_1}(\R^d))}\right).
\end{align*}
\end{prop}

\begin{proof}[Proof of Proposition \ref{p1}] Given $u$ solving (\ref{paralinearized}), $u_\lambda = S_\lambda u$ solves (\ref{paralinearized}) with inhomogeneity
$$S_\lambda f + [T_V \cdot \nabla, S_\lambda]u + i[T_\gamma, S_\lambda]u.$$
Note this has frequency support $\{|\xi| \approx \lambda\}$ by the paradifferential calculus. Let $\tilde{S}_\lambda = \sum_{\lambda/4 \leq \mu \leq 4\lambda} S_\mu$. By (\ref{sobolevcommutator}) and the fact that $S_\lambda (T_V \cdot \nabla u) = S_\lambda (T_V \cdot \nabla \tilde{S}_\lambda u)$,
\begin{align*}
\|[T_V \cdot \nabla, S_\lambda]u\|_{L^1(I;H^{s_1 - s_0})} &\lesssim \|V\|_{L^1(I;W^{1,\infty})} \|\tilde{S}_\lambda u\|_{L^\infty(I;H^{s_1})}.
\end{align*}
Similarly, by (\ref{sobolevcommutator}) and Corollary \ref{gammabd},
$$\|[T_\gamma, S_\lambda]u\|_{H^{s_1 - s_0}} \lesssim M_\half^\half(\gamma) \|\tilde{S}_\lambda u\|_{H^{s_1}} \leq \FF(s, r)(t) \|\tilde{S}_\lambda u\|_{H^{s_1}}$$
and hence
$$\|[T_\gamma, S_\lambda]u\|_{L^1(I;H^{s_1 - s_0})} \leq \FF(s, r, T)\|\tilde{S}_\lambda u\|_{L^\infty(I;H^{s_1})}.$$
We then decompose $u$ into frequency pieces $u_\lambda$ on which we can apply Proposition \ref{p2} with $\eps/2$ in place of $\eps$:
\begin{align*}
\|u\|_{L^p(I; W^{s_1 - \frac{d}{2} + \mu - \eps, \infty})} &\leq \sum_{\lambda = 0} \|u_\lambda \|_{L^p(I; W^{s_1 - \frac{d}{2} + \mu - \eps, \infty})} \\
&\leq \FF(s, r, T)\sum_{\lambda = 0} \lambda^{-\eps/2}( \|S_\lambda f\|_{L^1(I; H^{s_1 - s_0})} + \|\tilde{S}_\lambda u \|_{L^\infty(I;H^{s_1})}) \\
&\leq \FF(s, r, T)( \| f\|_{L^1(I; H^{s_1 - s_0})} + \| u \|_{L^\infty(I;H^{s_1})})
\end{align*}
as desired.
\end{proof}

\subsection{Symbol Truncation}

Next, we reduce Proposition \ref{p2} to an estimate with frequency truncated symbols. For fixed $0 < \delta < 1$, let $V_\delta = S_{\leq \lambda^\delta}V$ and $\gamma_\delta = S_{\leq \lambda^\delta} \gamma$.

\begin{prop} \label{p3}
Remain in the setting of Theorem \ref{p1}. Let $\frac{9}{10} \leq \delta < 1$. Consider a smooth solution $u_\lambda$ to 
\begin{equation}
(\D_t + T_{V_\delta} \cdot \nabla + iT_{\gamma_\delta}) u_\lambda = f_\lambda
\end{equation}
where $u_\lambda(t, \cdot)$ and $f_\lambda$ have frequency support $\{|\xi| \approx \lambda\}$. Then
\begin{align*}
\|u_\lambda\|_{L^p(I; W^{s_1 - \frac{d}{2} + \mu - \eps, \infty}(\R^d))} \leq  \FF(s, r, T)\left(\|f_\lambda\|_{L^1(I; 
H^{s_1 - s_0}(\R^d))} + \|u_\lambda \|_{L^\infty(I;H^{s_1}(\R^d))}\right).
\end{align*}
\end{prop}

\begin{proof}[Proof of Proposition \ref{p2}]
This will follow from Proposition \ref{p3} and the following estimates, by viewing $(T_V - T_{V_\delta}) \cdot \nabla u_\lambda$ and $(T_\gamma - T_{\gamma_\delta})u_\lambda$ as components of the inhomogeneous term:
\begin{align}
\|(T_V - T_{V_\delta}) \cdot \nabla u_\lambda \|_{L^1(I; H^{s_1 - s_0})} &\lesssim \|V\|_{L^1(I;W^{1,\infty})}\|u_\lambda \|_{L^\infty(I;H^{s_1})}   \label{transtrunc}   \\
\|(T_\gamma - T_{\gamma_\delta})u_\lambda\|_{L^1(I; H^{s_1 - s_0})} &\leq \FF(s, r, T)\|u_\lambda \|_{L^\infty(I;H^{s_1})} \label{disptrunc}
\end{align}

We first prove (\ref{transtrunc}). We express
$$(T_V - T_{V_\delta}) \cdot \nabla u_\lambda = \sum_{\mu = 0} (S_{\lambda^\delta < \cdot \leq \mu/8}V) \cdot \nabla S_\mu u_\lambda.$$
$S_\mu u_\lambda$ vanishes except for a bounded number of terms where $\mu \approx \lambda$, so it suffices to consider only the term where $\mu = \lambda$. Then
\begin{align*}
\|(S_{\lambda^\delta < \cdot \leq \mu/8}V) \cdot \nabla S_\mu u_\lambda\|_{L^1(I; H^{s_1 - s_0})} &\lesssim \lambda^{-s_0}\|(S_{\lambda^\delta < \cdot \leq \lambda/8}V) \cdot \nabla u_\lambda\|_{L^1(I; H^{s_1})} \\
&\lesssim \lambda^{-s_0} \lambda^{1 - \delta} \|V\|_{L^1(I;W^{1, \infty})} \|u_\lambda\|_{L^\infty(I; H^{s_1})}
\end{align*}
where $1 - \delta - s_0 \leq 0$ as desired.

We next prove (\ref{disptrunc}). By (\ref{ordernorm}),
\begin{align*}
\|(T_\gamma - T_{\gamma_\delta}) u_\lambda \|_{H^{s_1 - s_0}} &\lesssim \lambda^{-s_0}M_0^\frac{1}{2}( S_{> \lambda^\delta}\gamma) \|u_\lambda\|_{H^{s_1 + \frac{1}{2}}} \\
&\lesssim \lambda^{-s_0}\lambda^{\half -\half \delta} M_{\half}^{\half}(\gamma) \|u_\lambda\|_{H^{s_1}}.
\end{align*}
Then $\half - \half \delta - s_0 < 0$ is better than needed. Using Corollary \ref{gammabd} and integrating in time yields the desired result.
\end{proof}

\subsection{Pseudodifferential Symbol}

To perform a change of variables and construct a parametrix, it is convenient to replace the paradifferential symbol $T_{\gamma_\delta}$ with the pseudodifferential symbol $\gamma_\delta = \gamma_\delta(t, x, \xi)$, and likewise the paraproduct $T_{V_\delta}$ by $V_\delta$. It is harmless to do so since $u_\lambda$ is frequency localized.  

\begin{prop} \label{redstrich}
Remain in the setting of Theorem \ref{p1}. Let $\frac{9}{10} \leq \delta < 1$. Consider a smooth solution $u_\lambda$ to 
\begin{equation} \label{reduced}
(\D_t + V_\delta \cdot \nabla + i\gamma_\delta(t, x, D)) u_\lambda = f_\lambda
\end{equation}
where $u_\lambda(t, \cdot)$ and $f_\lambda$ have frequency support $\{|\xi| \approx \lambda\}$. Then
\begin{align*}
\|u_\lambda\|_{L^p(I; W^{s_1 - \frac{d}{2} + \mu - \eps, \infty}(\R^d))} \leq  \FF(s, r, T)\left(\|f_\lambda\|_{L^1(I;
H^{s_1 - s_0}(\R^d))} + \|u_\lambda \|_{L^\infty(I;H^{s_1}(\R^d))}\right).
\end{align*}
\end{prop}

\begin{proof}[Proof of Proposition \ref{p3}]
As $(T_{\gamma_\delta} - \gamma_\delta)u_\lambda$ and $(T_{V_\delta} - V_\delta) \cdot \nabla u_\lambda$ have frequency support $\{|\xi| \approx \lambda\}$, we may apply Proposition \ref{redstrich} with inhomogeneity
$$f_\lambda + (T_{\gamma_\delta} - \gamma_\delta)u_\lambda + (T_{V_\delta} - V_\delta) \cdot \nabla u_\lambda.$$
By \cite[Proposition 2.7]{nguyen2015sharp} and Corollary \ref{gammabd},
$$\|(T_{\gamma_\delta} - \gamma_\delta)u_\lambda\|_{H^{s_1 - s_0}} \lesssim M^\half_\half(\gamma)\|u_\lambda\|_{H^{s_1 - s_0}} \leq \FF(s, r)(t) \|u_\lambda\|_{H^{s_1 - s_0}}.$$
Integrating in time, we have
$$\|(T_{\gamma_\delta} - \gamma_\delta) u_\lambda\|_{L^1(I;H^{s_1 - s_0})} \leq \FF(s, r, T) \|u_\lambda\|_{L^\infty(I;H^{s_1 - s_0})}$$
which is better than needed.

Similarly,
\begin{align*}
\|(T_{V_\delta} - V_\delta) \cdot \nabla u_\lambda\|_{H^{s_1 - s_0}} &\lesssim M_1^1(V \cdot \xi) \|u_\lambda\|_{H^{s_1 - s_0}} \lesssim \|V\|_{W^{1, \infty}} \|u_\lambda\|_{H^{s_1 - s_0}}.
\end{align*}
Integrating in time yields the desired estimate.

\end{proof}

%%%%%%%%%%%%%%%%%%%%%

\section{Flow of the Vector Field}\label{COVsec}

\subsection{Integrating the Vector Field}

Let $(\eta, \psi)$ solve (\ref{zak}). Recall that the traces of the velocity field on the surface $(V, B)$ can be expressed directly in terms of $\eta, \psi$, and satisfy additional relations \cite[Proposition 4.3]{alazard2014cauchy}:
$$B = \frac{\nabla \eta \cdot \nabla \psi + G(\eta) \psi}{1 + |\nabla \eta|^2}, \quad \nabla \psi = V + B \nabla \eta$$
\begin{equation}\label{structure}
(\D_t + V \cdot \nabla)\nabla \eta = G(\eta) V + \nabla \eta G(\eta) B + \Gamma_x + \nabla \eta \Gamma_y, \quad G(\eta)B = -\nabla \cdot V - \Gamma_y.
\end{equation}
Here, $\Gamma_y, \Gamma_x$ arise only in the case of finite bottom, and are described and studied in Appendix \ref{bottomsection}.

If the vector field $V$ of (\ref{paralinearized}) arises in such a way from solutions $(\eta, \psi)$, it may be integrated along $\D_t + V \cdot \nabla$ as in the following proposition. This additional structure will imply improved regularity on the flow of the vector field.

Recall we write $\Lambda$ for the principal symbol of the Dirichlet to Neumann map, and write $\eta_\delta = S_{\leq \lambda^\delta}\eta$. The following proposition applies for fixed $t \in I$ so we omit it.

\begin{prop}\label{integrate}
Consider a smooth solution $(\eta, \psi)$ to the system (\ref{zak}). Let $s > \frac{d}{2} + \half$, $r > 1$, and $\alpha \geq \half$. Then
\begin{equation}\label{integratedeqn}
\D_x V_\delta = (\D_t + V_\delta \cdot \nabla) T_{q^{-1}} \D_x \nabla \eta_\delta + g
\end{equation}
where $q = \Lambda I - i\nabla \eta \cdot \xi^T$ is a matrix-valued symbol of order $1$, and $g$ satisfies
\begin{align*}
\|g\|_{W^{\alpha, \infty}(\R^d)} \leq \ &\lambda^{\delta(\alpha - \half)}\FF(\|\eta\|_{H^{s + \half}}, \|(\psi, V, B)\|_{H^\half \times H^s \times H^s}) \\
&\cdot (1 +  \|\eta\|_{W^{r + \half, \infty}})( 1 +  \|\eta\|_{W^{r + \half, \infty}}  +  \|(V, B)\|_{W^{r, \infty}}).
\end{align*}
\end{prop}

\begin{proof}
We may assume $\alpha = \half$, as the general case is then immediate from the fact that $g$ has frequency support $\{|\xi| \lesssim \lambda^\delta\}$, in turn obtained by observing that the other terms of (\ref{integratedeqn}) have frequency support $\{|\xi| \lesssim \lambda^\delta\}$.

We have from (\ref{structure})
$$(\D_t + V \cdot \nabla)\nabla \eta = G(\eta) V - (\nabla \eta) \nabla \cdot V + \Gamma_x.$$

\subsubsection*{Step 1. Paralinearization}

Using the notation
\begin{align*}
R(\eta)V &= (G(\eta) - T_\Lambda )V, \\
R(\nabla \eta, \nabla \cdot V) &= (\nabla \eta) \nabla \cdot V - T_{\nabla \eta} \nabla \cdot V - T_{\nabla \cdot V} \nabla \eta
\end{align*}
for the paralinearization errors, and rearranging,
$$T_\Lambda V - T_{\nabla \eta} \nabla \cdot V = (\D_t + V \cdot \nabla)\nabla \eta - R(\eta)V  + T_{\nabla \cdot V} \nabla \eta + R(\nabla \eta, \nabla \cdot V) - \Gamma_x.$$
We estimate the error terms on the right hand side. By Proposition \ref{DNparalinear}, (\ref{ordernorm}), (\ref{holderparaerror}), and Proposition \ref{roughbottomest} respectively,
\begin{align*}
\|R(\eta)V\|_{W^{\half, \infty}} &\leq \FF(\|\eta\|_{H^{s + \half}}, \|V\|_{H^s})(1 + \|\eta\|_{W^{r + \half, \infty}})(1 + \|\eta\|_{W^{r + \half, \infty}}+ \|V\|_{W^{r, \infty}})  \\
\|T_{\nabla \cdot V} \nabla \eta \|_{W^{\half, \infty}} &\lesssim \|V\|_{W^{1,\infty}} \|\eta \|_{W^{\frac{3}{2}, \infty}} \\
\|R(\nabla \eta, \nabla \cdot V)\|_{W^{\half, \infty}} &\lesssim \| \eta\|_{W^{\frac{3}{2}, \infty}}\|V\|_{W^{1, \infty}} \\
\|\Gamma_x\|_{W^{\half,\infty}} &\leq \FF(\|\eta\|_{H^{s + \half}}, \|(\psi, V, B)\|_{H^\half})(1 + \|\eta\|_{W^{r + \half, \infty}}).
\end{align*}
Write the matrix-valued paradifferential symbol $q := \Lambda I - i\nabla \eta \cdot \xi^T$, which has order 1 and the $x$-regularity of $\nabla \eta$. Then
$$T_q V = (\D_t + V \cdot \nabla)\nabla \eta + g_1$$
where 
\begin{align*}
\|g_1\|_{W^{\half,\infty}} \leq \ &\FF(\|\eta\|_{H^{s + \half}}, \|(\psi, V, B)\|_{H^\half \times H^s \times H^s}) \\
&\cdot (1 +  \|\eta\|_{W^{r + \half, \infty}})( 1 +  \|\eta\|_{W^{r + \half, \infty}}  +  \|(V, B)\|_{W^{r, \infty}}).
\end{align*}
For brevity, denote the right hand side by $E$.

\subsubsection*{Step 2. Inversion of $q$}

To invert $T_q$, note the outer product $\nabla \eta \cdot \xi^T$ has only real eigenvalues, so $\Lambda I - i\nabla \eta \cdot \xi^T = -i(\nabla \eta \cdot \xi^T+ i\Lambda I)$ is invertible except when $\Lambda = 0$. Furthermore, since $\Lambda \geq |\xi|$, we see that for fixed $|\xi| \geq \half$ the inverse is a smooth function of $\nabla \eta$. We write
$$V = T_{q^{-1}}(\D_t + V \cdot \nabla)\nabla \eta + T_{q^{-1}}g_1 + (1 - T_{q^{-1}}T_q)V.$$

We again estimate the error terms on the right hand side. For the first error term, $q^{-1}$ is a symbol of order $-1$ and has the $x$-regularity of $\nabla \eta$, so that by Sobolev embedding with $s - \half > \frac{d}{2}$,
\begin{equation}\label{qcontrol}
M_0^{-1}(q^{-1}) \lesssim \|\nabla \eta\|_{L^\infty}\lesssim \|\eta\|_{H^{s + \half}}.
\end{equation}
Then by (\ref{ordernorm}) and the estimate on $g_1$ in the previous step,
\begin{align*}
\|T_{q^{-1}}g_1\|_{W^{\frac{3}{2}, \infty}} &\lesssim M_0^{-1}(q^{-1}) \|g_1\|_{W^{\half, \infty}} \leq E.
\end{align*}
Similarly, by (\ref{holdercommutator}), we control the second error term by
\begin{align*}
\|(1 - T_{q^{-1}}T_q)V\|_{W^{\frac{3}{2}, \infty}} &\lesssim \left(M_{1/2}^{-1}(q^{-1}) M_{0}^{1}(q) + M_{0}^{-1}(q^{-1}) M_{1/2}^{1}(q)\right)\|V\|_{W^{1, \infty}} \\
& \lesssim \|\eta\|_{W^{r + \half, \infty}}\|V\|_{W^{1, \infty}}.
\end{align*}
We conclude
$$ V = T_{q^{-1}}(\D_t + V \cdot \nabla)\nabla \eta + g_2$$
with
$$\|g_2\|_{W^{\frac{3}{2}, \infty}} \leq E.$$

\subsubsection*{Step 3. Frequency localization and differentiation}

Applying $\D_x S_{\leq \lambda^\delta}$ to both sides of our identity, we have
$$\D_x V_\delta = T_{q^{-1}}\D_x S_{\leq \lambda^\delta} (\D_t + V \cdot \nabla)\nabla \eta + \D_x S_{\leq \lambda^\delta}g_2+ [\D_x S_{\leq \lambda^\delta}, T_{q^{-1}}] (\D_t + V \cdot \nabla)\nabla \eta.$$
For the first error term, we have
$$\|\D_x S_{\leq \lambda^\delta}g_2\|_{W^{\half, \infty}} \lesssim \|g_2\|_{W^{\frac{3}{2}, \infty}}.$$
To estimate the second error term, recall the identity (4.13) in \cite{alazard2014cauchy}
$$(\D_t + V \cdot \nabla) \D_{x_i} \eta = \D_{x_i} B - \nabla \eta \cdot \D_{x_i} V.$$
Then by (\ref{holdercommutator}) and (\ref{qcontrol}),
\begin{align*}
\|[\D_x S_{\leq \lambda^\delta}, T_{q^{-1}}]&(\D_t + V \cdot \nabla)\nabla \eta\|_{W^{\half,\infty}} \\
&\lesssim \left(M_{1/2}^{-1}(q^{-1}) + M_{0}^{-1}(q^{-1})\right) \|\D B - \nabla \eta \cdot \D V\|_{L^\infty} \\
&\lesssim \|\eta\|_{W^{r + \half, \infty}} \left(\| B \|_{W^{1, \infty}} + \|\nabla \eta \|_{L^\infty}\| \D V\|_{L^\infty} \right) \\
&\leq \FF(\|\eta\|_{H^{s + \half}}) \|(V, B)\|_{W^{1, \infty}}\|\eta\|_{W^{r + \half, \infty}}.
\end{align*}
We conclude
$$\D_x V_\delta = T_{q^{-1}}\D_x S_{\leq \lambda^\delta} (\D_t + V \cdot \nabla)\nabla \eta + g_3$$
with 
$$\|g_3\|_{W^{\half,\infty}} \leq E.$$

\subsubsection*{Step 4. Paralinearization of the vector field}

Writing the paraproduct expansion
$$(V \cdot \nabla)\nabla \eta = (T_{V} \cdot \nabla)\nabla \eta + T_{\nabla (\nabla \eta)}\cdot V + R(V, \nabla(\nabla \eta)),$$
we have
$$\D_x V_\delta = T_{q^{-1}}\D_x S_{\leq \lambda^\delta}(\D_t + T_V \cdot \nabla)\nabla \eta + g_3 + T_{q^{-1}}\D_x S_{\leq \lambda^\delta}(R(V, \nabla(\nabla \eta))+ T_{\nabla (\nabla \eta)}\cdot V).$$
Then by (\ref{ordernorm}), (\ref{holderparaerror}), and (\ref{holderparaproduct}),
\begin{align*}
\|T_{q^{-1}}\D_x S_{\leq \lambda^\delta}&(R(V, \nabla(\nabla \eta)) + T_{\nabla (\nabla \eta)}\cdot V)\|_{W^{\half,\infty}} \\
&\lesssim M_0^{0}(q^{-1}\xi) \|S_{\leq \lambda^\delta}(R(V, \nabla(\nabla \eta))+ T_{\nabla (\nabla \eta)}\cdot V)\|_{W^{\half, \infty}} \\
&\lesssim \|\nabla \eta\|_{L^\infty}\|V\|_{W^{1, \infty}}\| \nabla^2\eta\|_{W^{-\half, \infty}} \\
&\lesssim \|\eta\|_{H^{s + \half}} \|V\|_{W^{1, \infty}} \|\eta\|_{W^{\frac{3}{2}, \infty}}.
\end{align*}
We may thus replace $V$ with $T_V$, yielding, for $g_4$ satisfying the same estimate as $g_3$,
$$\D_x V_\delta = T_{q^{-1}}\D_x S_{\leq \lambda^\delta}(\D_t + T_V \cdot \nabla)\nabla \eta + g_4.$$

\subsubsection*{Step 5. Vector field commutator estimate}
Applying Proposition \ref{com} with $m = 0$, $s = \half$ and $\eps = 1$, we may exchange $T_{q^{-1}}\D_x S_{\leq \lambda^\delta}(\D_t + T_V \cdot \nabla)\nabla\eta$ for $(\D_t + V \cdot \nabla)T_{q^{-1}}\D_x \nabla \eta_\delta$ with an error bounded in $W^{\half, \infty}$ by (using again the identity (4.13) from \cite{alazard2014cauchy})
\begin{align*}
M_0^{0}&(q^{-1}\xi ) \|V\|_{W^{1, \infty}} \|\nabla \eta\|_{B^{\half}_{\infty, 1}} + M_0^0((\D_t + V \cdot \nabla)q^{-1} \xi) \|\nabla \eta\|_{W^{\half, \infty}} \\
\lesssim& \ (\|\eta \|_{W^{1, \infty}} \|V\|_{W^{1, \infty}}\|\eta\|_{B^{\frac{3}{2}}_{\infty, 1}} + \|\D B - \nabla \eta \cdot \D V\|_{L^\infty} \| \eta\|_{W^{\frac{3}{2}, \infty}}) \\
\leq& \ \FF(\|\eta\|_{H^{s + \half}})\|(V, B)\|_{W^{1, \infty}}\|\eta\|_{W^{r + \half, \infty}} \leq E.
\end{align*}

\subsubsection*{Step 6. Truncation of the vector field}
Lastly we frequency truncate the vector field, using (\ref{holderproduct}) and (\ref{ordernorm}):
\begin{align*}
\|((S_{> \lambda^\delta}V)\cdot \nabla) &T_{q^{-1}}\D_x \nabla \eta_\delta\|_{W^{\half, \infty}} \\
&\lesssim \|S_{> \lambda^\delta}V\|_{W^{\half, \infty}}M_0^{0}(q^{-1}\xi)\| \nabla \eta_\delta \|_{W^{1, \infty}} + \|S_{> \lambda^\delta}V\|_{L^\infty}M_0^{0}(q^{-1}\xi)\| \nabla \eta_\delta \|_{W^{\frac{3}{2}, \infty}}\\
&\lesssim \|\eta\|_{H^{s + \half}}(\lambda^{\half\delta}\|S_{> \lambda^\delta}V\|_{W^{\half, \infty}}\| \nabla \eta_\delta \|_{W^{\half, \infty}} + \lambda^\delta\|S_{> \lambda^\delta}V\|_{L^\infty}\| \nabla \eta_\delta \|_{W^{\half, \infty}})\\
&\lesssim \|\eta\|_{H^{s + \half}}\|V\|_{W^{r,\infty}}\| \nabla \eta_\delta \|_{W^{\half, \infty}}.
\end{align*}
\end{proof}

\subsection{Regularity of the Flow}

We straighten the vector field $\D_t + V_\delta\cdot \nabla$ by considering the system
\begin{align}\label{theODE}
\begin{cases}
\dot{X}(s)= V_\delta(s, X(s)) \\
X(s_0)= x.
\end{cases}
\end{align}
Since we assume that $V \in L^\infty([0, T] \times \R^d)$, $V_\delta$ is smooth with bounded derivatives so this system has a unique solution defined on $I = [0, T]$, which we denote $X(s, x)$. For emphasis, we may also write $X = X(s)$ or $X(x)$. 
\begin{prop}[{\cite[Proposition 2.16]{alazard2014strichartz}}] \label{p4}
The map $X(s, \cdot)$ is smooth, with the following estimates:
\begin{align} \label{COV}
\|(\D_x X)(s, \cdot) - Id \|_{L^\infty(\R^d)} &\leq \FF(\|V\|_{L^2(I;W^{1, \infty})}) |s - s_0|^{1/2} \\
\label{COV2}\|(\D_x^\alpha X)(s, \cdot)\|_{L^\infty(\R^d)} &\leq \FF(\|V\|_{L^2(I;W^{1, \infty})}) \lambda^{\delta(|\alpha| - 1)} |s - s_0|^{1/2}, \quad |\alpha| \geq 1
\end{align}
\end{prop}

In the case that $V$ arises from a solution to (\ref{zak}), we can improve upon the regularity of $X$ by using the integrability of $V$ along the vector field established in the previous section:
\begin{prop} \label{p6}
Consider a smooth solution $(\eta, \psi)$ to the system (\ref{zak}). Let $s > \frac{d}{2} + \half$ and $r > 1$. There exists $s_0 \in I$ such that for $\half \leq \alpha < 1$,
\begin{align}
\label{COV2.3} \|\D_x X &\|_{L^p(I; C_*^\alpha(\R^d))} \leq \lambda^{\delta(\alpha - \half)}\FF(\|\eta\|_{L^\infty(I;H^{s + \half})}, \|(\psi, V, B)\|_{L^\infty(I;H^s)}, \|V\|_{L^2(I;W^{1, \infty})}) \\
&\cdot (1 +  \|\eta\|_{L^2(I;W^{r + \half, \infty})})( 1 +  \|\eta\|_{L^p(I;W^{r + \half, \infty})}  +  \|(V, B)\|_{L^2(I;W^{r, \infty})}). \nonumber
\end{align}
\end{prop}

\begin{proof}
In the following let $\D_i = \D_{x_i}$, and sum over repeated indices. Also denote $|D|^\alpha S_\mu = |D|^\alpha_\mu$ for brevity. Differentiating the system for $X$, we have
\begin{equation}\label{2ndODE}
\frac{d}{ds} (|D|^\alpha_\mu \D_j X)(s) = |D|^\alpha_\mu((\D_k V_\delta)(s, X(s)) \D_j X^k).
\end{equation}
We decompose the right hand side into paraproducts:
\begin{align*}
|D|^\alpha_\mu(T_{(\D_k V_\delta)(s, X(s))} \D_j X^k +  T_{\D_j X^k}((\D_k V_\delta)(s, X(s))) &+ R((\D_k V_\delta)(s, X(s)), \D_j X^k)) \\
&=: I + II + III.
\end{align*}

We estimate $I$ using (\ref{holderparaproduct0}):
\begin{align*}
\||D|^\alpha_\mu T_{(\D_k V_\delta)(s, X(s))} \D_j X^k\|_{L^\infty_x} &\lesssim \|T_{(\D_k V_\delta)(s, X(s))} \D_j X^k\|_{C_{*,x}^\alpha} \lesssim \|V\|_{W^{1, \infty}}\|\D_j X\|_{C_{*,x}^\alpha}.
\end{align*}
$III$ satisfies the same estimate but using (\ref{holderparaerror}) in place of (\ref{holderparaproduct0}).

To study $II$, use Proposition \ref{integrate} and the fact that $X$ is the flow of $V_\delta$ to write
\begin{align*}
(\D_k V_\delta)(s,X(s)) &= ((\D_t + V_\delta \cdot \nabla)T_{1/q} \D_k \nabla \eta_\delta)(s,X(s))+ g_{k}(s, X(s)) \\
&= \frac{d}{ds}((T_{1/q}\D_k \nabla \eta_\delta)(s,X(s)))+ g_{k}(s, X(s))
\end{align*}
and hence
\begin{align*}
II = \ &\frac{d}{ds}(|D|^\alpha_\mu T_{\D_j X^k}((T_{1/q}\D_k \nabla \eta_\delta)(s,X(s)))) - |D|^\alpha_\mu T_{\frac{d}{ds}\D_j X^k}((T_{1/q}\D_k \nabla \eta_\delta)(s,X(s))) \\
&+ |D|^\alpha_\mu T_{\D_j X^k}(g_{k}(s, X(s))).
\end{align*}
The first term of $II$ will be moved to the left hand side of (\ref{2ndODE}). The second term of $II$ may be estimated using the system for $X$, (\ref{holderparaproduct0}), and (\ref{COV}):
\begin{align*}
\||D|^\alpha_\mu T_{\frac{d}{ds}\D_j X^k}&((T_{1/q}\D_k \nabla \eta_\delta)(s,X(s)))\|_{L^\infty_x} \\
&\lesssim \|T_{\D_\ell V_\delta^k(s, X(s)) \D_j X^\ell}((T_{1/q}\D_k \nabla \eta_\delta)(s,X(s)))\|_{C_{*,x}^\alpha} \\
&\leq \FF(\|V\|_{L^2(I;W^{1, \infty})})\|V\|_{W^{1, \infty}}\|(T_{1/q}\D_x \nabla \eta_\delta)(s,X(s))\|_{C_{*,x}^\alpha}.
\end{align*}
In turn, by the Lipschitz regularity of $X$ from (\ref{COV}), Proposition \ref{chainprop1}, (\ref{ordernorm}), and (\ref{qcontrol}),
\begin{align*}
\|(T_{1/q}\D_x \nabla \eta_\delta)(s,X(s))\|_{C_{*,x}^\alpha} &\leq \FF(\|V\|_{L^2(I;W^{1, \infty})}) \|(T_{1/q}\D_x \nabla \eta_\delta)(s, x)\|_{C_{*,x}^\alpha}\\
&\leq \FF(\|V\|_{L^2(I;W^{1, \infty})}) M_0^0(q^{-1}\xi) \|\nabla \eta_\delta \|_{W^{\alpha, \infty}_x} \\
&\leq \lambda^{\delta(\alpha - \half)}\FF(\|V\|_{L^2(I;W^{1, \infty})})\|\eta\|_{H^{s + \half}}\|\eta\|_{W^{r + \half, \infty}_x}.
\end{align*}
We estimate the third term of $II$ similarly, using the Lipschitz regularity of $X$ from (\ref{COV}), Proposition \ref{chainprop1}, and Proposition \ref{integrate} to see that $g \in W^{\half, \infty}$:
\begin{align*}
\||D|^\alpha_\mu T_{\D_j X^k}(g_{k}(s, X(s)))\|_{L_x^\infty} &\lesssim \|T_{\D_j X^k}(g_{k}(s, X(s)))\|_{C_{*,x}^\alpha} \\
&\leq \FF(\|V\|_{L^2(I;W^{1, \infty})})\|g(s, X(s))\|_{C_{*,x}^\alpha} \\
&\leq \lambda^{\delta(\alpha - \half)}\FF(\|V\|_{L^2(I;W^{1, \infty})})\|g\|_{W^{\half, \infty}_x}.
\end{align*}

Collecting the above estimates for $I, II$ and $III$, we can write (\ref{2ndODE}) as
$$\frac{d}{ds}|D|^\alpha_\mu(\D_j X -  F)(s, x) = G(s, x) + H(s, x)$$
where
$$F = T_{\D_j X^k}(T_{1/q}\D_k \nabla \eta_\delta)(s,X(s)),$$
$$\|G\|_{L_x^\infty} \lesssim \|V\|_{W_x^{1, \infty}}\|\D_j X\|_{C_{*,x}^\alpha},$$
and
\begin{align*}
\|H\|_{L^\infty_x} &\leq \lambda^{\delta(\alpha - \half)}\FF(\|\eta\|_{L^\infty(I;H^{s + \half})}, \|V\|_{L^2(I;W^{1, \infty})})(\|V\|_{W_x^{1, \infty}}\|\eta\|_{W_x^{r + \half, \infty}} + \|g\|_{W_x^{\half, \infty}}).
\end{align*}

Integrating in $s$, we may further write (\ref{2ndODE}) as
$$|D|^\alpha_\mu(\D_j X - F)(s, x) = |D|^\alpha_\mu (\D_j X - F)(s_0, x) + \int_{s_0}^s G(\sigma, x) + H(\sigma, x) \, d\sigma$$
so that
\begin{align*}
\||D|^\alpha_\mu(\D_j X - F)(s)\|_{L_x^\infty} \lesssim \||D|^\alpha_\mu(\D_j X - F)(s_0)\|_{L_x^\infty} + \int_{s_0}^s &\|V\|_{W_x^{1, \infty}}\|\D_j X\|_{C_{*,x}^\alpha} + \|H\|_{L_x^\infty} \, d\sigma.
\end{align*}
Then taking the supremum over $\mu$ and using triangle inequality on the right hand side,
\begin{align*}
\|(\D_j X - F)(s)\|_{C^\alpha_{*, x}} \lesssim \|(\D_j X - F)(s_0)\|_{C^\alpha_{*,x}} + \int_{s_0}^s &\|V\|_{W_x^{1, \infty}}\|\D_j X - F\|_{C^\alpha_{*,x}} \\
&+ \|H\|_{L_x^\infty} + \|V\|_{W_x^{1, \infty}}\|F\|_{C^\alpha_{*,x}} \, d\sigma.
\end{align*}
Then by the Gronwall and H\"older inequalities,
\begin{align*}
\|(\D_j X - F)(s)\|_{C^\alpha_{*,x}} \leq& \ \FF(\|V\|_{L^2(I;W^{1, \infty})}) \\
&\cdot (\|(\D_j X - F)(s_0)\|_{C^\alpha_{*,x}} + \|H\|_{L^1(I;L^\infty)} + \|F\|_{L^2(I;C_*^\alpha)}).
\end{align*}
Finally, integrating in time,
\begin{align}\label{gronwalleqn}
\|(\D_j X - F)(s)\|_{L^p(I;C^\alpha_{*})} \leq& \ T^\frac{1}{p}\FF(\|V\|_{L^2(I;W^{1, \infty})}) \\
&\cdot (\|(\D_j X - F)(s_0)\|_{C^\alpha_{*,x}} + \|H\|_{L^1(I;L^\infty)} + \|F\|_{L^2(I;C_*^\alpha)}). \nonumber
\end{align}
It remains to estimate the terms of the right hand side, and $\|F\|_{L^p(I;C_*^\alpha)}$ on the left, by the right hand side of (\ref{COV2.3}). The term with $H$ is already suitably estimated above, using additionally H\"older in time and the estimate on $g$ from Proposition \ref{integrate}. It remains to study $F$.

$F$ may be estimated in the same way as the second term of $II$:
$$\|F\|_{C_{*,x}^\alpha} = \| T_{\D_j X^k}(T_{1/q}\D_k \nabla \eta_\delta)(s,X(s))\|_{C_{*,x}^\alpha} \leq \lambda^{\delta(\alpha - \half)}\FF(\|V\|_{L^2(I;W^{1, \infty})})\|\eta\|_{H^{s + \half}}\|\eta\|_{W^{r + \half, \infty}}.$$
We conclude
$$\|F\|_{L^p(I;C_*^\alpha)} \leq \lambda^{\delta(\alpha - \half)}\FF(\|V\|_{L^2(I;W^{1, \infty})})\|\eta\|_{L^\infty(I;H^{s + \half})}\|\eta\|_{L^p(I;W^{r + \half, \infty})}$$
as desired, and similarly with $L^2(I;C_*^\alpha)$.

It remains to estimate $\|(\D_j X - F)(s_0)\|_{C_*^\alpha}$. Note there exists $s_0 \in I$ such that 
$$\|\eta(s_0)\|_{W^{r + \half, \infty}}^p \leq T^{-1}\|\eta \|^p_{L^p(I; W^{r + \half, \infty})}.$$
Fixing such an $s_0$, by the previous estimate on $\|F\|_{C_*^\alpha}$,
\begin{align*}
\|(\D_j X - F)(s_0)\|_{C_*^\alpha} &\lesssim 1 + \|F(s_0)\|_{C_*^\alpha} \\
&\leq 1 + \lambda^{\delta(\alpha - \half)}\FF(\|V\|_{L^2(I;W^{1, \infty})})\|\eta(s_0)\|_{H^{s + \half}}\|\eta(s_0)\|_{W^{r + \half, \infty}} \\
&\leq 1 + T^{-\frac{1}{p}}\lambda^{\delta(\alpha - \half)}\FF(\|V\|_{L^2(I;W^{1, \infty})})\|\eta\|_{L^\infty(I;H^{s + \half})}  \|\eta \|_{L^p(I; W^{r + \half, \infty})}.
\end{align*}
Then $T^\frac{1}{p}\|(\D_j X - F)(s_0)\|_{C_*^\alpha}$ from the right hand side of (\ref{gronwalleqn}) is bounded by the right hand side of (\ref{COV2.3}).
\end{proof}

It will be convenient to have estimates on the higher derivatives of $X$, to later see that our operator has a symbol in a classical symbol class:
\begin{prop} \label{p6x2}
Consider a smooth solution $(\eta, \psi)$ to the system (\ref{zak}). Let $s > \frac{d}{2} + \half$ and $r > 1$. There exists $s_0 \in I$ such that for $|\alpha| \geq 2$,
\begin{align*}
\|\D_x^\alpha X &\|_{L^p(I; L^\infty(\R^d))} \leq \lambda^{\delta(|\alpha| - \frac{3}{2})}\FF(\|\eta\|_{L^\infty(I;H^{s + \half})}, \|(\psi, V, B)\|_{L^\infty(I;H^s)}, \|V\|_{L^2(I;W^{1, \infty})}) \\
&\cdot (1 +  \|\eta\|_{L^2(I;W^{r + \half, \infty})})( 1 +  \|\eta\|_{L^p(I;W^{r + \half, \infty})}  +  \|(V, B)\|_{L^2(I;W^{r, \infty})}). \nonumber
\end{align*}
\end{prop}

\begin{proof}
The proof is similar to that of Proposition \ref{p6}, except easier as it does not require the paradifferential calculus, using only the chain and product rules directly. We differentiate both sides of the flow for $V_\delta$,
$$\frac{d}{ds} \D_x^\alpha X = \D_x^\alpha(V_\delta(s, X(s))).$$
On the right hand side, the term for which all the derivatives fall on a single copy of $X$ is treated as a Gronwall term. The term on which all the derivatives fall on $V$ should be handled using Proposition \ref{integrate} as in the proof of Proposition \ref{p6}. The remaining terms are estimated either as with the analogous terms in the proof of Proposition \ref{p6}, or inductively.
\end{proof}

\section{Change of Variables}\label{symbolregsec}

On a sufficiently small time interval $[s_0, s_0 + T']$, (\ref{COV}) implies that $\D_x X$ is invertible. It is also straightforward to check that $x \mapsto X(t, x)$ is proper, so we can conclude by Hadamard's theorem that $x \mapsto X(t, x)$ is a smooth diffeomorphism for each $t \in [s_0, s_0 + T']$. The length of the time interval depends only on $\|V\|_{L^2(I;W^{1, \infty})}$, so we may partition $[0, T]$ into a number of time intervals of length $T'$ on which $x \mapsto X(t, x)$ is a diffeomorphism. Without loss of generality, consider the first interval $[0, T']$.

We now return to the setting of Proposition \ref{redstrich} to perform the change of variables $x \mapsto X(t, x)$. Consider a smooth solution $u_\lambda$ to (\ref{reduced}). Writing
$$v_\lambda(t, y) := u_\lambda(t, X(t, y)),$$
we have by (\ref{theODE}) that
$$\D_t v_\lambda(t, y) = (\D_t u_\lambda)(t, X(t, y)) + V_\delta(t, X(t, y)) \cdot (\nabla u_\lambda)(t, X(t, y))$$
and hence
$$\D_t v_\lambda(t, y) + i(\gamma_\delta u_\lambda)(t, X(t, y)) = f_\lambda(t, X(t, y)).$$

Next, we write the dispersive term in terms of $v_\lambda$. Fix $t$ in the following and omit it for brevity. We have 
$$(\gamma_\delta u_\lambda)(X(y)) = \int e^{i(X(y) - x')\eta} \gamma_\delta(X(y), \eta)u_\lambda(x') \ dx' d\eta.$$
By the frequency support of $u_\lambda$, we may write
$$(\gamma_\delta u_\lambda)(X(y)) = \int e^{i(X(y) - x')\eta} \gamma_\delta(X(y), \eta)\psi_\lambda(\eta) u_\lambda(x') \ dx' d\eta,$$
though abusing notation by writing $\psi$ in place of the appropriate smooth cutoff with broader support. To make the change of variables $x' = X(y')$, we use the following notation:
$$H(y, y') = \int_0^1 (\D_x X)(hy + (1 - h)y') \ dh, \quad M(y, y') = (H(y, y')^t)^{-1},$$
$$J(y, y') = \left| \det \left((\D_xX)(y') \right) \right| |\det M(y, y')|.$$
Then
$$(\gamma_\delta u_\lambda)(X(y)) = \int e^{i(X(y) - X(y'))\eta} \gamma_\delta(X(y), \eta)\psi_\lambda(\eta) u_\lambda(X(y')) \left| \det \left((\D_xX)(y') \right) \right| \ dy' d\eta.$$
Then make a second change of variables $\eta = M(y, y')\xi$, noting the identity $X(y) - X(y') = H(y, y')(y - y')$:
$$(\gamma_\delta u_\lambda)(t, X(y)) = \int e^{i(y - y')\xi} \gamma_\delta(X(y), M(y, y')\xi)\psi_\lambda(M(y, y')\xi)v_\lambda(y') J(y, y') \ dy' d\xi.$$
We thus have
\begin{equation}\label{coveqn}
\D_t v_\lambda(t, y) + i(p(t, y, y', D)v_\lambda)(t, y) = f_\lambda(t, X(t, y))
\end{equation}
where
$$p(t, y, y', \xi) = \gamma_\delta(X(y), M(y, y')\xi)\psi_\lambda(M(y, y')\xi)J(y, y').$$
It remains to study the regularity and curvature properties of $p$ needed for Strichartz estimates.

\subsection{Symbol Regularity}

Recall that $\FF(s, r, T)$ denotes (\ref{constf}).

\begin{prop}\label{p5}
Remain in the setting of Proposition \ref{paralinearization}. There exists $T' > 0$ sufficiently small depending on $\|V\|_{L^2(I;W^{1, \infty})}$ such that for $I = [0, T']$ and $\half \leq \alpha < 1$,
$$\|\D_\xi^\beta p(t, y, y', \xi)\|_{L^p(I; L_\xi^\infty C_{*,y, y'}^\alpha)} \leq \lambda^{\half - |\beta| + \delta(\alpha - \half)}\FF(s, r, T).$$
\end{prop}

\begin{proof}
Choose $T'$ sufficiently small so that $x \mapsto X(t, x)$ is a diffeomorphism for each $t \in I$. Let $m_{ij}(y, y')$ denote the entries of the matrix $M(y, y')$. Then $\D_\xi^\beta p$ is a sum of products of the form, with $\beta_1 + \beta_2 = \beta$,
$$(\D_\xi^{\beta_1} \gamma_\delta)(X(y), M(y, y')\xi)(\D_\xi^{\beta_2}\psi_\lambda)(M(y, y')\xi)P_\beta(m_{ij}(y, y'))J(y, y')$$
where $P_\beta$ is a polynomial of degree $|\beta|$. 

By (\ref{COV}), for $T' > 0$ sufficiently small, $\|M\|, \|M^{-1}\| \leq 1/2$ so that $(\D_\xi^{\beta_2}\psi_\lambda)(M(y, y')\xi)$ and hence $p$ have support $\{|\xi| \approx \lambda\}$. Thus, in the following, $L_\xi^\infty = L_\xi^\infty(\{|\xi| \approx \lambda\})$ unless otherwise specified.

By using the product estimate (\ref{holderproduct}) and recalling the estimates on $\gamma$ from Corollary \ref{gammabd}, it suffices to show the following estimates:
\begin{align}
\label{chain1}\|(\D_\xi^{\beta_1} \gamma_\delta)(X(y), M(y, y')\xi)\|_{L^p(I; L_\xi^\infty C_*^\alpha)} &\leq \lambda^{\half -|\beta_1| + \delta(\alpha - \half)}\FF(s, r, T) \\
\cdot\left(\right.\sum_{|b| \leq |\beta| + 1}\sup_{|\xi| = 1}\|\D_\xi^b \gamma\|_{L^\infty_{t, x}}
+ \ &\|\sup_{|\xi| = 1}\| \D_\xi^b \gamma\|_{W^{\half, \infty}_x}\|_{L^p_t(I)}\left.\right) \nonumber \\
\label{chain0}\|(\D_\xi^{\beta_2} \psi_\lambda)(M(y, y')\xi)\|_{L^p(I; L_\xi^\infty C_*^\alpha)} &\leq \lambda^{-|\beta_2| + \delta(\alpha - \half)}\FF(s, r, T) \\
\label{chain2}\|P_\beta(m_{ij})\|_{L^p(I; C_*^\alpha)} + \|J(y, y')\|_{L^p(I; C_*^\alpha)} &\leq \lambda^{\delta(\alpha - \half)}\FF(s, r, T) \\
\label{chain4}\|(\D_\xi^{\beta_1} \gamma_\delta)(X(y), M(y, y')\xi)\|_{L^\infty_{t, y, y', \xi}} &\lesssim \lambda^{\half -|\beta_1|}\sup_{|\xi| = 1}\|\D_\xi^{\beta_1} \gamma\|_{L^\infty_{t, x}} \\
\label{chain5}\|(\D_\xi^{\beta_2} \psi_\lambda)(M(y, y')\xi)\|_{L^\infty_{t, y, y', \xi}} &\lesssim \lambda^{-|\beta_2|} \\
\label{chain3}\|P_\beta(m_{ij})\|_{L^\infty_{t, y, y'}} + \|J(y, y')\|_{L^\infty_{t, y, y'}} &\leq \FF(\|V\|_{L^2(I;W^{1, \infty})}).
\end{align}

First we show (\ref{chain4}). Using $T' > 0$ sufficiently small so that $\|M\|, \|M^{-1}\| \leq 1/2$, by homogeneity we have
$$\|(\D_\xi^{\beta_1} \gamma_\delta)(X(y), M(y, y')\xi)\|_{L^\infty_{t, y, y', \xi}} \lesssim \lambda^{\half -|\beta_1|}\sup_{|\xi| = 1}\|(\D_\xi^{\beta_1} \gamma_\delta)(X(y), \xi)\|_{L^\infty_{t, y}}.$$
Then (\ref{chain4}) is clear. (\ref{chain5}) is similarly proven.

To see (\ref{chain3}), note that by (\ref{COV}),
\begin{equation}\label{onederiv}
\|m_{ij}(y, y')\|_{L^\infty_{t, y, y'}} \leq \FF(\|V\|_{L^2(I;W^{1, \infty})}).
\end{equation}
Then (\ref{chain3}) holds, as $P_\beta(m_{ij})$ and $J$ are polynomials in $m_{ij}$. 

Next we prove (\ref{chain2}). From (\ref{COV2.3}) and the definition of $M$,
\begin{equation}\label{twoderiv}
\|m_{ij}(y, y')\|_{L^p(I;C_{*,y,y'}^\alpha)} \leq \lambda^{\delta(\alpha - \half)} \FF(s, r, T).
\end{equation}
Then (\ref{chain2}) holds, using the product estimate (\ref{holderproduct}), the fact that $P_\beta(m_{ij})$ and $J$ are polynomials in $m_{ij}$, and (\ref{onederiv}).

To prove (\ref{chain0}), write $(F)_\lambda(\cdot) = F((\cdot)/\lambda)$ so that we have
$$(\D_\xi^{\beta_2} \psi_\lambda)(M(y, y')\xi) = \lambda^{-|\beta_2|}(\D_\xi^{\beta_2} \psi)_\lambda(M(y, y')\xi).$$
Then view $(\D_\xi^{\beta_2} \psi)_\lambda((\cdot)\xi)$ as a smooth function vanishing near 0 to apply the Moser-type estimate (\ref{smoothholder})
$$\|(\D_\xi^{\beta_2} \psi)_\lambda(M(y, y')\xi)\|_{C_{*,y,y'}^\alpha} \leq \FF(\|m_{ij}(y, y')\|_{L_{y, y'}^\infty}) \|m_{ij}(y, y')\|_{C_{*,y,y'}^\alpha}.$$
Then the desired estimate is obtained by taking $L_t^pL_\xi^\infty$ and using (\ref{onederiv}) and (\ref{twoderiv}).

It remains to show (\ref{chain1}). We are in a position to apply Proposition \ref{chainprop2}, by writing $x = (y, y') \in \R^{2d}$, $a(x, \zeta) = (\D_\xi^{\beta_1} \gamma_\delta)(X(y), \zeta)$, and $f(x) = f(y, y') = M(y, y')\xi$. Since the range of $f$ may be assumed to be $\{|\zeta| \approx \lambda\}$ for $T'$ sufficiently small, we may smoothly cut off $a(x, \zeta)$ to have support $\{|\zeta| \approx \lambda\}$. We obtain
\begin{align*}
\|(\D_\xi^{\beta_1} \gamma_\delta)(X(y), M(y, y')\xi)\|_{C_{*, y, y'}^\alpha} &\lesssim 
\sup_{|\zeta| \approx \lambda}\|(\D_\xi^{\beta_1} \gamma_\delta)(X(y), \zeta)\|_{C_{*, y}^\alpha} \\
&+ \sup_{|\zeta| \approx \lambda}\|(\nabla_\xi \D_\xi^{\beta_1} \gamma_\delta)(X(y), \zeta)\|_{L^\infty_{y}} \|M(y, y')\xi\|_{C_{*, y, y'}^\alpha}.
\end{align*}
The second term is estimated as before by taking $L_t^pL_\xi^\infty$ and using homogeneity of $\gamma$ and (\ref{twoderiv}) to obtain a bound by
\begin{align*}
\|(\nabla_\xi \D_\xi^{\beta_1} \gamma_\delta)(X(y), \xi)\|_{L^\infty_{t, y, \xi}} &\|M(y, y')\|_{L^p(I; C_{*, y, y'}^\alpha)} \\
&\leq \lambda^{\half - (|\beta_1| + 1) + \delta(\alpha - \half)} \FF(s, r, T)\sup_{|\xi| = 1}\|\nabla_\xi \D_\xi^{\beta_1} \gamma_\delta\|_{L^\infty_{t, x}}
\end{align*}
which is better than desired.

For the first term, we apply Proposition \ref{chainprop1} and the Lipschitz regularity of $X$ from (\ref{COV}):
\begin{align*}\|(\D_\xi^{\beta_1} \gamma_\delta)(X(y), \zeta)\|_{C_{*, y}^\alpha} &\leq \|(\D_\xi^{\beta_1} \gamma_\delta)(y, \zeta)\|_{C_{*, y}^\alpha} \|\D_x X\|_{L^\infty}^\alpha \\
&\leq \lambda^{\delta(\alpha - \half)} \FF(\|V\|_{L^2(I;W^{1, \infty})}) \|(\D_\xi^{\beta_1} \gamma_\delta)(y, \zeta)\|_{W_y^{\half, \infty}}.
\end{align*}
Then taking $\sup_{|\zeta| \approx \lambda}$, using homogeneity of $\gamma$, and taking $L^p_t$ yields the desired estimate.
\end{proof}

It will be convenient to have the following symbol property for $p$, for later application of the mapping properties of $p(t, y, y', D)$. 

\begin{prop}\label{p5x2}
Remain in the setting of Proposition \ref{paralinearization}. Fix $k \in \N$. There exists $T' > 0$ sufficiently small depending on $\|V\|_{L^2(I;W^{1, \infty})}$ such that for $I = [0, T']$, $|\alpha| + |\alpha'| +|\beta| \leq k$, and $|\alpha|+ |\alpha'| \geq 1$,
\begin{align*}
\|\D_y^{\alpha}\D_{y'}^{\alpha'} \D_\xi^\beta p(t, y, y', \xi)\|_{L^p(I;L^\infty_{y, y', \xi})} \leq \lambda^{\half - |\beta| + \delta(|\alpha| + |\alpha'| - \frac{1}{2})}\FF(s, r, T).
\end{align*}
\end{prop}

\begin{proof}
We consider the case where $|\alpha| \geq 1$ and $\alpha' = 0$. The general case is similar.

Choose $T'$ sufficiently small so that $x \mapsto X(t, x)$ is a diffeomorphism for each $t \in I$. Let $m_{ij}(y, y')$ denote the entries of the matrix $M(y, y')$. Then $\D_\xi^\beta p$ is a sum of products of the form, with $\beta_1 + \beta_2 = \beta$,
$$(\D_\xi^{\beta_1} \gamma_\delta)(X(y), M(y, y')\xi)(\D_\xi^{\beta_2}\psi_\lambda)(M(y, y')\xi)P_\beta(m_{ij}(y, y'))J(y, y')$$
where $P_\beta$ is a polynomial of degree $|\beta|$. 

By (\ref{COV}), for $T' > 0$ sufficiently small, $\|M\|, \|M^{-1}\| \leq 1/2$ so that $(\D_\xi^{\beta_2}\psi_\lambda)(M(y, y')\xi)$ and hence $p$ have support $\{|\xi| \approx \lambda\}$. Thus, in the following, $L_\xi^\infty = L_\xi^\infty(\{|\xi| \approx \lambda\})$ unless otherwise specified.

By using the product estimate (\ref{holderproduct}) and recalling the estimates on $\gamma$ from Corollary \ref{gammabd}, it suffices to show the following estimates, for $|\nu|\geq 1$:
\begin{align*}
\|\D_y^\nu((\D_\xi^{\beta_1} \gamma_\delta)(X(y), M(y, y')\xi))\|_{L^p(I; L^\infty_{y, y', \xi})} &\leq \lambda^{\half - |\beta_1| + \delta(|\nu| - \frac{1}{2})}\FF(s, r, T) \\
\cdot\left(\right.\sum_{|b| \leq k}\sup_{|\eta| = 1}\|\D_\xi^b \gamma\|_{L^\infty_{t, x}}
+ \ &\|\sup_{|\eta| = 1}\| \D_\xi^b \gamma\|_{W^{\half, \infty}_x}\|_{L^p_t(I)}\left.\right) \nonumber \\
\|\D_y^\nu((\D_\xi^{\beta_2} \psi_\lambda)(M(y, y')\xi))\|_{L^p(I; L^\infty_{y, y', \xi})} &\leq \lambda^{- |\beta_2| + \delta(|\nu| - \frac{1}{2})}\FF(s, r, T) \nonumber \\
\|\D_y^\nu P_\beta(m_{ij})\|_{L^p(I; L^\infty)} + \|\D_y^\nu J(y, y')\|_{L^p(I; L^\infty_{y, y', \xi})} &\leq \lambda^{\delta(|\nu| - \frac{1}{2})}\FF(s, r, T) \\
\|\D^{\nu'}_y((\D_\xi^{\beta_1} \gamma_\delta)(X(y), M(y, y')\xi))\|_{L^\infty_{t, y, y', \xi}} &\lesssim \lambda^{\half - |\beta_1| + \delta|\nu'|} \sup_{|\xi| = 1}\|\D_\xi^{\beta_1} \gamma\|_{L^\infty_{t, x}}. \\
\|\D^{\nu'}_y((\D_\xi^{\beta_2} \psi_\lambda)(M(y, y')\xi))\|_{L^\infty_{t, y, y', \xi}} &\lesssim \lambda^{- |\beta_1| + \delta|\nu'|} \\
\|\D^{\nu'}_yP_\beta(m_{ij})\|_{L^\infty_{t, y, y'}} + \|\D^{\nu'}_y J(y, y')\|_{L^\infty_{t, y, y'}} &\leq \lambda^{\delta|\nu'|}\FF(\|V\|_{L^2(I;W^{1, \infty})}).
\end{align*}

Other than the first, these estimates are straightforward, similar to the proof of Proposition \ref{p5}. The difference is that here we apply Proposition \ref{p6x2} in place of Proposition \ref{p6}.

We will only focus on the first estimate. For brevity, write $\beta = \beta_1$. By the chain rule, $\D_y^\nu (\D_\xi^\beta \gamma_\delta)(X(y), M(y, y')\xi)$ consists of terms of the form
$$K := (\D_y^{a} \D_\xi^{\beta + b} \gamma_\delta)(X(y), M(y, y')\xi) \prod_{j = 1}^r  \left(\D^{\ell_j}_yX(y)\right)^{p_j} \left(\D_y^{\ell_j}M(y, y')\xi\right)^{q_j}$$
where 
$$\sum p_j = a, \ \sum q_j = b, \ |\ell_j| \geq 1, \ \sum (|p_j| + |q_j|)\ell_j = \nu.$$

We study the frequency contribution of each member of the product $K$. For the first term, we can remove the $M(y, y')$ as before by choosing $T'$ sufficiently small, and using homogeneity:
\begin{align*}
\|(\D_y^{a} \D_\xi^{\beta + b} \gamma_\delta)(X(y), M(y, y')\xi)\|_{L^\infty_{y, y', \xi}} &\lesssim \lambda^{\half - (|\beta| + b)}\sup_{|\xi| = 1}\|(\D_y^{a} \D_\xi^{\beta + b} \gamma_\delta)(X(y), \xi)\|_{L^\infty_{t, y}}\\
&\leq  \lambda^{\half - |\beta|}\sum_{|b| \leq k}\sup_{|\xi| = 1}\|(\D_y^{a} \D_\xi^{b} \gamma_\delta)(x, \xi)\|_{L^\infty_{t, x}}.
\end{align*}
Assuming for now that $|a| \geq 1$, then by the frequency localization of $\gamma_\delta$,
$$\|(\D_y^{a} \D_\xi^{\beta + b} \gamma_\delta)(X(y), M(y, y')\xi)\|_{L^p(I;L^\infty_{y, y', \xi})} \lesssim \lambda^{\half - |\beta| + \delta(|a| - \half)}\sum_{|b| \leq k}\|\sup_{|\xi| = 1}\| \D_\xi^{b} \gamma_\delta(x, \xi)\|_{W^{\half, \infty}}\|_{L^p(I)}.$$

For the second term, we apply (\ref{COV}) and (\ref{COV2}):
\begin{equation}
\label{Xbd} \prod_{j = 1}^r  \|\D^{\ell_j}_yX(y)\|_{L^\infty_{t, y}}^{p_j} \leq \lambda^{\delta \sum |p_j| (|\ell_j| - 1)}\FF(\|V\|_{L^2(I;W^{1, \infty})}).
\end{equation}

For the third term, from (\ref{COV}), we have
\begin{equation*}
\|\D^{\nu'}_y m_{ij}(y, y')\|_{L^\infty_{t, y, y'}} \leq  \lambda^{\delta|\nu'|}\FF(\|V\|_{L^2(I;W^{1, \infty})}).
\end{equation*}
Applying this,
\begin{equation}
\label{Mbd} \prod_{j = 1}^r \|\D_y^{\ell_j}M(y, y')\xi\|_{L^\infty_{t, y, y', \xi}}^{q_j} \leq  \lambda^{\delta \sum |q_j| |\ell_j|}\FF(\|V\|_{L^2(I;W^{1, \infty})}).
\end{equation}

Then putting the three terms together, the exponent on $\lambda$ is
$$\delta\left(|a| - \half\right) + \delta \sum_{j = 1}^r |p_j| (|\ell_j| - 1) + \delta \sum_{j  =1}^r|q_j||\ell_j| = \delta\left(|a| - \half\right) + \delta (|\nu| - |a|) = \delta\left(|\nu| - \half\right)$$
as desired. 

It remains to consider the case $|a| = 0$, in which case we have $|b| \geq 1$, so that without loss of generality, $|q_1| \geq 1$. We apply (\ref{chain4}) and (\ref{Xbd}) on the first and second terms in $K$. For the third term, we apply (\ref{Mbd}) except on a single copy of $\D_y^{\ell_1}M(y, y')\xi$ for which we apply (\ref{COV2.3}):
$$\|\D_y^{\ell_1}M(y, y')\xi\|_{L^p(I;L^\infty_{y, y', \xi})} \leq \lambda^{\delta(|\ell_1| - \frac{1}{2})} \lambda \FF(s, r, T).$$
Then we would have the desired exponent on $\lambda$, once we account for the extra $\lambda$ term arising from the fact that $\xi \approx \lambda$. Note from the third term of $K$ that this extra factor in fact appears $\sum q_j = b$ times, which is cancelled by the $4\lambda^{-b}$ gain from the first term of $K$.
\end{proof}

\subsection{Time Regularity}

In this section we study the regularity of the symbol $p$ in time.

\begin{prop}\label{timeregularity}
Remain in the setting of Proposition \ref{paralinearization}. Fix $I = [0, T']$ as in Proposition \ref{p5}. Then
$$\|\D_t\D_\xi^2 p(t, y, y', \xi)\|_{L^p(I; L^\infty_{y, y',\xi})} \leq \lambda^{-3/2}\FF(s, r, T)$$
where $L_\xi^\infty = L_\xi^\infty(\{|\xi| \approx \lambda\})$.
\end{prop}

\begin{proof}
Recall
$$p(t, y, y', \xi) = \gamma_\delta(X(t, y), M(y, y')\xi) \psi_\lambda(M(y, y')\xi) J(y, y').$$
By (\ref{COV}), for $T' > 0$ sufficiently small, $\|M\|, \|M^{-1}\| \leq 1/2$. Thus on $\{|\xi| \approx \lambda\}$ with the appropriate constant, $\psi_\lambda(M(y, y')\xi) \equiv 1$. Restricting our attention to this domain,
$$p(t, y, y', \xi) = \gamma_\delta(X(y), M(y, y')\xi) J(y, y').$$
Then, by homogeneity, it suffices to show
$$\|\D_t\D_\xi^2 p(t, y, y', \xi)\|_{L^p(I; L^\infty_{y, y',\xi})} \leq \FF(s, r, T)$$
on $\{|\xi| = 1\}$. 

Let $m_{ij}(y, y')$ be the entries of the matrix $M(y, y')$. Then $\D_\xi^2p$ is a product of three terms,
$$(\D_\xi^2 \gamma_\delta)(X(y), M(y, y')\xi)P_2(m_{ij}(y, y')) J(y, y') $$
where $P_2$ is a polynomial of degree 2. By the product rule, (\ref{COV}), and the $L^\infty$ bound of Corollary \ref{gammabd}, it suffices to study the time derivative on each of $\gamma_{\xi\xi\delta}$, $J$, and $P_2(m_{ij})$ individually. Throughout the proof, since $\|M\|, \|M^{-1}\| \leq 1/2$, we will use the fact that we may scale to $|M(y, y')\xi| = 1$ with an acceptable loss.

Recall 
$$J(y, y') = |\det (\D_x X)(y') | |\det M(y, y')|$$
is a polynomial in $\D_x X$ and the matrix coefficients $m_{ij}$ of $M$. Further, the matrix coefficients $m_{ij}$ are smooth functions of $\D_x X$. Thus using again (\ref{COV}), to study the time derivative of $J$ it suffices to study the time derivative of $\D_x X$, . We have by the flow of $V_\delta$ that
$$\frac{d}{dt} \D_i X = (\D_k V_\delta) (t, X(t)) \D_i X^k$$
and thus
$$\|\frac{d}{dt} \D_i X\|_{L^p(I; L^\infty)} \lesssim \|V\|_{L^p(I; W^{r, \infty})} \|\D_i X^k\|_{L^\infty(I;L^\infty)} \leq \FF(\|V\|_{L^p(I; W^{r, \infty})}).$$
The analysis of $P_2(m_{ij})$ is similar.

Next we turn to $\gamma_{\xi\xi\delta}$. Write
$$\frac{d}{dt}\gamma_{\xi\xi\delta} = \D_t \gamma_{\xi\xi\delta} + \D_t X \cdot \nabla_x \gamma_{\xi\xi\delta} + \D_t M \cdot \nabla_\xi \gamma_{\xi\xi\delta}.$$
Similar to before, we estimate $\D_t M \cdot \nabla_\xi \gamma_{\xi\xi\delta}$ by applying Corollary \ref{gammabd} and noting that $\D_t M$ satisfies the same estimates as $\D_t \D_xX$, discussed above. 

For the first two terms, note by the flow of $V_\delta$ that
$$\D_t \gamma_{\xi\xi\delta} + \D_t X \cdot \nabla_x \gamma_{\xi\xi\delta} = (\D_t + V_\delta \cdot \nabla)\gamma_{\xi\xi\delta}.$$
We need to remove the frequency localization to exploit the material derivative, $\D_t + V \cdot \nabla$.
First, we can replace $V_\delta$ with $V$ via the following estimate:
$$\|(S_{> \lambda^\delta} V) \cdot \nabla \gamma_{\xi\xi\delta} \|_{L^\infty} \lesssim \lambda^\delta \|S_{> \lambda^\delta} V\|_{L^\infty} \|\gamma_{\xi\xi\delta} \|_{L^\infty} \lesssim \|V\|_{W^{1, \infty}} \|\gamma_{\xi\xi}\|_{L^\infty}$$
and thus
$$\|(S_{> \lambda^\delta} V) \cdot \nabla \gamma_{\xi\xi\delta}\|_{L^p(I;L^\infty)} \lesssim \|V\|_{L^p(I; W^{1, \infty})} \|\gamma_{\xi\xi}\|_{L^\infty(I;L^\infty)}$$
as desired (estimating $\gamma_{\xi\xi}$ via Corollary \ref{gammabd} as above). Second, by the commutator estimate
$$\|[V, \nabla S_{\leq \lambda^\delta}] \gamma_{\xi\xi}\|_{L^\infty} \lesssim \|V\|_{W^{r,\infty}} \|\gamma_{\xi\xi}\|_{L^\infty}$$
it remains to bound 
$$S_{\leq \lambda^\delta}(\D_t \gamma_{\xi\xi} + \nabla \cdot (V\gamma_{\xi\xi})).$$
Third, by the product rule, we are reduced to studying
$$(\D_t + V \cdot \nabla) \gamma_{\xi\xi}$$
by observing
$$\|(\nabla \cdot V) \gamma_{\xi\xi}\|_{L^\infty} \leq \|V\|_{W^{1, \infty}} \|\gamma_{\xi\xi}\|_{L^\infty}.$$

To estimate $(\D_t + V \cdot \nabla) \gamma_{\xi\xi}$, it suffices to estimate each of
$$(\D_t + V \cdot \nabla)a, \quad (\D_t + V \cdot \nabla)\Lambda.$$
The former is estimated in \cite[Proposition C.1]{alazard2014strichartz}:
$$\|(\D_t + V\cdot \nabla)a\|_{W^{\eps, \infty}} \leq \FF(\|(\eta, \psi)\|_{H^{s + \half}}, \|(V, B)\|_{H^s})(1 + \|\eta\|_{W^{r + \half, \infty}} + \|(V, B)\|_{W^{r, \infty}}).$$
Then take the $L^p$ integral in time.

For the latter, it suffices to estimate
$$(\D_t + V \cdot \nabla)\nabla \eta = G(\eta) V + \nabla \eta G(\eta)B + \Gamma_x + \nabla \eta \Gamma_y.$$
We may estimate the terms arising from the bottom using Sobolev embedding and \cite[Proposition 4.3]{alazard2014cauchy} (which uses the notation $\gamma := \Gamma_x + \nabla \eta \Gamma_y$):
$$\|\Gamma_x + \nabla \eta \Gamma_y\|_{L^\infty} \lesssim \|\Gamma_x + \nabla \eta \Gamma_y\|_{H^{s - \half}} \leq \FF(\|\eta\|_{H^{s + \half}}, \|(\psi, V, B)\|_{H^\half}).$$
Finally, it remains to estimate $G(\eta)V$ (as $\nabla \eta G(\eta)B$ is similar). We have
$$\|G(\eta) V\|_{L^\infty} \leq \|G(\eta) - T_\Lambda V\|_{L^\infty} + \|T_\Lambda V\|_{L^\infty}.$$
and
$$\|T_\Lambda V\|_{L^\infty} \lesssim M_0^1(\Lambda) \|V\|_{C_*^1} \lesssim \|\nabla \eta\|_{L^\infty} \|V\|_{W^{r, \infty}}.$$
For the paralinearization error, we may use \cite[Theorem 1.4]{alazard2014strichartz} and Sobolev embedding, to obtain
$$\|G(\eta) - T_\Lambda V\|_{L^\infty} \leq \FF(\|\eta\|_{H^{s + \half}}, \|V\|_{H^s})(1 + \|\eta\|_{W^{r + \half, \infty}}+ \|V\|_{W^{r, \infty}}).$$
Taking $L^p$ in time yields the desired estimate.

\end{proof}

\subsection{Curvature Estimates}

We recall the following estimate on the Hessian of $\gamma$:

\begin{prop}[{\cite[Proposition 2.11]{alazard2014strichartz}}]\label{hessian}
There exists $c_0 > 0$ and $\lambda_0 > 0$ such that 
$$|\det \D_\xi^2 \gamma_\delta(t ,x, \xi) | \geq c_0$$
for all $\lambda \geq \lambda_0$ and $(t, x, \xi) \in [0, T] \times \R^d \times \{|\xi| \approx 1\}$.
\end{prop}

We can obtain the same estimates for $p$:

\begin{cor}\label{hessian2}
There exists $T' > 0$ (depending on $\|V\|_{L^2(I;W^{1, \infty})}$), $c_0 > 0$, and $\lambda_0$ such that
$$|\det \D_\xi^2 p(t, y, y', \xi)| \geq c_0\lambda^{-\frac{3d}{2}}$$
for all $\lambda \geq \lambda_0$ and $(t, y, y', \xi) \in [0, T'] \times \R^{2d} \times \{|\xi| \approx \lambda\}$. 
\end{cor}

\begin{proof}
Recall
$$p(t, y, y', \xi) = \gamma_\delta(X(y), M(y, y')\xi) \psi_\lambda(M(y, y')\xi) J(y, y').$$
By (\ref{COV}), for $T > 0$ sufficiently small, $\|M\|, \|M^{-1}\| \leq 1/2$. Thus on $\{|\xi| \approx \lambda\}$ with the appropriate constant, $\psi_\lambda(M(y, y')\xi) \equiv 1$. Restricting our attention to this domain,
$$p(t, y, y', \xi) = \gamma_\delta(X(y), M(y, y')\xi) J(y, y').$$
Thus, we have
$$\D_{\xi_i\xi_j}^2 p(t, y, y', \xi) = M(y, y')^T (\D_{\xi_i\xi_j}^2 \gamma_\delta)(X(y), M(y, y')\xi) M(y, y') J(y, y').$$
By (\ref{COV}), $M(y, y') \approx Id$ and $J(y, y') \approx 1$ on $[0, T']$ for $T'$ sufficiently small. Thus by Proposition \ref{hessian} and homogeneity,
$$|\det \D_\xi^2 p(t ,y, y', \xi) | \gtrsim \inf_{x \in \R^d} |\det \D_\xi^2 \gamma_\delta(t ,x, \xi) | \gtrsim c_0\lambda^{-\frac{3d}{2}}.$$
\end{proof}

\section{Strichartz Estimates for Order 1/2 Evolution Equations}\label{strichsec}

\subsection{The Parametrix Construction}

In this section we consider a general evolution equation of the form
\begin{equation}\label{general}
\begin{cases}
(D_t + a^w(t, x, D))u = f, \qquad &\text{in } (0, 1) \times \R^d \\
u(0) = u_0, \qquad &\text{on } \R^d
\end{cases}
\end{equation}
where $a(t, x, \xi)$ is a real symbol continuous in $t$ and smooth with respect to $x$ and $\xi$. In this setting $a^w$ is self-adjoint and thus generates an isometric evolution $S(t, s)$ on $L^2(\R^d)$. We outline the construction of a phase space representation of the fundamental solution for (\ref{general}), following \cite{koch2005dispersive}, \cite{tataru2004phase}, \cite{marzuola2008wave}. 

The FBI transform \cite{tataru2004phase}
$$(Tf)(x, \xi) = 2^{-\frac{d}{2}} \pi^{-\frac{3d}{4}} \int e^{-\half(x - y)^2}e^{i\xi(x - y)} f(y)\, dy$$
is an isometry from $L^2(\R^d)$ to phase space $L^2(\R^{2d})$ with an inversion formula
$$f(y) = (T^*Tf)(y) = 2^{-\frac{d}{2}} \pi^{-\frac{3d}{4}} \int e^{-\half(x - y)^2}e^{-i\xi(x - y)} (Tf)(x, \xi) \, dx d\xi.$$

First we would like to describe the phase space localization properties of $S(t, s)$ relative to the Hamilton flow corresponding to (\ref{general}),
\begin{equation}\label{hamilton}\begin{cases}
\dot{x} = a_\xi(t, x, \xi) \\
\dot{\xi} = -a_x(t, x, \xi).
\end{cases}\end{equation}
More precisely, let 
$$(x^t, \xi^t) = (x^t(x, \xi), \xi^t(x, \xi))$$
denote the solution to (\ref{hamilton}), with initial data $(x, \xi)$ at time $0$. Let $\chi(t, s)$ denote the family of canonical transformations on phase space $L^2(\R^{2d})$ corresponding to (\ref{hamilton}),
$$\chi(t, s)(x^s, \xi^s) = (x^t, \xi^t).$$
Then we would like an estimate on the kernel $\tilde{K}$ of the phase space operator $TS(t, s)T^*$, of the form
$$|\tilde{K}(t, x, \xi, s, y, \eta)| \lesssim (1 + |(x, \xi) - \chi(t, s)(y, \eta)|)^{-N}.$$

Such an estimate has been established in \cite{tataru2004phase} for the class of symbols $a \in S_{0, 0}^{0, (k)}$ satisfying
$$|\D_x^\alpha \D_\xi^\beta a(t, x, \xi)| \leq c_{\alpha ,\beta}, \qquad |\alpha| + |\beta| \geq k,$$
for $k = 2$. This was generalized in \cite{marzuola2008wave} to the class of symbols $a\in S^{(k)}L_\chi^1$ satisfying
$$\sup_{x, \xi} \int_0^1|\D_x^\alpha \D_\xi^\beta a(t, \chi(t, 0)(x, \xi))| \, dt \leq c_{\alpha, \beta}, \qquad |\alpha| + |\beta| \geq k.$$
For our purposes it suffices to consider an intermediate class of symbols $a\in L^1S_{0, 0}^{0, (k)}$ satisfying
$$\|\D_x^\alpha \D_\xi^\beta a\|_{L^1_t([0, 1]; L^\infty(\R^{2d}))} \leq c_{\alpha, \beta}, \qquad |\alpha| + |\beta| \geq k.$$
Precisely, we have the following corollary of \cite{marzuola2008wave} for this smaller class of symbols:
\begin{thm}\label{bilipthm}
Let $a(t, x, \xi) \in L^1S_{0, 0}^{0, (2)}$. Then
\begin{enumerate}
\item The Hamilton flow (\ref{hamilton}) is well-defined and bilipschitz.
\item The kernel $\tilde{K}(t,s)$ of the phase space operator $TS(t, s)T^*$ decays rapidly away from the graph of the Hamilton flow,
$$|\tilde{K}(t, x, \xi, s, y, \eta)| \lesssim (1 + |(x, \xi) - \chi(t, s)(y, \eta)|)^{-N}.$$
\end{enumerate}
\end{thm}

Then we have the following phase space representation for the exact solution to (\ref{general}):
\begin{thm}\label{repformula}
Let $a(t, x, \xi) \in L^1S_{0, 0}^{0, (2)}$. Then the kernel $K(t, s)$ of the evolution operator $S(t, s)$ for $D_t + a^w$ can be represented in the form 
$$K(t, y, s, \tilde{y}) = \int e^{-\half(\tilde{y} - x^s)^2} e^{-i\xi^s(\tilde{y} - x^s)} e^{i(\psi(t, x, \xi) - \psi(s, x, \xi))} e^{i\xi^t(y - x^t)} G(t, s, x, \xi, y) \, dx d\xi$$
where the function $G$ satisfies
$$|(x^t - y)^\gamma \D_x^\alpha \D_\xi^\beta \D_y^\nu G(t, s, x, \xi, y)| \lesssim c_{\gamma, \alpha, \beta, \nu}.$$
\end{thm}

\begin{proof}
This is a consequence of Theorem \ref{bilipthm} and \cite[Theorem 4]{tataru2004phase}, once we prove that the canonical transformation $\chi(t, s)$ is smooth with uniform bounds,
$$|\D_x^\alpha \D_\xi^\beta \chi(x, \xi)| \leq c_{\alpha, \beta}, \qquad |\alpha| + |\beta| > 0.$$
This can be proven for instance by using the argument in Section 2 of \cite{marzuola2008wave} showing that $\chi$ is uniformly bilipschitz, along with an induction.
\end{proof}

\subsection{Dispersive Estimates}

We combine the representation formula in Theorem \ref{repformula} with a curvature condition to yield a dispersive estimate. 

First we define a class of symbols analogous to the class $\lambda S_\lambda^k$ in \cite{koch2005dispersive}. Define the class of symbols $a(t, x, \xi) \in L^1S_{1, \delta}^{m, (k)}(\lambda)$ satisfying
\begin{align*}
\|\D_x^\alpha \D_\xi^\beta a\|_{L^1_t([0, 1]; L^\infty(\R^{2d}))} &\leq c_{\alpha, \beta} \lambda^{m - |\beta| + \delta(|\alpha| - k)}, \qquad |\alpha| \geq k.
\end{align*}
Note this definition makes sense even for noninteger $k$.

We will work with symbols which also partially satisfy uniform bounds. Define the class of symbols $a \in S_1^m(\lambda)$ satisfying
$$|\D_\xi^\beta a(t, x, \xi)| \leq c_{\beta} \lambda^{m - |\beta|}.$$

\begin{prop}\label{dispersive}
Let $a \in L^1S_{1, \frac{3}{4}}^{\half, (\frac{2}{3})}(\lambda) \cap S_1^\half(\lambda)$ such that for each $(t, x, \xi) \in [0, 1] \times \R^d \times \{|\xi| \approx \lambda\}$, $\D_\xi^2a$ satisfies
$$| \det \D_\xi^2a (t, x, \xi)| \geq c\lambda^{-\frac{3d}{2}}.$$
Also assume that
$$\|\lambda^{3/2}\D_t\D_\xi^2 a\|_{L^1_t([0, 1];L^\infty_{x,\xi})} \leq c_1 \ll c$$ 
where $L_\xi^\infty = L_\xi^\infty(\{|\xi| \approx \lambda\})$.

Let $u_0$ have frequency support $\{|\xi| \approx \lambda\}$. Then there exists $0 < T \leq 1$ such that for all $|t - s| < T$, we have
$$\|S(t, s) u_0\|_{L^\infty} \lesssim \lambda^{\frac{3d}{4}} |t - s|^{-\frac{d}{2}} \|u_0\|_{L^1}.$$
\end{prop}

\begin{proof}
Without loss of generality let $s = 0$. Fix $T$ small to be chosen later. We fix $t_0 \in [0, T]$ and prove the estimate when $t = t_0$. To do so, we reduce the problem to an estimate for $t = 1$ by rescaling. Write $u = S(t, s)u_0$ and set
$$v(t, x) = u(t_0 \cdot t, \lambda^{-3/4}\sqrt{t_0}x).$$
Then $v$ solves
$$(D_t + \tilde{a}^w(t, x, D)) v = 0, \qquad v(0) = u_0(0, \lambda^{-3/4}\sqrt{t_0}x) =: v_0$$
where
$$\tilde{a}(t, x, \xi) = t_0 a\left(t_0 \cdot t, \lambda^{-3/4} \sqrt{t_0} x, \lambda^{3/4}\frac{\xi}{\sqrt{t_0}}\right).$$
Then it suffices to show
$$\|v(1)\|_{L^\infty} \lesssim \|v_0\|_{L^1}.$$

We first show $\tilde{a} \in L^1S_{0, 0}^{0, (2)}$ in order to apply Theorem \ref{repformula}. Indeed,
\begin{align}
\label{tildeabd2}\|\D_\xi^\beta \tilde{a}\|_{L^1_t([0, 1]; L^\infty)} \leq \|\D_\xi^\beta \tilde{a}\|_{L^\infty} &= \lambda^{\frac{3}{4}|\beta|} t_0^{-\half |\beta|}t_0 \|\D_\xi^\beta a \|_{L^\infty} \\
&\leq \lambda^{\frac{3}{4}|\beta|}t_0^{1-\half |\beta|}c_\beta\lambda^{\half - |\beta|} = c_\beta (t_0^{-1} \lambda^{-\half})^{\half (|\beta| - 2)}. \nonumber
\end{align}
Note that when $t_0^{-1} \lambda^{-\half} \geq 1$, the dispersive estimate is trivial by Sobolev embedding, so we may assume $t_0^{-1} \lambda^{-\half} \leq 1$. Thus when $|\beta| \geq 2$,
$$\|\D_\xi^\beta \tilde{a}\|_{L^1_t([0, 1]; L^\infty)} \leq c_\beta.$$
On the other hand, for $|\alpha| \geq 1$, again using $t_0^{-1} \lambda^{-\half} \leq 1$,
\begin{align}
\|\D_x^\alpha\D_\xi^\beta \tilde{a}\|_{L^1_t([0, 1]; L^\infty)} &\leq \lambda^{-\frac{3}{4}|\alpha|} t_0^{\half |\alpha|}\lambda^{\frac{3}{4}|\beta|} t_0^{-\half |\beta|}t_0 \|\D_x^\alpha\D_\xi^\beta a(t_0 \cdot t)\|_{L^1_t([0, 1]; L^\infty)} \nonumber \\
&\leq c_{\alpha, \beta} \lambda^{-\frac{3}{4}|\alpha|} t_0^{\half |\alpha|}\lambda^{\frac{3}{4}|\beta|} t_0^{-\half |\beta|} \lambda^{\half - |\beta| + \frac{3}{4}(|\alpha| - \frac{2}{3})} \nonumber \\
 \label{tildeabd}&= c_{\alpha, \beta} t_0^{\half |\alpha|} (t_0^{-1} \lambda^{-\half})^{\half |\beta|} \leq c_{\alpha, \beta} t_0^{\half |\alpha|} \leq c_{\alpha, \beta}.
\end{align}

Thus, we may use the representation formula in Theorem \ref{repformula},
$$v(t, y) = \int G(t, x, \xi, y) e^{-\half(\tilde{y} - x)^2 + i\xi^t(y - x^t) - i\xi(y - x)+i\psi(t, x, \xi)} v_0(\tilde{y}) \, dx d\xi d\tilde{y}.$$
By the frequency support of $v_0$ in $B = \{|\xi| \approx \lambda^{\frac{1}{4}}t_0^\half\}$, the contribution of the complement of $B$ to the integral is negligible, so it suffices to study
$$\int\int_{B} |G(t, x, \xi, y)| \, d\xi \, e^{-\half(\tilde{y} - x)^2} |v_0(\tilde{y})| \, dx d\tilde{y} \lesssim \|v_0\|_{L^1}\sup_x \int_B |G(t, x, \xi, y)| \, d\xi.$$
It remains to show
$$\int_B |G(1, x, \xi, y)| \, d\xi \lesssim 1.$$
Given the bound for $G$ in Theorem \ref{repformula}, this reduces to
$$\int_B (1 + |x^1 - y|)^{-N} \, d\xi \lesssim 1.$$

To show this, we study the dependence of $x^1 = x^1(x, \xi)$ on $\xi$. Write
$$X = \D_\xi x^t, \quad \Xi = \D_\xi \xi^t,$$
which by the Hamilton flow (\ref{hamilton}) for $\tilde{a}$ solve
\begin{equation}
\begin{cases}
\dot{X} = \tilde{a}_{\xi x} X + \tilde{a}_{\xi \xi} \Xi, \qquad X(0) = 0 \\
\dot{\Xi} = -\tilde{a}_{xx} X - \tilde{a}_{x \xi} \Xi, \qquad \Xi(0) = I.
\end{cases}
\end{equation}
From (\ref{tildeabd}), we have
$$\|\tilde{a}_{x\xi}\|_{L^1_t([0, 1]; L^\infty)} + \|\tilde{a}_{xx}\|_{L^1_t([0, 1]; L^\infty)}\lesssim \sqrt{t_0}.$$
Similarly, (\ref{tildeabd2}) implies $\|\D_{\xi\xi} \tilde{a}\|_{L^\infty} \lesssim 1.$ Further, since the Hamilton flow for $\tilde{a}$ is bilipschitz by Theorem \ref{bilipthm}, we have
$$\|X\|_{L^\infty} + \|\Xi\|_{L^\infty} \lesssim 1.$$
Thus, we have
$$\Xi(t) = \Xi(0) + \int_0^t \dot{\Xi} \, ds = I - \int_0^t \tilde{a}_{xx} X + \tilde{a}_{x \xi} \Xi \, ds = I + O(\sqrt{t_0})$$
and hence
$$\tilde{a}_{\xi \xi} \Xi = \tilde{a}_{\xi \xi} + O(\sqrt{t_0})$$
Then
\begin{equation}\label{Xflow}
X(t) = X(0) + \int_0^t \dot{X} \, ds = \int_0^t \tilde{a}_{\xi x} X + \tilde{a}_{\xi \xi} \Xi \, ds = \int_0^t \tilde{a}_{\xi \xi} \, ds + O(\sqrt{t_0}).
\end{equation}

We would like to replace $\tilde{a}_{\xi \xi} = \tilde{a}_{\xi \xi}(s, x^s, \xi^s)$ in the integral by $\tilde{a}_{\xi \xi}(0, x, \xi)$. Combining (\ref{hamilton}), (\ref{tildeabd}), and (\ref{tildeabd2}), we have
\begin{align*}
\|\tilde{a}_{\xi\xi x}\dot{x}\|_{L^1_t([0, 1]; L^\infty)} = \|\tilde{a}_{\xi\xi x}\tilde{a}_\xi\|_{L^1_t([0, 1]; L^\infty)} &\leq c_{\alpha, \beta} t_0^{\half} (t_0^{-1} \lambda^{-\half})\cdot c_\beta (t_0^{-1} \lambda^{-\half})^{-\half}\leq c_{\alpha, \beta} c_\beta \lambda^{-\frac{1}{4}}\\
\|\tilde{a}_{\xi\xi\xi}\dot{\xi}\|_{L^1_t([0, 1]; L^\infty)}= \|\tilde{a}_{\xi\xi\xi}\tilde{a}_x\|_{L^1_t([0, 1]; L^\infty)} &\leq c_\beta (t_0^{-1} \lambda^{-\half})^{\half} \cdot c_{\alpha, \beta} t_0^{\half} \leq c_{\alpha, \beta} c_\beta \lambda^{-\frac{1}{4}}.
\end{align*}
In addition, recalling the assumption that $\lambda^{3/2}\D_t\D_\xi^2 a \in L^1_t([0, 1];L^\infty),$ we have
$$\|\D_t\D_\xi^2 \tilde{a}\|_{L^1_t([0, 1]; L^\infty)} = \lambda^{3/2}\|\D_t\D_\xi^2 a\|_{L^1_t([0, t_0]; L^\infty)} \leq c_1.$$
Thus,
\begin{align*}
\tilde{a}_{\xi \xi}(s, x^s, \xi^s) &= \tilde{a}_{\xi \xi}(0, x, \xi) + \int_0^s \frac{d}{dt} \tilde{a}_{\xi \xi} \, dr \\
&= \tilde{a}_{\xi \xi}(0, x, \xi) + \int_0^s \D_t\D_\xi^2 \tilde{a} + \tilde{a}_{\xi\xi x}\dot{x} + \tilde{a}_{\xi\xi\xi}\dot{\xi} \, dr \\
&= \tilde{a}_{\xi \xi}(0, x, \xi) + O(c_1) + O(\lambda^{-\frac{1}{4}}).
\end{align*}
We conclude from (\ref{Xflow}) that
$$ \D_\xi x^1 = X(1) = \tilde{a}_{\xi \xi}(0, x, \xi) + O(c_1) + O(\lambda^{-\frac{1}{4}} + \sqrt{t_0}) = \tilde{a}_{\xi \xi}(0, x, \xi) + O(c_1) + O(\sqrt{t_0}).$$

For $t_0 \in [0, T]$, with $T$ chosen sufficiently small,
$$|\det \D_\xi x^1| = |\det \tilde{a}_{\xi \xi}| + O(c_1)+ O(\sqrt{t_0}) \gtrsim 1$$
and hence 
$$\int_B (1 + |x^1 - y|)^{-N} \, d\xi \lesssim \int (1 + |x^1 - y|)^{-N} \, dx^1 \lesssim 1$$
as desired.

\end{proof}

\subsection{Strichartz Estimate}

We establish a Strichartz estimate from the previous dispersive estimate. Say $(p, q, \rho)$ is \emph{admissible} if 
$$\frac{2}{p} + \frac{d}{q} = \frac{d}{2}, \quad 2 \leq p, q \leq \infty, \quad (p, q) \neq (2, \infty), \quad \rho = \frac{3}{2p}.$$

\begin{thm}\label{strich}
Consider a symbol $a(t, x, \xi)$ as in Proposition \ref{dispersive}. Let $u$ and $f_1$ have frequency support $\{|\xi| \approx \lambda\}$ and solve
$$(D_t + a^w)u = f_1 + f_2, \qquad u(0) = u_0$$
on $t \in [0, 1] = I$. Then for $(p, q, \rho)$ admissible we have
$$\||D|^{-\rho} u\|_{L^p(I;L^q)} \lesssim \||D|^\rho f_1\|_{L^{p'}(I;L^{q'})} + \|f_2\|_{L^1(I;L^2)} + \|u_0\|_{L^2}.$$
\end{thm}

\begin{proof}
It suffices to prove this in a sufficiently small time interval $[0, T]$, as we can iterate it to obtain it in the full interval $[0, 1]$.

The proof follows the standard $TT^*$ formalism. By Duhamel's formula, it suffices to show
\begin{align}
\label{T} S(t, s)&: L^2 \rightarrow \lambda^\rho L^pL^q \\
\label{TT} 1_{t > s} S(t, s) &: \lambda^{-\rho} L^{p'}L^{q'} \rightarrow \lambda^\rho L^pL^q.
\end{align}
In turn, by a $TT^*$ argument for (\ref{T}) and the Christ-Kiselev lemma for (\ref{TT}), it suffices to show
$$S(t, s) : \lambda^{-\rho} L^{p'}L^{q'} \rightarrow \lambda^\rho L^pL^q.$$

To show this, we interpolate the trivial energy bound
$$\|S(t, s) \|_{L^2 \rightarrow L^2} \leq 1$$
with the dispersive estimate in Proposition \ref{dispersive}
$$\|S(t, s)\|_{L^1 \rightarrow L^\infty} \lesssim \lambda^{\frac{3d}{4}} |t - s|^{-\frac{d}{2}}$$
and then apply the Hardy-Littlewood-Sobolev inequality. 
\end{proof}

%%%%%%%%%%%%%%%%%%%%%%%%%%%%%%%%%%%%%

\subsection{Symbol Truncation and Rescaling}

We will consider symbols 
$$a(t, x, y, \xi) \in L^1_t ([0, T];L^\infty_\xi \dot{W}^{\frac{2}{3}, \infty}_{x, y})$$
and study the evolution equation (\ref{general}) but with the Kohn-Nirenberg quantization. We can reduce Strichartz estimates in this setting to Theorem \ref{strich} by frequency truncating in the spatial variables. 

In our applications, we will also need to rescale to time intervals of variable length.

\begin{cor}\label{strichcor}
Let $I = [0, T]$ with $T > 0$. Consider a symbol $a(t, x, y, \xi)$ with support on $\{|\xi| \approx \lambda\}$ and satisfying
$$\|\D^\beta_\xi a\|_{L^\infty_{t, x, y, \xi}} \lesssim \lambda^{\half - |\beta|}, \qquad T^{\frac{1}{3}}\|\D^\beta_\xi a\|_{L^1_t(I; L^\infty_\xi \dot{W}^{\frac{2}{3}, \infty}_{x, y})} \lesssim  \lambda^{\half - |\beta|}.$$
For each $(t, x, y, \xi) \in [0, 1] \times \R^{2d} \times \{|\xi| \approx \lambda\}$, assume $\D_\xi^2a$ satisfies
$$| \det \D_\xi^2a (t, x, y, \xi)| \geq c\lambda^{-\frac{3d}{2}}.$$
Also assume that
$$\|\lambda^{3/2}\D_t\D_\xi^2 a\|_{L^1_t(I;L^\infty_{x,y,\xi})} \leq c_1 \ll c$$ 
where $L_\xi^\infty = L_\xi^\infty(\{|\xi| \approx \lambda\})$.

Let $u$ have frequency support $\{|\xi| \approx \lambda\}$ and solve
\begin{equation}\label{basestricheqn}
(D_t + a(t, x, y, D))u = f, \qquad u(0) = u_0.
\end{equation}
Then for $(p, q, \rho)$ admissible we have
\begin{equation}\label{basestrichest}
\||D|^{-\rho} u\|_{L^p(I;L^q)} \lesssim \|f\|_{L^1(I;L^2)} + \|u_0\|_{L^\infty(I;L^2)}.
\end{equation}
\end{cor}

\begin{proof}

By applying the scaling
$$\tilde{a} = a(Tt, T^2x, T^2y, \xi), \quad \tilde{u} = u(Tt, T^2x), \quad \tilde{f} = Tf(Tt, T^2x)$$
and the scaling invariance of the symbol $a$ and the desired estimate (\ref{basestrichest}), we may assume $T = 1$.

Next, we frequency truncate the symbol $a$ in the $x$ and $y$ variables. Write, with $\delta = \frac{3}{4}$,
$$a = S_{\leq \lambda^\delta} a + S_{> \lambda^\delta} a =: a_\delta + a_{> \delta}.$$

Then we may assume $a = a_\delta$, since  viewing $a_{> \delta}(t, x, y, D) u$ as a component of $f$ on the right hand side of (\ref{basestricheqn}),
\begin{align*}
\|a_{> \delta}(t, x, y, D)u\|_{L^1(I;L^2)} &= \lambda^{\half}\|\lambda^{-\half} a_{> \delta}(t, x, y, D)u\|_{L^1(I;L^2)} \\
&\lesssim \lambda^\half \|\lambda^{-\half} a_{> \delta}\|_{L^1(I;L^\infty_{x, y, \xi})} \|u\|_{L^\infty(I;L^2)} \\
&\lesssim \| \lambda^{-\half} a \|_{L^1(I;L_\xi^\infty \dot{W}^{\frac{2}{3},\infty})} \|u\|_{L^\infty(I;L^2)}
\end{align*}
which by assumption is bounded by $\|u\|_{L^\infty(I;L^2)}$ on the right hand side of (\ref{basestrichest}).

Next we pass to the Weyl quantization. Since $a = a_\delta$, we have
$$a \in L^1 S_{1, \delta, \delta}^{\half, (\frac{2}{3})}.$$
We can write (see \cite[Proposition 0.3.A]{taylor1991pseudodifferential})
$$a(t, x, y, D) = a(t, x, x, D) + r(t, x, D), \qquad r \in L^1 S_{1, \delta}^{\half - 1 + \delta(1 - \half), (0)} = L^1 S_{1, \delta}^{-\frac{1}{8}, (0)}$$
so that
$$\|r(t, x, D)u\|_{L^1(I; L^2)} \lesssim \|u\|_{L^\infty(I; L^2)}.$$
Similarly, writing $\tilde{a}(t, x, \xi) = a(t, x, x, \xi)$, we can write
$$\tilde{a}(t, x, D) = \tilde{a}^w(t, x, D) + \tilde{r}(t, x, D), \qquad \tilde{r} \in L^1 S_{1, \delta}^{-\frac{1}{8}, (0)}$$
and hence
$$\|\tilde{r}(t, x, D)u\|_{L^1(I; L^2)} \lesssim \|u\|_{L^\infty(I; L^2)}.$$
Viewing $(r + \tilde{r})u$ as an inhomogeneous term bounded by $\|u\|_{L^\infty(I; L^2)}$, we may assume
\begin{equation}\label{weyleqn}
(D_t + \tilde{a}^w)u = f.
\end{equation}
Since 
$$\tilde{a} \in L^1S_{1, \frac{3}{4}}^{\half, (\frac{2}{3})}(\lambda)$$
and the remaining properties required of $\tilde{a}$ in Proposition \ref{dispersive} are straightforward to check, we may apply Theorem \ref{strich} with $f_2 = f$ and $f_1 = 0$ to yield (\ref{basestrichest}).

\end{proof}

\section{Strichartz Estimates for Rough Symbols}\label{timeintervaldecompsection}

As our symbol $p$ in (\ref{coveqn}) is only in $L^1_t ([0, T];L^\infty_\xi \dot{W}^{\half, \infty}_{x, y})$, we cannot directly apply Corollary \ref{strichcor}. We will consider Strichartz estimates on time intervals of variable length, adapted to our lower regularity.

\subsection{Strichartz Estimates on Variable Time Intervals}

\begin{cor}\label{strichcor2}
Corollary \ref{strichcor} holds with a symbol $a(t, x, y, \xi)$ satisfying
$$T^{\frac{1}{4}}\lambda^{\frac{1}{8}}\|\D^\beta_\xi a\|_{L^1_t(I; L^\infty_\xi \dot{W}^{\half, \infty}_{x, y})} \lesssim  \lambda^{\half - |\beta|}$$
in place of
$$T^{\frac{1}{3}}\|\D^\beta_\xi a\|_{L^1_t(I; L^\infty_\xi \dot{W}^{\frac{2}{3}, \infty}_{x, y})} \lesssim  \lambda^{\half - |\beta|}.$$
\end{cor}

\begin{proof}

We frequency truncate the symbol $a$ in the $x$ and $y$ variables. Set
$$\mu = \|a\|_{L^1_t(I; L^\infty_\xi \dot{W}^{\half, \infty}_{x, y})}^2$$
and write
$$a = S_{\leq \mu} a + S_{> \mu} a =: a_\mu + a_{> \mu}.$$

Then we may assume $a = a_\mu$, since viewing $a_{> \mu}(t, x, y, D)u$ as a component of $f$ on the right hand side of (\ref{basestricheqn}),
\begin{align*}
\|a_{> \mu}(t, x, y, D)u\|_{L^1(I;L^2)} &= \lambda^{\half}\|\lambda^{-\half} a_{> \mu}(t, x, y, D)u\|_{L^1(I;L^2)} \\
&\lesssim \lambda^\half \|\lambda^{-\half} a_{> \mu}\|_{L^1(I;L^\infty_{x, y, \xi})} \|u\|_{L^\infty(I;L^2)} \\
&\lesssim \mu^{-\half}\lambda^\half \|\lambda^{-\half} a \|_{L^1(I;L_\xi^\infty \dot{W}^{\half,\infty})} \|u\|_{L^\infty(I;L^2)}
\end{align*}
which by the choice of $\mu$ is bounded by $\|u\|_{L^\infty(I;L^2)}$ on the right hand side of (\ref{basestrichest}).

Then we have 
$$T^{\frac{1}{3}}\|\D^\beta_\xi a_\mu\|_{L^1_t(I; L^\infty_\xi \dot{W}^{\frac{2}{3}, \infty}_{x, y})} \lesssim \mu^{\frac{1}{6}}T^{\frac{1}{3}}\|\D^\beta_\xi a\|_{L^1_t(I; L^\infty_\xi \dot{W}^{\half, \infty}_{x, y})}.$$
By choice of $\mu$ and assumption,
$$\mu^{\frac{1}{6}} = \|a\|_{L^1_t(I; L^\infty_\xi \dot{W}^{\half, \infty}_{x, y})}^{\frac{1}{3}} \lesssim T^{-\frac{1}{12}} \lambda^{\frac{1}{8}}.$$
Again by assumption,
$$\|\D^\beta_\xi a\|_{L^1_t(I; L^\infty_\xi \dot{W}^{\half, \infty}_{x, y})} \lesssim T^{-\frac{1}{4}}\lambda^{\frac{3}{8} - |\beta|}.$$
Collecting these estimates, we conclude
$$T^{\frac{1}{3}}\|\D^\beta_\xi a_\mu\|_{L^1_t(I; L^\infty_\xi \dot{W}^{\frac{2}{3}, \infty}_{x, y})} \lesssim T^{-\frac{1}{12}} \lambda^{\frac{1}{8}}T^{\frac{1}{3}}T^{-\frac{1}{4}}\lambda^{\frac{3}{8} - |\beta|} = \lambda^{\half - |\beta|}$$
which permits us to apply Corollary \ref{strichcor}.
\end{proof}

\subsection{Time Interval Partition}

We will apply Corollary \ref{strichcor2} to the elements of a partition of our time interval to obtain a Strichartz estimate with loss:

\begin{prop}\label{scalestrichprop}
Let $I = [0, T']$ with $0 < T' \leq 1$. Consider a symbol $a(t, x, y, \xi)$ with support on $\{|\xi| \approx \lambda\}$ and satisfying
$$\|\D^\beta_\xi a\|_{L^\infty_{t, x, y, \xi}} \lesssim \lambda^{\half - |\beta|}, \qquad \|\D^\beta_\xi a\|_{L^1_t(I; L^\infty_\xi \dot{W}^{\half, \infty}_{x, y})} \lesssim  \lambda^{\half - |\beta|}.$$
For each $(t, x, y, \xi) \in [0, T'] \times \R^{2d} \times \{|\xi| \approx \lambda\}$, assume $\D_\xi^2a$ satisfies
$$| \det \D_\xi^2a (t, x, y, \xi)| \geq c\lambda^{-\frac{3d}{2}}.$$
Also assume that
$$\|\lambda^{3/2}\D_t\D_\xi^2 a\|_{L^1_t(I;L^\infty_{x,y,\xi})} \leq c_1 \ll c$$ 
where $L_\xi^\infty = L_\xi^\infty(\{|\xi| \approx \lambda\})$.

Let $u$ have frequency support $\{|\xi| \approx \lambda\}$ and solve
\begin{equation*}
(D_t + a(t, x, y, D))u = f, \qquad u(0) = u_0.
\end{equation*}
Then for $(p, q, \rho)$ admissible we have
\begin{equation}\label{scalestrich}
\||D|^{-\rho} u\|_{L^p(I;L^q)} \lesssim \lambda^{-\frac{1}{10p'}}\|f\|_{L^1(I;L^2)} + \lambda^{\frac{1}{10p}}\|u_0\|_{L^\infty(I;L^2)}.
\end{equation}
\end{prop}

\begin{proof}
Decompose $[0, T']$ into maximal subintervals 
$$0 = t_0 < t_1 <...< t_k = T'$$
satisfying both
\begin{equation}\label{inhomogsplit}
\|f\|_{L^1([t_j, t_{j + 1}]; L^2)} \leq \lambda^{-\frac{1}{10}}\|f\|_{L^1(I; L^2)} 
\end{equation}
and
\begin{equation}\label{coeffsplit}
(t_{j + 1} - t_j)^\frac{1}{4} \lambda^{\frac{1}{8}}\lambda^{-\half + |\beta|}\|\D^\beta_\xi a\|_{L^1_t([t_j, t_{j + 1}]; L^\infty_\xi \dot{W}_{x,y}^{\half, \infty})} \leq 1
\end{equation}
for each $0 \leq |\beta| \leq N$. We claim that the number $k$ of intervals satisfies
$$k \approx \lambda^{\frac{1}{10}}.$$
The lower bound follows from (\ref{inhomogsplit}). For the upper bound, observe that for each $j$, equality must hold either in (\ref{inhomogsplit}) or (\ref{coeffsplit}) for some $\beta$. The number $k_0$ of intervals for which equality in (\ref{inhomogsplit}) holds is at most $\lambda^{\frac{1}{10}}$ as desired. On each of the $k_\beta$ intervals in which equality holds in (\ref{coeffsplit}) with $\beta$, we have, since $c^4 + c^{-1} \gtrsim 1$ for any $c$, 
$$\lambda^{\frac{1}{10}} (t_{j + 1} - t_j) + \lambda^{-\frac{1}{40}}\lambda^{\frac{1}{8}}\lambda^{-\half + |\beta|}\|\D^\beta_\xi a\|_{L^1_t([t_j, t_{j + 1}]; L^\infty_\xi \dot{W}_{x,y}^{\half, \infty})} \gtrsim 1.$$
Then summing over all such intervals, we have
$$\lambda^{\frac{1}{10}}T' + \lambda^{\frac{1}{10}} \gtrsim \lambda^{\frac{1}{10}}\sum_j (t_{j + 1} - t_j) + \lambda^{-\frac{1}{40}}\lambda^{\frac{1}{8}}\sum_j\lambda^{-\half + |\beta|}\|\D^\beta_\xi a\|_{L^1_t([t_j, t_{j + 1}]; L^\infty_\xi \dot{W}_{x,y}^{\half, \infty})} \gtrsim k_\beta.$$
and thus
$$k = k_0 + \sum_\beta k_\beta \lesssim \lambda^{\frac{1}{10}}.$$
as desired. 

We conclude from (\ref{coeffsplit}) that $a$ restricted to $[t_j, t_{j + 1}]$ satisfies the conditions of Corollary \ref{strichcor2} with $T = t_{j + 1} - t_j$. We obtain
$$\||D|^{\rho}u\|_{L^p([t_j, t_{j + 1}];L^q)} \lesssim \|f\|_{L^1([t_j, t_{j + 1}];L^2)} + \|u\|_{L^\infty([t_j, t_{j + 1}];L^2)}.$$
Using (\ref{inhomogsplit}) we have
$$\||D|^{\rho}u\|_{L^p([t_j, t_{j + 1}];L^q)} \lesssim \lambda^{-\frac{1}{10}}\|f\|_{L^1(I;L^2)} + \|u\|_{L^\infty([t_j, t_{j + 1}];L^2)}.$$
Raising to the power $p$ and summing over $j$, we obtain
$$\||D|^{\rho}u\|_{L^p(I;L^q)}  \lesssim \lambda^{\frac{1-p}{10p}}\|f\|_{L^1(I;L^2)} + \lambda^{\frac{1}{10p}}\|u\|_{L^\infty(I; L^2)}.$$
\end{proof}

\subsection{Proof of Proposition \ref{redstrich}}

In this section we combine the estimates on the symbol $p$ from Section \ref{symbolregsec} and the Strichartz estimates from the previous section to prove Proposition \ref{redstrich}. Recall that if we define $v_\lambda(t, y) = u_\lambda(t, X(t, y)),$ then $v_\lambda$ satisfies (\ref{coveqn}),
$$\D_t v_\lambda(t, y) + i(p(t, y, y', D)v_\lambda)(t, y) = f_\lambda(t, X(t, y)).$$
Up to a perturbative inhomogeneous error, we will apply Proposition \ref{scalestrichprop} with $u = v_\lambda$ and $a = p(t, y, y', \xi)$. In the following, implicit constants will depend on $\FF(s, r, T)$. 

For instance, we may partition $[0, T]$ into a number of time intervals of length $T'$, depending on $\|V\|_{L^2(I;W^{1, \infty})}$ and hence $\FF(s, r, T)$. Without loss of generality, we consider the first such interval $I = [0, T']$.

\subsubsection*{Step 1}
We recall the properties of $p$ required by Proposition \ref{scalestrichprop}. For the curvature lower bound on $p$, we recall Corollary \ref{hessian2}.

We have $\lambda^{3/2}\D_t\D_\xi^2 p \in L^p_t(I;L^\infty_{y, y', \xi})$ by Proposition \ref{timeregularity}. Then on a sufficiently short time interval $I$, we have by H\"older in time,
$$\|\lambda^{3/2}\D_t\D_\xi^2 p\|_{L^1_t(I;L^\infty_{x, \xi})} \leq c_1.$$

Proposition \ref{p5} implies
$$\|\D^\beta_\xi p\|_{L^1_t(I; L^\infty_\xi \dot{W}^{\half, \infty}_{y, y'})} \lesssim  \lambda^{\half - |\beta|}.$$
It is also easy to see that $\lambda^{-\half + |\beta|}\D^\beta_\xi p \in L^\infty_{t, y, y', \xi}$ from the proof of Proposition \ref{p5}, by using only $L_{t, y, y', \xi}^\infty$ in place of $L_t^p(I; L_\xi^\infty C_*^\alpha)$.

\subsubsection*{Step 2}
Before we can apply Proposition \ref{scalestrichprop}, we need to frequency localize 
$$v_\lambda(t, y) = u_\lambda(t, X(t, y)).$$
By the computations for the change of variables in Section \ref{symbolregsec}, but with $\psi_\lambda(\xi)$, the symbol of $S_\lambda$, in place of $\gamma_\delta$, we have
$$v_\lambda(t, y) = u_\lambda(t, X(y)) = (S_\lambda u_\lambda)(t, X(y)) = (\chi(t, y, y', D)v_\lambda)(t, y)$$
where
$$\chi(t, y, y', \xi) = \psi_\lambda(M(y, y')\xi)J(y, y').$$
Further, following the proof of Proposition \ref{p5x2} with $\chi$ in place of $p$, we have
$$\chi \in L^1S_{1, \delta, \delta}^{0, (\half)}$$
and hence, recalling that $0 < \delta < 1$,
\begin{equation}\label{microlocalfreq}
\chi(t, y, y', D) = \chi(t, y', y', \xi) + r(t, y', D), \qquad r \in L^1S_{1, \delta}^{-1 + \delta(1 - \half), (0)} \subseteq L^1S_{1, \delta}^{-\half, (0)}.
\end{equation}
Writing $\tilde{\chi}(t, y', \xi) = \chi(t, y', y', \xi)$, we conclude from (\ref{coveqn}) and (\ref{microlocalfreq}) that
\begin{equation}\label{COVfreqeqn}
(\D_t \tilde{\chi}v_\lambda)(t, y) + i(p\tilde{\chi}v_\lambda)(t, y) = f_\lambda(t, X(t, y)) - (\D_t r v_\lambda)(t, y) - i(p rv_\lambda)(t, y).
\end{equation}

\subsubsection*{Step 3}
From the estimate of Proposition \ref{scalestrichprop}, we see that an estimate
\begin{equation}\label{COVfreqerr}
\|\D_t r v_\lambda - ip rv_\lambda\|_{L^1(I;L^2)} \lesssim \|f_\lambda(t, X(t, y))\|_{L^1(I;L^2)} + \|v_\lambda\|_{L^\infty(I;L^2)}
\end{equation}
would suffice.

First we estimate $p rv_\lambda$. Observe that 
$$p \in L^\infty S_{1, \delta, \delta}^{\half, (0)},$$
by using the uniform bound (\ref{COV2}) in place of (\ref{COV2.3}) (and similarly choosing the uniform bound in Corollary \ref{gammabd}) in the proof of Proposition \ref{p5x2}. Thus we have
$$\|prv_\lambda\|_{L^1(I;L^2)} \lesssim \|rv_\lambda\|_{L^1(I;H^\half)} \lesssim \|v_\lambda\|_{L^\infty(I;L^2)}$$
as desired.

Next, we estimate $\D_t r v_\lambda$. Write
$$\D_t r v_\lambda = \dot{r} v_\lambda + r \D_t v_\lambda = \dot{r} v_\lambda - ir p v_\lambda + r f_\lambda(t, X(t, y)).$$
$r p v_\lambda$ is estimated in the same way as $p r v_\lambda$. To estimate $rf_\lambda$, view $r = \chi - \tilde{\chi}$ and without loss of generality, estimate $\chi f_\lambda$. Similar to the analysis of $p$ with uniform bounds, we have
$$\chi \in L^\infty S_{1, \delta, \delta}^{0, (0)},$$
from which we obtain
$$\|\chi (f_\lambda(t, X(t, y))) \|_{L^1(I;L^2)} \lesssim \|f_\lambda(t, X(t, y))\|_{L^1(I;L^2)}.$$
Similarly, to estimate $\dot{r} v_\lambda$, view $\dot{r} = \dot{\chi} - \dot{\tilde{\chi}}$ and without loss of generality, estimate $\dot{\chi} v_\lambda$. From the definition of $\chi$ and the relation $\dot{X} = V_\delta$, we have (with implicit constant depending on $\|V\|_{L^p(I;W^{1, \infty})}$)
$$\dot{\chi} \in L^1 S_{1, \delta, \delta}^{0, (0)},$$
from which we obtain
$$\|\dot{\chi} v_\lambda\|_{L^1(I;L^2)} \lesssim \|v_\lambda\|_{L^\infty(I;L^2)}$$
as desired.

\subsubsection*{Step 4}
Applying Proposition \ref{scalestrichprop} to (\ref{COVfreqeqn}) and applying the estimate (\ref{COVfreqerr}), we obtain
\begin{align*}
\||D|^{-\rho} \tilde{\chi} v_\lambda\|_{L^p(I;L^q)} &\lesssim \lambda^{-\frac{1}{10p'}}(\|f_\lambda(t, X(t, y))\|_{L^1(I;L^2)} + \|v_\lambda\|_{L^\infty(I;L^2)}) + \lambda^{\frac{1}{10p}}\|\tilde{\chi} v_\lambda\|_{L^\infty(I;L^2)} \\
&\lesssim \lambda^{-\frac{1}{10p'}}\|f_\lambda(t, X(t, y))\|_{L^1(I;L^2)} + \lambda^{\frac{1}{10p}}\|v_\lambda\|_{L^\infty(I;L^2)}.
\end{align*}

We would like to replace $\tilde{\chi} v_\lambda$ on the left hand side by $v_\lambda$. Recall the difference is $rv_\lambda$. We have using Sobolev embedding and $r \in L^1S_{1, \delta}^{-\half, (0)}$,
\begin{align*}
\||D|^{-\rho} r v_\lambda\|_{L^p(I;L^q)} &\lesssim \||D|^{-\rho} r v_\lambda\|_{L^p(I;H^{\frac{2}{p} + \eps})} \lesssim \|r v_\lambda\|_{L^p(I;H^{\frac{1}{2p} + \eps})} \lesssim \|v_\lambda\|_{L^\infty(I;H^{\frac{1}{2p} + \eps - \half})}
\end{align*}
which more than suffices as $p \geq 2$. 

Lastly, by (\ref{COV}), $X$ is bilipschitz uniformly in time, so we may replace $v_\lambda$ with $u_\lambda$ and $f_\lambda(X)$ with $f_\lambda$:
\begin{equation}\label{finaleqn}
\||D|^{-\rho} u_\lambda\|_{L^p(I;L^q)} \lesssim \lambda^{-\frac{1}{10p'}}\|f_\lambda\|_{L^1(I;L^2)} + \lambda^{\frac{1}{10p}}\| u_\lambda\|_{L^\infty(I;L^2)}.
\end{equation}

\subsubsection*{Step 5}

Rearranging (\ref{finaleqn}),
$$\||D|^{-\rho} u_\lambda\|_{L^p(I;L^q)} \lesssim \lambda^{\frac{1}{10p}}(\lambda^{-\frac{1}{10}}\|f_\lambda\|_{L^1(I;L^2)} + \| u_\lambda\|_{L^\infty(I;L^2)}).$$

In the case $d = 1$ and $p = 4$,
$$\||D|^{-\rho} u_\lambda\|_{L^4(I;L^\infty)} \lesssim \lambda^{\frac{1}{40}}(\lambda^{-\frac{1}{10}}\|f_\lambda\|_{L^1(I;L^2)} + \| u_\lambda\|_{L^\infty(I;L^2)}),$$
We conclude by the frequency localization of $u_\lambda, f_\lambda$ that
$$\|u_\lambda \|_{L^4(I;W^{s_1 - \frac{d}{2} + \frac{1}{10}, \infty})} \leq \FF(s, r, T)( \|f_\lambda\|_{L^1(I;H^{s_1 - \frac{1}{10}})} + \|u_\lambda\|_{L^\infty(I;H^{s_1})})$$
as desired.

In the case $d \geq 2$, choose $p = 2 + \eps'$. Then by Bernstein's,
$$\||D|^{-\rho}u_\lambda\|_{L^\infty_x} \lesssim \lambda^{\frac{d}{q}}\||D|^{-\rho}u_\lambda\|_{L^q_x} = \lambda^{\frac{d}{2} - \frac{2}{p}}\||D|^{-\rho}u_\lambda\|_{L^q_x}$$
and hence
$$\|u_\lambda \|_{L^p(I;L^\infty)} \lesssim \lambda^{\rho + \frac{d}{2} - \frac{2}{p}}\lambda^\frac{1}{10p} (\lambda^{\frac{1}{10}}\|f_\lambda\|_{L^1(I;L^2)} + \|u_\lambda\|_{L^\infty(I;L^2)}).$$
Compute that, for small $\eps_i > 0$ depending on $\eps'$,
$$\rho + \frac{d}{2} - \frac{2}{p} = \frac{d}{2} - \frac{1}{4} + \eps_1$$
and
$$\frac{1}{10p} = \frac{1}{20} - \eps_2$$
so that
$$\|u_\lambda \|_{L^p(I;L^\infty)} \lesssim \lambda^{\frac{d}{2} - \frac{1}{4} + \eps_1 + \frac{1}{20} - \eps_2}(\lambda^{\frac{1}{10}}\|f_\lambda\|_{L^1(I;L^2)} + \|u_\lambda\|_{L^\infty(I;L^2)}).$$
We conclude by the frequency localization of $u_\lambda, f_\lambda$ and collecting $\eps_i$ into $\eps$ that
$$\|u_\lambda \|_{L^p(I;W^{s_1 - \frac{d}{2} + \frac{1}{5} - \eps, \infty})} \leq \FF(s, r, T)( \|f_\lambda\|_{L^1(I;H^{s_1 - \frac{1}{10}})} + \|u_\lambda\|_{L^\infty(I;H^{s_1})}).$$
After applying a H\"older estimate in time to replace $p$ on the left hand side by $2$, this concludes the proof of Proposition \ref{redstrich}.

\appendix

\section{Elliptic Estimates}\label{ellipticsection}

The goal of this section is to establish elliptic estimates on the Dirichlet problem with rough boundary, in preparation for later analysis of the Dirichlet to Neumann map with rough boundary. The estimates here are roughly the H\"older counterparts to Sobolev estimates established in \cite{alazard2014cauchy}.

Throughout this section, let
$$s > \frac{d}{2} + \half, \qquad 0 < r - 1 < s - \frac{d}{2} - \half, \qquad r \leq \frac{3}{2}, \qquad I = [-1, 0].$$
We denote $\nabla = \nabla_x$ and $\Delta = \Delta_x$. It is also convenient to define the following spaces for $J \subseteq \R$:
\begin{align*}
X^\sigma(J) &= C_z^0(J;H^{\sigma}) \cap L_z^2(J;H^{\sigma + \half})  \\
Y^\sigma(J) &= L_z^1(J;H^{\sigma}) + L_z^2(J;H^{\sigma - \half}) \\
U^\sigma(J) &= C_z^0(J;C_*^{\sigma}) \cap L_z^2(J;C_*^{\sigma + \half}).  \\
\end{align*}

\subsection{Flattening the Boundary}

Consider the Dirichlet problem with rough boundary. Let
\begin{equation}\label{preelliptic}
\Delta_{x, y} \theta(x, y) = F, \qquad \theta|_{y = \eta(x)} = f.
\end{equation}
%The Dirichlet to Neumann map is given by solving (\ref{preelliptic}) for $\theta$ with $F = 0$ and setting
%$$(G(\eta)f)(x) = \sqrt{1 + |\nabla \eta|^2} (\D_n \theta) |_{y = \eta(x)} = ((\D_y - \nabla \eta \cdot %\nabla)\theta) |_{y = \eta(x)}.$$

We would like estimates on $\theta$ and its derivatives. To address the rough boundary $\eta$ and corresponding rough domain, it will be convenient to state and prove estimates after a change of variables to a problem with flat boundary and domain. Following \cite{lannes2005well}, \cite{alazard2014cauchy}, we choose a Lipschitz diffeomorphism $\rho$ from a subset of the flat domain $\R^d \times (-\infty, 0)$ to the domain $\{(x, y) \in \OO: y < \eta(x)\}$. See \cite{alazard2014cauchy} or \cite{alazard2014strichartz} for the specific form of the diffeomorphism. For our purposes, the estimates recalled below will suffice. Denote
$$\tilde{u}(x, z) = u(x, \rho(x, z)).$$

The function $\tilde{\theta}$ of the flattened domain satisfies
\begin{equation}\label{inhomogelliptic}
(\D_z^2 + \alpha \Delta + \beta \cdot \nabla \D_z - \gamma \D_z)\tilde{\theta} = F_0, \quad v |_{z = 0} = f
\end{equation}
where
$$\alpha = \frac{(\D_z \rho)^2}{1 + |\nabla \rho|^2}, \quad \beta = -2 \frac{\D_z \rho \nabla \rho}{1 + |\nabla \rho|^2}, \quad \gamma = \frac{\D_z^2 \rho + \alpha \Delta \rho + \beta \cdot \nabla \D_z \rho}{\D_z \rho}$$
and
$$F_0(x, z) = \alpha \tilde{F}.$$

We recall the following estimates on the diffeomorphism $\rho$. 

\begin{prop}\label{diffeoest}
The diffeomorphism $\rho$ satisfies
\begin{align}
\label{diffeoinfty}\|(\D_z \rho)^{-1}\|_{C^0(I; C_*^{r - 1})} + \|\nabla_{x,z} \rho\|_{C^0(I; C_*^{r - 1})} &\leq \FF(\|\eta\|_{H^{s + \half}}) \\
\label{diffeosobolev}\|(\D_z \rho - h, \nabla \rho)\|_{X^{s - \half}(I)} + \|\D_z^2 \rho\|_{X^{s - \frac{3}{2}}(I)} &\leq \FF(\|\eta\|_{H^{s + \half}}) \\
\label{diffeoholder}\|\nabla_{x,z} \rho\|_{U^{r - \half}(I)} + \|\D_z^2 \rho\|_{U^{r - \frac{3}{2}}(I)}&\leq \FF(\|\eta\|_{H^{s + \half}})(1 + \|\eta\|_{W^{r + \half, \infty}}) \\
\label{diffeoL1sobolev}\|\nabla_{x,z} \rho\|_{L^1(I; H^{s + \half})} + \|\D_z^2 \rho\|_{L^1(I; H^{s - \half})}&\leq \FF(\|\eta\|_{H^{s + \half}}) \\
\label{diffeoL1}\|\nabla_{x,z} \rho\|_{L^1(I; C_*^{r + \half})} + \|\D_z^2 \rho\|_{L^1(I; C_*^{r - \half})}&\leq \FF(\|\eta\|_{H^{s + \half}})(1 + \|\eta\|_{W^{r + \half, \infty}}).
\end{align}
\end{prop}
\begin{proof}
The first estimate is a consequence of \cite[Lemmas 3.6, 3.7]{alazard2014cauchy} combined with Sobolev embedding. The first terms of the second and third estimates are from \cite[Lemma 3.7]{alazard2014cauchy} and \cite[Lemma B.1]{alazard2014strichartz} respectively. The corresponding estimates on $\D^2_z \rho$ are proven similarly from the definition of $\rho$. Lastly, the proofs of the fourth and fifth estimates are similar to those of the second and third respectively, by using $L^1_z$ in place of $L^2_z$.
\end{proof}

Then the following functions of $\rho$ satisfy similar estimates, with proofs straightforward from paraproduct rules:
\begin{prop}\label{prelimest}
For $\alpha, \beta, \gamma$ defined as above,
\begin{align}
\label{sobolevabg}\|(\alpha - h^2, \beta)\|_{X^{s - \half}(I)} + \|\gamma\|_{X^{s - \frac{3}{2}}(I)} &\leq \FF(\|\eta\|_{H^{s + \half}}) \\
\label{holderabg}\|(\alpha, \beta)\|_{U^{r - \half}(I)} + \|\gamma\|_{L^2(I; C_*^{r - 1})} &\leq \FF(\|\eta\|_{H^{s + \half}})(1 + \|\eta\|_{W^{r + \half, \infty}}) \\
\label{L1abg} \|(\alpha, \beta)\|_{L^1(I; C_*^{r + \half})}  + \|\gamma\|_{L^1(I; C_*^{r - \frac{1}{2}})}  &\leq \FF(\|\eta\|_{H^{s + \half}})(1 + \|\eta\|_{W^{r + \half, \infty}}).
\end{align}
\end{prop}

\begin{comment}
\begin{proof}
The first four estimates follow from a more careful analysis of the proof of \cite[Lemma B.9]{alazard2014strichartz} to obtain the $r - 1$ gain, and use of the $L^2$ or $L^\infty$ estimates in \cite[Lemma B.1]{alazard2014strichartz} for the missing $L^2$ or $L^\infty$ counterparts. 
\end{proof}

\begin{lem}\label{diffeoinfty}
Let $s > \frac{d}{2} + \half$. Then
$$\|\nabla_{x,z} \rho\|_{L^\infty([-1, 0]; L^\infty)} \lesssim \|\eta\|_{H^{s + \half}}.$$
\end{lem}

\end{comment}

\subsection{Factoring the Elliptic Equation}

To establish elliptic estimates, we factor the equation (\ref{inhomogelliptic}) as the product of forward and backward paralinearized parabolic evolutions,
\begin{equation}\label{factoreqn}
(\D_z - T_a)(\D_z - T_A)\tilde{\theta} \approx F_0
\end{equation}
where
$$a = \half (-i\beta \cdot \xi - \sqrt{4\alpha |\xi|^2 - (\beta \cdot \xi)^2}), \quad A = \half (-i\beta \cdot \xi + \sqrt{4\alpha |\xi|^2 - (\beta \cdot \xi)^2}).$$
A pair of parabolic estimates then yields the desired elliptic estimates.

First, we record estimates on the symbols $a$ and $A$ in (\ref{factoreqn}). Define
$$\MM_\rho^m(a) = \sup_{z \in I} M_\rho^m(a(z)), \qquad \MM_\rho^{m, 2}(a) = \| M_\rho^m(a(z))\|_{L^2_z(I)}.$$

\begin{prop}\label{aabds}
For $a, A$ defined as above,
\begin{align*}
%\label{L1abg}\|\gamma\|_{U^{r - \frac{3}{2}}([-1, 0])} &\leq \FF(\|\eta\|_{H^{s + \half}})(1 + \|\eta\|_{W^{r + \half, \infty}}) \\
%\|\nabla \rho\|_{L^\infty([-1, 0]; W^{r - \half, \infty})} &\leq \FF(\|\eta\|_{H^{s + \half}}) \|\eta\|_{W^{r + \half, \infty}} \\
%\left\|\frac{1 + |\nabla \rho|^2}{\D_z \rho}\right\|_{L^\infty([-1, 0]; W^{r - \half, \infty})} &\leq \FF(\|\eta\|_{H^{s + \half}}) \|\eta\|_{W^{r + \half, \infty}} \\
\MM_{0}^1(a) + \MM_{0}^1(A) &\leq \FF(\|\eta\|_{H^{s + \half}}) \\
\MM_{r - \half}^1(a) + \MM_{r - \half}^1(A) + \MM^1_{r-\frac{3}{2}}(\D_z A) &\leq \FF(\|\eta\|_{H^{s + \half}})(1 + \|\eta\|_{W^{r + \half, \infty}}) \\
\MM_{r}^{1, 2}(a) + \MM_{r}^{1,2}(A) + \MM^{1,2}_{r-1}(\D_z A) &\leq \FF(\|\eta\|_{H^{s + \half}})(1 + \|\eta\|_{W^{r + \half, \infty}}) 
\end{align*}
\end{prop}

\begin{proof}
The first estimate is from \cite[Lemma 3.22]{alazard2014cauchy}. The second estimate follows from a more careful analysis of the proof of \cite[(B.45)]{alazard2014strichartz} to obtain the $r - 1$ gain. The third estimate is similar, but is proven using the $L^2_z$ part of (\ref{diffeoholder}), instead of $C_z^0$.
\end{proof}

Second, we estimate the errors on the right hand side of (\ref{factoreqn}) that arise from the parabolic factorization and paralinearization. 

\begin{prop} \label{inhomogbd}

Let $\half \leq \sigma < r$ and $z_0 \in [-1, 0]$, $J = [z_0, 0]$. Consider $\tilde{\theta}$ solving (\ref{inhomogelliptic}). Then we can write
$$(\D_z - T_a)(\D_z - T_A)\tilde{\theta} = F_0 + F_1 + F_2 + F_3$$
where for $i \geq 1$,
\begin{align*}
\|F_i\|_{L^1(J;C_*^{\sigma-\frac{1}{2}})} &\leq \FF(\|\eta\|_{H^{s + \half}})(1 + \|\eta\|_{W^{r + \half, \infty}}) \|\nabla_{x, z} \tilde{\theta}\|_{U^{\sigma - 1}(J)}.
%\|F_2\|_{L^1(J;C_*^{\sigma-\frac{3}{2} + r})} + \|F_3\|_{L^1(J;C_*^{\sigma-\frac{3}{2} + r})} &\leq \FF(\|\eta\|_{H^{s + \half}})(1 + \|\eta\|_{W^{r + \half, \infty}}) \|\nabla_{x, z} \tilde{\theta}\|_{L^2(J; C_*^{\sigma - \half})}.
\end{align*}
\end{prop}
\begin{proof} 

Here we have used $F_1, F_2, F_3$ to represent the errors arising from, respectively, the first order term on the left hand side of (\ref{inhomogelliptic}), the paralinearization errors, and the lower order terms from applying the symbolic calculus. We estimate these one by one.

We begin with $F_3$. Factor
\begin{align*}
(\D_z - T_a)(\D_z - T_A)\tilde{\theta} &= \D_z^2 \tilde{\theta} + T_a T_A \tilde{\theta} - T_a \D_z \tilde{\theta} - \D_z T_A \tilde{\theta} \\
&= \D_z^2 \tilde{\theta} + T_{aA} \tilde{\theta} - (T_a + T_A) \D_z \tilde{\theta} + (T_aT_A - T_{aA})\tilde{\theta} - (\D_z T_A - T_A \D_z)\tilde{\theta} \\
&= \D_z^2 \tilde{\theta} + T_\alpha \Delta \tilde{\theta} + T_\beta \cdot \nabla \D_z \tilde{\theta} + (T_aT_A - T_{aA})\tilde{\theta} - T_{\D_z A}\tilde{\theta}.
\end{align*}
For the first error term, by (\ref{holdercommutator}) and Proposition \ref{aabds},
\begin{align*}
\|(T_aT_A - T_{aA})\tilde{\theta}\|_{L^1(J;C_*^{\sigma-\frac{1}{2}})} &\lesssim (\MM_{r}^{1,2}(a)\MM_0^1(A) + \MM_0^1(a)\MM_{r}^{1,2}(A)) \|\tilde{\theta}\|_{L^2(J;C_*^{\sigma - \frac{1}{2} + 2 - r})} \\
&\leq \FF(\|\eta\|_{H^{s + \half}})(1 + \|\eta\|_{W^{r + \half, \infty}}) \|\tilde{\theta}\|_{L^2(J;C_*^{\sigma + \frac{1}{2}})}.
\end{align*}
Note that by the definition of the inhomogeneous paradifferential operator, we may exchange $v$ for $S_{>1/10} \tilde{\theta}$ in the previous inequalities, and hence bound
$$\|S_{>1/10} \tilde{\theta}\|_{L^2(J;C_*^{\sigma + \frac{1}{2}})} \lesssim \|\nabla \tilde{\theta}\|_{L^2(J;C_*^{\sigma - \half})}.$$
Similarly, by (\ref{ordernorm}),
\begin{align*}
\|T_{\D_z A}\tilde{\theta}\|_{L^1(J;C_*^{\sigma - \frac{1}{2}})} &\lesssim \MM_{0}^{1,2}(\D_z A) \|S_{> 1/10} \tilde{\theta}\|_{L^2(J; C_*^{\sigma - \half})} \\ 
&\leq \FF(\|\eta\|_{H^{s + \half}}) (1 + \|\eta\|_{W^{r + \half, \infty}}) \|\nabla \tilde{\theta}\|_{L^2(J; C_*^{\sigma - \half})}.
\end{align*}
We hence have 
$$(\D_z - T_a)(\D_z - T_A)\tilde{\theta} = \D_z^2 \tilde{\theta} + T_\alpha \Delta \tilde{\theta} + T_\beta \cdot \nabla \D_z \tilde{\theta} + F_3$$
where 
$$\|F_3\|_{L^1(J;C_*^{\sigma-\frac{1}{2}})} \leq \FF(\|\eta\|_{H^{s + \half}}) (1 + \|\eta\|_{W^{r + \half, \infty}}) \|\nabla \tilde{\theta}\|_{L^2(J; C_*^{\sigma - \half})}.$$

Next we estimate the error $F_2$ consisting of errors from paralinearization. Write 
$$(\D_z - T_a)(\D_z - T_A)\tilde{\theta} = \D_z^2 \tilde{\theta} + \alpha \Delta \tilde{\theta} + \beta \cdot \nabla \D_z \tilde{\theta} + (T_\alpha - \alpha) \Delta \tilde{\theta} + (T_\beta - \beta)\cdot \nabla \D_z \tilde{\theta}$$
and expand
$$(T_\alpha - \alpha) \Delta \tilde{\theta} + (T_\beta - \beta)\cdot \nabla \D_z \tilde{\theta} = -(T_{\Delta \tilde{\theta}} \alpha + R(\alpha, \Delta \tilde{\theta}) + T_{\nabla \D_z \tilde{\theta}} \cdot \beta + R(\beta, \nabla \D_z \tilde{\theta})).$$
By (\ref{holderparaproduct}) and (\ref{holderparaerror}),
\begin{align*}
\|T_{\Delta \tilde{\theta}} \alpha\|_{L^1(J;C_*^{\sigma -\frac{1}{2}})} + \| R(\alpha, \Delta \tilde{\theta})\|_{L^1(J;C_*^{\sigma -\frac{1}{2}})} &\lesssim \|\Delta \tilde{\theta}\|_{L^2(J;C^{\sigma - \frac{3}{2}}_*)} \|\alpha\|_{L^2(J;C_*^r)} \\
\|T_{\nabla \D_z \tilde{\theta}} \cdot \beta\|_{L^1(J;C_*^{\sigma -\frac{1}{2}})} + \|R(\beta, \nabla \D_z \tilde{\theta})\|_{L^1(J;C_*^{\sigma - \frac{1}{2}})} &\lesssim \|\nabla \D_z \tilde{\theta}\|_{L^2(J;C^{\sigma - \frac{3}{2}}_*)} \|\beta\|_{L^\infty(J;C_*^r)}.
\end{align*}
We estimate $\alpha$ and $\beta$ via (\ref{holderabg}) while
$$\|\Delta \tilde{\theta}\|_{C_*^{\sigma - \frac{3}{2}}} \lesssim \|\nabla \tilde{\theta} \|_{C_*^{\sigma - \half}}, \quad \|\nabla \D_z \tilde{\theta}\|_{C^{\sigma - \frac{3}{2}}_*} \lesssim \|\D_z \tilde{\theta}\|_{C^{\sigma - \half}_*}.$$
We hence have 
$$(\D_z - T_a)(\D_z - T_A)\tilde{\theta} = \D_z^2 \tilde{\theta} + \alpha \Delta \tilde{\theta} + \beta \cdot \nabla \D_z \tilde{\theta} + F_2 + F_3$$
where 
$$\|F_2\|_{L^1(J;C_*^{\sigma -\frac{1}{2}})} \leq \FF(\|\eta\|_{H^{s + \half}})(1 +  \|\eta\|_{W^{r + \half, \infty}})\|\nabla_{x, z} \tilde{\theta}\|_{L^2(J; C_*^{\sigma - \half})}.$$

Lastly, we estimate the first order term appearing on the left hand side of (\ref{inhomogelliptic}):
$$(\D_z - T_a)(\D_z - T_A)\tilde{\theta} = F_0 + \gamma \D_z \tilde{\theta} + F_2 + F_3.$$
To estimate $F_1 := \gamma \D_z \tilde{\theta}$, we decompose into paraproducts,
$$\gamma \D_z \tilde{\theta} = T_{\gamma} \D_z \tilde{\theta} + T_{\D_z \tilde{\theta}} \gamma + R(\gamma, \D_z \tilde{\theta}).$$
We have by (\ref{holderparaproduct0}), (\ref{holderparaerror}), and (\ref{holderabg}),
\begin{align*}
\|T_{\gamma} \D_z \tilde{\theta} \|_{L^1(J; C_*^{\sigma -\frac{1}{2}})} + \|R(\gamma, \D_z \tilde{\theta})\|_{L^1(J; C_*^{\sigma - \half})} &\lesssim \|\gamma\|_{L^2(J; C_*^{r - 1})}\|\D_z \tilde{\theta}\|_{L^2(J; C_*^{\sigma - \half})}  \\
&\leq \FF(\|\eta\|_{H^{s + \half}})(1 +  \|\eta\|_{W^{r + \half, \infty}})\|\D_z \tilde{\theta}\|_{L^2(J; C_*^{\sigma - \half})}.
\end{align*}
Similarly, using (\ref{L1abg})
\begin{align*}
\|T_{\D_z \tilde{\theta}} \gamma \|_{L^1(J; C_*^{\sigma -\frac{1}{2}})} &\lesssim \|\gamma\|_{L^1(J; C_*^{r -\half})}\|\D_z \tilde{\theta}\|_{L^\infty(J; C_*^{\sigma - r})}  \\
&\leq \FF(\|\eta\|_{H^{s + \half}})(1 +  \|\eta\|_{W^{r + \half, \infty}})\|\D_z \tilde{\theta}\|_{L^\infty(J; C_*^{\sigma - 1})}.
\end{align*}

\end{proof}

\subsection{Parabolic to Elliptic Estimates}

Now that we have estimates on each of the coefficients in the factored equation (\ref{factoreqn}), we can apply parabolic estimates, which we now recall:

\begin{prop}[{\cite[Proposition B.4]{alazard2014strichartz}}]\label{parabolic}
Let $\rho \in (0, 1), \ J = [z_0, z_1] \subseteq \R$, and $p \in \Gamma_\rho^1(J \times \R^d)$ with
$$\text{Re } p(z; x, \xi) \geq c |\xi|.$$
Consider a solution $w$ to
$$(\D_z + T_p) w = F_1 + F_2, \quad w|_{z = z_0} = w_0.$$
Then for any $q \in [1, \infty]$ and $(r_0, r) \in \R^2$ with $r_0 < r$, and $\delta > 0$,
$$\|w\|_{C^0(J;C^r_*)} \lesssim \|w_0\|_{C_*^r} + \|F_1\|_{L^1(J;C_*^{r})} + \|F_2\|_{L^q(J;C_*^{r - 1 + \frac{1}{q} + \delta})} + \|w\|_{L^\infty(J; C_*^{r_0})}$$
with a constant depending on $r_0, r, \rho, c, \delta, q$, and $\MM_\rho^1(p)$. 
\end{prop}

By a simple modification of the proof of this result, we also have
\begin{prop}\label{parabolic2}
Let $\rho \in (0, 1), \ J = [z_0, z_1] \subseteq \R$, and $p \in \Gamma_\rho^1(J \times \R^d)$ with
$$\text{Re } p(z; x, \xi) \geq c |\xi|.$$
Consider a solution $w$ to
$$(\D_z + T_p) w = F, \quad w|_{z = z_0} = w_0.$$
Then for any $q \in [1, \infty]$ and $(r_0, r) \in \R^2$ with $r_0 < r$, and $\delta > 0$,
$$\|w\|_{L^1(J;C^{r + 1}_*) \cap L^2(J;C^{r + \half}_*)} \lesssim \|w_0\|_{C_*^{r + \delta}} + \|F\|_{L^1(J;C_*^{r + \delta})} + \|w\|_{L^2(J; C_*^{r_0})}$$
with a constant depending on $r_0, r, \rho, c, \delta, q$, and $\MM_\rho^1(p)$. 
\end{prop}

We now apply these parabolic estimates twice to (\ref{factoreqn}) to obtain an ``inductive'' elliptic estimate:
\begin{prop}\label{induction}
Let $\half \leq \sigma < r$, $\delta > 0$, and $-1 < z_1 < z_0 < 0$. Denote $J_0 = [z_0, 0], J_1 = [z_1, 0]$. Consider $\tilde{\theta}$ solving (\ref{inhomogelliptic}). Then
\begin{align*}
\|\nabla_{x,z} \tilde{\theta}\|_{C^0(J_0;C^{\sigma - \half}_*)} \leq \ &\FF(\|\eta\|_{H^{s + \half}})(1 + \|\eta\|_{W^{r + \half, \infty}})\|\nabla_{x, z} \tilde{\theta}\|_{U^{\sigma - 1}(J_1)}\\
& + \|f\|_{C_*^{\sigma + \half}} + \|F_0\|_{L^1(J_1;C_*^{\sigma - \frac{1}{2}})} \\
\|\nabla_{x,z} \tilde{\theta}\|_{L^1(J_0;C^{\sigma + \half}_*) \cap L^2(J_0;C^{\sigma}_*)} \leq \ &\FF(\|\eta\|_{H^{s + \half}})(1 + \|\eta\|_{W^{r + \half, \infty}})\|\nabla_{x, z} \tilde{\theta}\|_{U^{\sigma - 1 + \delta}(J_1)}\\
& + \|f\|_{C_*^{\sigma + \half + \frac{\delta}{2}}} + \|F_0\|_{L^1(J_1;C_*^{\sigma - \frac{1}{2} + \delta})}.
\end{align*}
\end{prop}

\begin{proof}
First, we would like to apply the parabolic estimate with symbol $-a$ (satisfying $\text{Re}(-a) \gtrsim |\xi|$) on the equation
$$(\D_z - T_a)w = (\D_z - T_a)(\D_z - T_A)\tilde{\theta} = F_0 + F_1 + F_2 + F_3.$$
However, it is convenient to apply the parabolic estimate on $w$ with vanishing initial condition, so instead we set
$$w := \chi(z)(\D_z - T_A)\tilde{\theta}$$
where $\chi$ is a smooth cutoff vanishing on $[-1, z_1]$ and $\chi = 1$ on $J_0 = [z_0, 0] \subseteq (-1, 0]$. We then have
$$(\D_z - T_a)w = \chi(z)(F_0 + F_1 + F_2 + F_3) + \chi'(z)(\D_z - T_A)\tilde{\theta} =: F'.$$
We estimate $\chi'(z)(\D_z - T_A)\tilde{\theta}$ directly. By (\ref{ordernorm}),
$$\|T_A \tilde{\theta}\|_{L^1(J_1; C_*^{\sigma - \frac{1}{2}})} \lesssim \MM_0^1(A) \|\tilde{\theta}\|_{L^1(J_1;C_*^{\sigma + \half})} \leq \FF(\|\eta\|_{H^{s + \half}}) \|\tilde{\theta}\|_{L^2(J_1;C_*^{\sigma + \half})}.$$
As in estimate of $F_3$, we may replace $v$ by $S_{> 1/10}\tilde{\theta}$ by the inhomogeneous paradifferential calculus and hence estimate 
$$\|S_{> 1/10}\tilde{\theta}\|_{L^2(J_1;C_*^{\sigma + \half})} \lesssim \|\nabla \tilde{\theta}\|_{L^2(J_1;C_*^{\sigma - \frac{1}{2}})}.$$
We conclude
\begin{align*}
\|\chi'(z)(\D_z - T_A)\tilde{\theta}\|_{L^1(J_1; C_*^{\sigma - \frac{1}{2}})} &\leq  \FF(\|\eta\|_{H^{s + \half}}) \|\nabla_{x, z} \tilde{\theta}\|_{L^2(J_1;C_*^{\sigma - \frac{1}{2}})}.
\end{align*}
Combining this estimate with the estimates of Proposition \ref{inhomogbd}, we have
\begin{align*}
\|F'\|_{L^1(J_1;C_*^{\sigma - \half})} &\lesssim \|F_1 + F_2 + F_3 + \chi'(z)(\D_z - T_A)\tilde{\theta}\|_{L^1(J_1;C_*^{\sigma - \half})} + \|F_0\|_{L^1(J_1;C_*^{\sigma - \half})} \\
&\leq \FF(\|\eta\|_{H^{s + \half}})(1 +  \|\eta\|_{W^{r + \half, \infty}})\|\nabla_{x, z} \tilde{\theta}\|_{U^{\sigma - 1}(J_1)} + \|F_0\|_{L^1(J_1;C_*^{\sigma - \half})}.
\end{align*}

We can now apply Proposition \ref{parabolic} with 
$$\rho = \half, \quad J = J_1, \quad p = -a, \quad q = \infty, \quad r = \sigma - \half.$$
Note we may also choose $r_0 = \sigma - \frac{3}{2}$, using the above analysis on $\chi'(z)(\D_z - T_A)\tilde{\theta}$ with $w = \chi(z)(\D_z - T_A)\tilde{\theta}$, yielding as well the same estimates. Thus 
\begin{align}
\|w\|_{C^0(J_1;C^{\sigma - \half}_*)} &\lesssim \|F'\|_{L^1(J_1;C_*^{\sigma - \half})} + \|w\|_{L^\infty(J_1; C_*^{\sigma - \frac{3}{2}})} \nonumber \\
\label{westimate} &\leq \FF(\|\eta\|_{H^{s + \half}})(1 +  \|\eta\|_{W^{r + \half, \infty}})\|\nabla_{x, z} \tilde{\theta}\|_{U^{\sigma - 1}(J_1)} + \|F_0\|_{L^1(J_1;C_*^{\sigma - \half})}.
\end{align}

Similarly, for the second estimate of our proposition, we use Proposition \ref{parabolic2} with 
$$\rho = \half, \quad J = J_1, \quad p = -a, \quad r = \sigma - \half + \frac{\delta}{2}$$
and $\delta/2$ in the place of $\delta$. We again choose $r_0 = \sigma - \frac{3}{2}$. Using $\sigma + \delta$ in place of $\sigma$ in the above estimate for $F'$ (we may assume $\delta$ is small enough so that $\sigma + \delta < r$)
\begin{align*}
\|w\|_{L^1(J_1;C^{\sigma + \half + \frac{\delta}{2}}_*)} &\lesssim \|F'\|_{L^1(J_1;C_*^{\sigma - \half + \delta})} + \|w\|_{L^2(J_1; C_*^{\sigma - \frac{3}{2}})} \nonumber \\
&\leq \FF(\|\eta\|_{H^{s + \half}})(1 +  \|\eta\|_{W^{r + \half, \infty}})\|\nabla_{x, z} \tilde{\theta}\|_{U^{\sigma - 1 + \delta}(J_1)} + \|F_0\|_{L^1(J_1; C_*^{\sigma - \frac{1}{2} + \delta})}.
\end{align*}

Next we apply the second parabolic estimate. On $J_0 = [z_0, 0]$, we have $\chi = 1$ and hence
$$(\D_z - T_A)\tilde{\theta} = w.$$
Define $\tilde{\theta}^*(x, z) = \tilde{\theta}(x, -z)$ and $w^*$, etc. in the analogous way, so that $\text{Re}(A) \gtrsim |\xi|$ and
$$(\D_z + T_A)\tilde{\theta}^* = -w^*, \qquad z \in [0, -z_0] = J_0^*.$$

We again apply Proposition \ref{parabolic} with
$$\rho = \half, \quad J = J_0^*, \quad p = A, \quad q = \infty, \quad r = \sigma + \half, \quad r_0 = 0.$$
We obtain
$$\|\tilde{\theta}^*\|_{C^0(J_0^*;C^{\sigma + \half}_*)} \lesssim \|f\|_{C_*^{\sigma + \half}} + \|w^*\|_{L^\infty(J_0^*; C_*^{\sigma - \half})} + \|\tilde{\theta}^*\|_{L^\infty(J_0^*;L^\infty)}.$$

For the second estimate of our proposition, we again apply Proposition \ref{parabolic2} with
$$\rho = \half, \quad J = J_0^*, \quad p = A, \quad r = \sigma + \half, \quad r_0 = 0$$
and $\delta/2$ in the place of $\delta$ to obtain
$$\|\tilde{\theta}^*\|_{L^1(J_0^*;C^{\sigma + \frac{3}{2}}_*) \cap L^2(J_0^*;C^{\sigma + 1}_*)} \lesssim \|f\|_{C_*^{\sigma + \half + \frac{\delta}{2}}} + \|w^*\|_{L^1(J_0^*; C_*^{\sigma + \half + \frac{\delta}{2}})} + \|\tilde{\theta}^*\|_{L^2(J_0^*;L^\infty)}.$$
Now the rest of the proof for the $L^1(J_0;C^{\sigma + \frac{3}{2}}_*) \cap L^2(J_0;C^{\sigma + 1}_*)$ estimate mirrors that of the $C^0(J_0;C^{\sigma + \half}_*)$ estimate, detailed in the following.

We estimate the last term on the right hand side by writing
$$\tilde{\theta}^*(z) = \tilde{\theta}^*(0) + \int_0^z \D_z \tilde{\theta}^* = f + \int_0^z \D_z \tilde{\theta}^*$$
and hence
$$\|\tilde{\theta}^*\|_{L^\infty(J_0^*;L^\infty)} \lesssim \|f\|_{L^\infty} + \|\D_z \tilde{\theta}^*\|_{L^2(J_0^*;L^\infty)} \lesssim \|f\|_{C_*^{\sigma + \half}} + \|\nabla_{x, z} \tilde{\theta}\|_{U^{\sigma - 1}(J_0)}.$$
Collecting the above estimates, we conclude
\begin{align*}
\|\nabla \tilde{\theta}\|_{C^0(J_0;C^{\sigma - \half}_*)} \lesssim \|\tilde{\theta}^*\|_{C^0(J_0;C^{\sigma + \half}_*)} \lesssim \ &\FF(\|\eta\|_{H^{s + \half}})(1 +  \|\eta\|_{W^{r + \half, \infty}})\|\nabla_{x, z} \tilde{\theta}\|_{U^{\sigma - 1}(J_1)} \\
&+ \|f\|_{C_*^{\sigma + \half}} + \|F_0\|_{L^1(J_1;C_*^{\sigma - \half})}.
\end{align*}
To attain the same estimate on $\D_z \tilde{\theta}$, write
$$\D_z \tilde{\theta} = T_A \tilde{\theta} + w.$$
$T_A \tilde{\theta}$ enjoys the same estimate as $\nabla \tilde{\theta}$ by using Proposition \ref{aabds} and that $A$ is of order 1, and $w$ already has the desired estimate above.
\end{proof}

We recall the following ``base case'' estimate in a special case of (\ref{preelliptic}):
\begin{prop}\cite[Remark 3.15]{alazard2014cauchy}\label{basecase}
Consider $\theta$ solving (\ref{specialelliptic}). Then
$$\|\nabla_{x, z}\tilde{\theta}\|_{X^{-\half}([-1, 0])} \leq \FF(\|\eta\|_{H^{s + \half}})\|f\|_{H^\half}.$$
\end{prop}

We also recall the Sobolev counterpart of the H\"older inductive estimate, which will help us span the gap down to the base case:
\begin{prop}\cite[Proposition 3.16]{alazard2014cauchy}\label{sobolevinductive}
Let $-\half \leq \sigma \leq s - \half$ and $-1 < z_1 < z_0 < 0$. Denote $J_0 = [z_0, 0], J_1 = [z_1, 0]$. Consider $\tilde{\theta}$ solving (\ref{inhomogelliptic}). Then
$$\|\nabla_{x,z} \tilde{\theta}\|_{X^\sigma(J_0)} \leq \FF(\|\eta\|_{H^{s + \half}})(\|f\|_{H^{\sigma + 1}} + \|F_0\|_{Y^\sigma(J_1)} + \|\nabla_{x,z} \tilde{\theta}\|_{X^{-\half}(J_1)}).$$
\end{prop}

\section{H\"older Estimates}\label{holdersec}

Throughout this section, let
$$s > \frac{d}{2} + \half, \qquad 0 < r - 1 < s - \frac{d}{2} - \half, \qquad r \leq \frac{3}{2}, \qquad I = [-1, 0].$$
Also recall that we denote $\nabla = \nabla_x$ and $\Delta = \Delta_x$, and have defined the following spaces for $J \subseteq \R$:
\begin{align*}
X^\sigma(J) &= C_z^0(I;H^{\sigma}) \cap L_z^2(I;H^{\sigma + \half})  \\
Y^\sigma(J) &= L_z^1(I;H^{\sigma}) + L_z^2(I;H^{\sigma - \half}) \\
U^\sigma(J) &= C_z^0(I;C_*^{\sigma}) \cap L_z^2(I;C_*^{\sigma + \half}).  \\
\end{align*}

\subsection{Paralinearization of the Dirichlet to Neumann Map}\label{holderonDN}

In this section we establish H\"older estimates on the paralinearization error of the Dirichlet to Neumann map. Tame Sobolev estimates (linear in the highest order H\"older norms) were established in \cite{alazard2014strichartz}. The H\"older estimates will instead be quadratic in the H\"older norms, but still be useful, for instance toward studying the structure of the vector field $V$ and hence the regularity of the associated flow.

Recall the Dirichlet to Neumann map is given by solving
\begin{equation}\label{specialelliptic}
\Delta_{x, y} \theta(t, x, y) = 0, \qquad \theta|_{y = \eta(x)} = f, \qquad \D_n \theta |_\Gamma = 0
\end{equation}
and setting
$$(G(\eta)f)(x) = \sqrt{1 + |\nabla \eta|^2} (\D_n \theta) |_{y = \eta(x)} = ((\D_y - \nabla \eta \cdot \nabla)\theta) |_{y = \eta(x)}.$$
In the flattened coordinates discussed in Appendix \ref{ellipticsection}, we may write the Dirichlet to Neumann map as (recall that we write $\tilde{u}(x, z) = u(x, \rho(x, z))$ where $\rho$ is the diffeomorphism that flattens the boundary defined by the graph of $\eta$)
$$G(\eta)f = \left. \left(\frac{1 + |\nabla \rho|^2}{\D_z \rho} \D_z \tilde{\theta} - \nabla \rho \cdot \nabla \tilde{\theta}\right)\right|_{z = 0}.$$

Having the elliptic estimates of Appendix \ref{ellipticsection}, the paralinearization of $G(\eta)f$ is straightforward except for the $\D_z \tilde{\theta}$ term. To paralinearize this term, we recall the flattened, factored form (\ref{factoreqn}) of (\ref{specialelliptic}) as the product of paralinearized parabolic evolutions,
$$(\D_z - T_a)(\D_z - T_A)\tilde{\theta} \approx 0.$$
Then by a single parabolic estimate, we have the paralinearization $\D_z v \approx T_A  v$. More precisely, we have from (\ref{westimate}) in the proof of Proposition \ref{induction} with $\sigma = 1$,

\begin{lem}\label{wlem}
Let $-1 < z_1 < z_0 < 0$. Denote $J_0 = [z_0, 0], J_1 = [z_1, 0]$. Consider $\theta$ solving (\ref{specialelliptic}). Then
\begin{align*}
\|(\D_z - T_A)\tilde{\theta}\|_{C^0(J_0;C^{\half}_*)} \leq \FF(\|\eta\|_{H^{s + \half}})(1 +  \|\eta\|_{W^{r + \half, \infty}})\|\nabla_{x, z} \tilde{\theta}\|_{U^{0}(J_1)}.
\end{align*}
\end{lem}

We can remove the instance of $\tilde{\theta}$ on the right hand side of the lemma, working our way down to the ``base case'':
\begin{cor}\label{higherw2}
Let $z_0 \in (-1, 0], J_0 = [z_0, 0]$. Consider $\theta$ solving (\ref{specialelliptic}). Then
\begin{align*}
\|(\D_z - T_A)\tilde{\theta}\|_{C^0(J_0;C^{\half}_*)} &\leq \FF(\|\eta\|_{H^{s + \half}}, \|f\|_{H^s})(1 +  \|\eta\|_{W^{r + \half, \infty}})(1 + \|\eta\|_{W^{r + \half, \infty}} + \|f\|_{C_*^{r}}) \\
\|\nabla_{x, z} \tilde{\theta}\|_{U^{0}(J_0)} &\leq \FF(\|\eta\|_{H^{s + \half}}, \|f\|_{H^s})(1 + \|\eta\|_{W^{r + \half, \infty}} + \|f\|_{C_*^{r}}).
\end{align*}
\end{cor}

\begin{proof}
To estimate the right hand side of Lemma \ref{wlem}, recall Proposition \ref{induction} with $\sigma = \half$, $\delta$ chosen small enough so that $\delta < r - 1$, and $F = 0$:
\begin{align*}
\|\nabla_{x, z} \tilde{\theta}\|_{U^{0}(J_1)} \leq \ &\FF(\|\eta\|_{H^{s + \half}})(1 + \|\eta\|_{W^{r + \half, \infty}})\|\nabla_{x, z} \tilde{\theta}\|_{U^{r - \frac{3}{2}}(J_2)} + \|f\|_{C_*^{r}}.
\end{align*}
In turn, to estimate the right hand side, we apply Sobolev embedding and Proposition \ref{sobolevinductive} with $\sigma = s - 1$ and $F = 0$:
$$\|\nabla_{x, z} \tilde{\theta}\|_{U^{r - \frac{3}{2}}(J_2)} \lesssim \|\nabla_{x,z} \tilde{\theta}\|_{X^{s - 1}(J_3)} \leq \FF(\|\eta\|_{H^{s + \half}})(\|f\|_{H^{s}} + \|\nabla_{x,z} \tilde{\theta}\|_{X^{-\half}(J_3)}).$$
Lastly, recalling Proposition \ref{basecase} yields the desired estimate.

\end{proof}

Recall that we write $\Lambda$ for the principal symbol of the Dirichlet to Neumann map. We now perform the paralinearization:
\begin{prop}\label{DNparalinear}
Write
$$\Lambda(t, x, \xi) = \sqrt{(1 + |\nabla \eta|^2)|\xi|^2 - (\nabla \eta \cdot \xi)^2}.$$
Then
$$\|G(\eta)f - T_\Lambda f\|_{W^{\half, \infty}} \leq \FF(\|\eta\|_{H^{s + \half}}, \|f\|_{H^s})(1 + \|\eta\|_{W^{r + \half, \infty}})(1 + \|\eta\|_{W^{r + \half, \infty}}+ \|f\|_{W^{r, \infty}}).$$
\end{prop}
\begin{proof}

Recall
$$G(\eta)f = \left. \left(\frac{1 + |\nabla \rho|^2}{\D_z \rho} \D_z \tilde{\theta} - \nabla \rho \cdot \nabla \tilde{\theta}\right)\right|_{z = 0} =: \left. \left(\zeta \D_z \tilde{\theta} - \nabla \rho \cdot \nabla \tilde{\theta}\right)\right|_{z = 0}.$$

First we reduce to paraproducts. Write
$$\zeta \D_z \tilde{\theta} - \nabla \rho \cdot \nabla \tilde{\theta} = T_\zeta \D_z \tilde{\theta} - T_{\nabla \rho} \nabla \tilde{\theta} + T_{\D_z \tilde{\theta}} \zeta - T_{\nabla \tilde{\theta}} \cdot \nabla \rho + R(\zeta, \D_z \tilde{\theta}) - R(\nabla \rho, \nabla \tilde{\theta}).$$
By (\ref{holderparaproduct0}),
\begin{align*}
\|T_{\D_z \tilde{\theta}} \zeta\|_{C_*^{r - \half}} &\lesssim \|\D_z \tilde{\theta}\|_{L^\infty} \|\zeta\|_{C_*^{r - \half}} \\
\|T_{\nabla \tilde{\theta}} \cdot \nabla \rho\|_{C_*^{r - \half}} &\lesssim \|\nabla \tilde{\theta}\|_{L^\infty} \|\nabla \rho\|_{C_*^{r - \half}}.
\end{align*}
Taking $L_z^\infty$ and using Proposition \ref{diffeoest} and Corollary \ref{higherw2}, the right hand sides are bounded by the right hand side of the desired estimate. Similarly, by (\ref{holderparaerror})
\begin{align*}
\|R(\zeta, \D_z \tilde{\theta})\|_{C_*^{r - \half}} &\lesssim \|\D_z \tilde{\theta}\|_{L^\infty} \|\zeta\|_{C_*^{r - \half}} \\
\|R(\nabla \rho, \nabla \tilde{\theta})\|_{C_*^{r - \half}} &\lesssim \|\nabla \tilde{\theta}\|_{L^\infty} \|\nabla \rho\|_{C_*^{r - \half}}.
\end{align*}

Second, we may replace the vertical derivative $\D_z \tilde{\theta}$ with $T_A \tilde{\theta}$ as a consequence of Corollary \ref{higherw2} along $z = 0$:
$$\|((\D_z - T_A )\tilde{\theta})|_{z = 0} \|_{C_*^{\half}} \leq \FF(\|\eta\|_{H^{s + \half}}, \|f\|_{H^s})(1 + \|\eta\|_{W^{r + \half, \infty}})(1 + \|\eta\|_{W^{r + \half, \infty}}+ \|f\|_{C_*^{r}}).$$
Thus (using $\zeta \in L^\infty$)
$$G(\eta)f = (T_\zeta T_A \tilde{\theta} - T_{\nabla \rho} \cdot \nabla \tilde{\theta})|_{z = 0} + R,$$
with the error $R$ satisfying
$$\|R\|_{C_*^{\half}} \leq \FF(\|\eta\|_{H^{s + \half}}, \|f\|_{H^s})(1 + \|\eta\|_{W^{r + \half, \infty}})(1 + \|\eta\|_{W^{r + \half, \infty}}+ \|f\|_{W^{r, \infty}}).$$

Lastly, by the symbolic calculus (\ref{holdercommutator}), Proposition \ref{diffeoest}, and Proposition \ref{aabds},
\begin{align*}
\|(T_\zeta T_A \tilde{\theta} - T_{\zeta A})\tilde{\theta}\|_{C_*^{r - \half}} &\lesssim (M_{r - \half}^0(\zeta)M_0^1(A) + M_0^0(\zeta)M_{r - \half}^1(A)) \|\tilde{\theta}\|_{C_*^{1}} \\
&\leq \FF(\|\eta\|_{H^{s + \half}}) \|\eta\|_{W^{r + \half, \infty}} \|\tilde{\theta}\|_{C_*^{1}}.
\end{align*}
We may exchange $\tilde{\theta}$ for $S_{>1/10} \tilde{\theta}$ in the previous inequalities by the inhomogeneous paradifferential calculus, and hence use Corollary \ref{higherw2} to conclude
$$\|(T_\zeta T_A \tilde{\theta} - T_{\zeta A})\tilde{\theta}\|_{C_*^{r - \half}} \leq \FF(\|\eta\|_{H^{s + \half}}, \|f\|_{H^s}) \|\eta\|_{W^{r + \half, \infty}}(1 + \|\eta\|_{W^{r + \half, \infty}} + \|f\|_{C_*^{r}}).$$
We thus may exchange $T_\zeta T_A \tilde{\theta}$ for $T_{\zeta A}\tilde{\theta}$ in the expression for $G(\eta)f$, with an error $R'$ satisfying the same estimates as $R$. We conclude, using that $\tilde{\theta}(0) = f$, 
$$G(\eta)f = T_{\zeta A - i\nabla \rho \cdot \xi}f + R + R'.$$
A routine computation shows that $(\zeta A - i\nabla \rho \cdot \xi)|_{z = 0} = \Lambda$ as desired.
\end{proof}

\subsection{Rough Bottom Estimates}\label{bottomsection}

In this section we study errors that arise due to the presence of a general bottom to our fluid domain. We recall the following identities from Propositions 4.3 and 4.5 in \cite{alazard2014cauchy}:
$$G(\eta)B = -\nabla \cdot V - \Gamma_y$$
$$(\D_t + V \cdot \nabla) \nabla \eta = G(\eta) V + \nabla \eta G(\eta) B + \Gamma_x + \nabla \eta \Gamma_y$$
where
$$\|\Gamma_x\|_{H^{s - \half}} + \|\Gamma_y\|_{H^{s - \half}} \leq \FF(\|\eta\|_{H^{s + \half}}, \|(\psi, V, B)\|_{H^\half}).$$
Here, $\Gamma_{x_i}$ is defined as follows: let $\theta_i$ be the solution to
$$\Delta_{x, y} \theta_i = 0, \qquad \theta_i|_{y = \eta(x)} = V_i, \qquad \D_n \theta_i|_{\Gamma} = 0.$$
Then 
\begin{equation}\label{gammadef}
\Gamma_{x_i} = ((\D_y - \nabla \eta \cdot \nabla) (\D_i \phi - \theta_i)) |_{y = \eta(x)}.
\end{equation}
Note that by the definition of $V$, $\D_i \phi = \theta_i$ if $\Gamma = \emptyset$, so $\Gamma_{x_i}$ is only nonzero in the presence of a bottom. $\Gamma_y$ is defined in the analogous way with $B$ in place of $V$.

We estimate $\Gamma_x$ and $\Gamma_y$ in H\"older norm. For this we apply the inhomogeneous elliptic H\"older estimates of Appendix \ref{ellipticsection}. 
\begin{lem}\label{higherelliptic}
Let $-1 < z_1 < z_0 < 0$. Denote $J_0 = [z_0, 0], J_1 = [z_1, 0]$. Consider $\tilde{\theta}$ solving (\ref{inhomogelliptic}) with $f = 0$. Then
\begin{align*}
\|\nabla_{x,z} \tilde{\theta}\|_{C^0(J_0;C^{\half}_*)} \leq \ &\FF(\|\eta\|_{H^{s + \half}})(1 + \|\eta\|_{W^{r + \half, \infty}})(\|F_0\|_{Y^{s -\half}(J_1)} + \|\nabla_{x,z} \tilde{\theta}\|_{X^{-\half}(J_1)}) \\
&+ \|F_0\|_{L^1(J_1;C_*^{\half})}.
\end{align*}
\end{lem}

\begin{proof}
Applying Proposition \ref{induction} with $\sigma = 1$ and $f = 0$, we have
\begin{align*}
\|\nabla_{x,z} \tilde{\theta}\|_{C^0(J_0;C^{\half}_*)} \leq \ &\FF(\|\eta\|_{H^{s + \half}})(1 + \|\eta\|_{W^{r + \half, \infty}})\|\nabla_{x, z} \tilde{\theta}\|_{U^{0}(J_0')} + \|F_0\|_{L^1(J_0';C_*^{\half})}.
\end{align*}
To estimate the right hand side, we apply Sobolev embedding and Proposition \ref{sobolevinductive} with $\sigma = s - \half$ and $f = 0$:
$$\|\nabla_{x, z} \tilde{\theta}\|_{U^{r - 1}(J_0')} \lesssim \|\nabla_{x,z} \tilde{\theta}\|_{X^{s - \half}(J_0')} \leq \FF(\|\eta\|_{H^{s + \half}})(\|F_0\|_{Y^{s -\half}(J_1)} + \|\nabla_{x,z} \tilde{\theta}\|_{X^{-\half}(J_1)}).$$
\end{proof}

The rest of the proof closely follows the proof of Proposition 4.3 in \cite{alazard2014cauchy} with the appropriate modifications to replace Sobolev with H\"older norms, but we provide the details here for completeness. We will need the following estimate:

\begin{lem}[{\cite[Lemma 3.11]{alazard2014cauchy}}]\label{smoothelliptic}
Let $-\half \leq a < b \leq -\frac{1}{5}$. Then the strip
$$S_{a, b} = \{(x, y) \in \R^{d + 1}: ah < y - \eta(x) < bh\}$$
is contained in $\Omega$ and for any $k \geq 1$, we have for any $\phi$ solving
$$\Delta_{x, y}\phi = 0, \qquad \phi|_{y = \eta(x)} = \psi,$$
the estimate
$$\|\phi\|_{H^k(S_{a, b})} \lesssim \|\psi\|_{H^{\half}(\R^d)}.$$

\end{lem}

\begin{prop}\label{roughbottomest}
Consider $\Gamma_x, \Gamma_y$ as defined in (\ref{gammadef}). Then
$$\|\Gamma_x\|_{W^{\half,\infty}} + \|\Gamma_y\|_{W^{\half,\infty}} \leq \FF(\|\eta\|_{H^{s + \half}}, \|(\psi, V, B)\|_{H^\half})(1 + \|\eta\|_{W^{r + \half, \infty}}).$$
\end{prop}

\begin{proof}
We prove the case of $\Gamma_x$ as the case of $\Gamma_y$ is similar.

First we localize the problem near the surface $\Gamma$, away from the general bottom. Let $\chi_0 \in C^\infty(\R), \eta_1 \in H^\infty(\R^d)$ be such that $\chi_0(z) = 1$ if $z \geq 0$, $\chi_0(z) = 0$ if $z \leq -\frac{1}{4}$, and
$$\eta(x) - \frac{h}{4} \leq \eta_1(x) \leq \eta(x) - \frac{h}{5}.$$
Set
$$U_i(x, y) = \chi_0\left( \frac{y - \eta_1(x)}{h}\right)(\D_i \phi - \theta_i)(x, y).$$
By construction, $\Gamma_{x_i} = ((\D_y - \nabla \eta \cdot \nabla)U_i) |_{y = \eta(x)}$. Moreover, $U_i$ satisfies
$$\Delta_{x, y} U_i = [\Delta_{x, y}, \chi_0\left( \frac{y - \eta_1(x)}{h}\right)](\D_i \phi - \theta_i) =: F_i, \qquad U_i|_{y = \eta(x)} = 0.$$
where
$$\supp \ F_i \subseteq S_{-\half, -\frac{1}{5}} = \left\{(x, y): x \in \R^d, \eta(x) - \frac{h}{2} \leq y \leq \eta(x) - \frac{h}{5}\right\}.$$
By Lemma \ref{smoothelliptic}, we have for arbitrary $\alpha \in \N^{d + 1}$,
$$\|\D_{x, y}^\alpha F_i\|_{L^\infty(S_{-\half, -\frac{1}{5}}) \cap L^2(S_{-\half, -\frac{1}{5}})} \lesssim \|(V, B)\|_{H^\half}$$
and in particular $F_i$ is smooth.

Next we flatten the boundary defined by the graph of $\eta$. Set 
$$\tilde{U}_i(x, z) = U_i(x, \rho(x, z)), \quad \tilde{F}_i = F_i(x, \rho(x, z))$$
where $\rho$ is the Lipschitz diffeomorphism discussed in Appendix \ref{ellipticsection}. In particular, the image of the strip $S_{-\half, -\frac{1}{5}}$ is a wider strip, but still away from zero. In turn, this implies $\tilde{F}_i$ is still smooth. We also have 
$$(\D_z^2 + \alpha \Delta + \beta \cdot \nabla \D_z - \gamma\D_z) \tilde{U}_i = \frac{(\D_z \rho)^2}{1 + |\nabla \rho|^2} \tilde{F}_i, \qquad \tilde{U}_i|_{z = 0} = 0.$$
Thus, apply Lemma \ref{higherelliptic} with $\tilde{\theta} = \tilde{U}_i$ and smooth inhomogeneity $\tilde{F}_i$:
\begin{align*}
\|\nabla_{x,z} \tilde{U}_i\|_{C^0(J_0;C^{\half}_*)} \leq \ &\FF(\|\eta\|_{H^{s + \half}}, \|(V, B)\|_{H^\half})(1 + \|\eta\|_{W^{r + \half, \infty}})(1 + \|\nabla_{x, z} \tilde{U}_i\|_{X^{-\half}(J_1)}).
\end{align*}
As noted in the proof of Proposition 4.3 in \cite{alazard2014cauchy},
$$\|\nabla_{x, z} \tilde{U}_i\|_{X^{-\half}(I)} \leq \FF(\|\eta\|_{H^{s + \half}}(\|\psi\|_{H^\half} + \|V_i\|_{H^\half})$$
so we conclude
\begin{align*}
\|\nabla_{x,z} \tilde{U}_i\|_{C^0(J_0;C^{\half}_*)} \leq \ &\FF(\|\eta\|_{H^{s + \half}}, \|(\psi, V, B)\|_{H^\half})(1 + \|\eta\|_{W^{r + \half, \infty}}).
\end{align*}

From the definition of $\rho$ in Section 3 of \cite{alazard2014cauchy}, 
$$\Gamma_{x_i} = \left.\left(\left( \frac{1 + |\nabla \eta|^2}{1 + \delta \langle D_x\rangle \eta} \D_z - \nabla \eta \cdot \nabla \right) \tilde{U}_i\right) \right|_{z = 0},$$
so we conclude by Sobolev embedding on $\nabla \eta$ that
$$\|\Gamma_{x_i}\|_{W^{\half, \infty}} \leq \FF(\|\eta\|_{H^{s + \half}}, \|(\psi, V, B)\|_{H^\half})(1 + \|\eta\|_{W^{r + \half, \infty}}).$$
\end{proof}

\subsection{Estimates on the Taylor Coefficient}

We recall H\"older estimates on the Taylor coefficient:
\begin{prop}[{\cite[Proposition C.1]{alazard2014strichartz}}]\label{taylorbd}
Remain in the setting of Proposition \ref{paralinearization}. Let $0 < \eps < \min(r - 1, s - \frac{d}{2} - \frac{3}{4})$. Then for all $t \in [0, T]$,
$$\|a(t) - g\|_{L^\infty} \lesssim \|a(t) - g\|_{H^{s - \half}} \leq \FF\left(\|(\eta, \psi, V, B)(t)\|_{H^{s + \half} \times H^{s + \half} \times H^s \times H^s} \right),$$
$$\|a(t)\|_{W^{\half + \eps, \infty}} \leq \FF(s, r)(t).$$
\end{prop}

We have the following straightforward consequence:
\begin{cor}\label{gammabd}
Remain in the setting of Proposition \ref{taylorbd} and fix multi-index $\beta$. Then for all $t \in [0, T]$, uniformly on $\{|\xi| = 1\}$,
$$\|\D_\xi^\beta\gamma(t, \cdot, \xi)\|_{L^\infty} \leq \FF\left(\|(\eta, \psi, V, B)(t)\|_{H^{s + \half} \times H^{s + \half} \times H^s \times H^s} \right),$$
$$\|\D_\xi^\beta \gamma(t, \cdot, \xi)\|_{W^{\half + \eps, \infty}} \leq \FF(s, r)(t).$$
\end{cor}
\begin{proof}
For simplicity consider the case $\beta = 0$; the general case is similar. 

Recall
$$\gamma = \sqrt{a \Lambda}.$$
The argument of the square root in the definition of $\Lambda$ remains away from the origin:
$$(1 + |\nabla \eta|^2)|\xi|^2 - (\nabla \eta \cdot \xi)^2 \geq 1,$$
so we may smoothly truncate the square root near the origin. By the Taylor sign condition, we have $a \geq c > 0$, so similarly the square root defining $\gamma$ may be truncated. We may then apply the Moser estimate (\ref{smoothholder}), and the Sobolev embedding
$$H^{\frac{d}{2} + \eps} \subseteq L^\infty$$
on $\nabla \eta$ (of which $\Lambda$ is a smooth function).
\end{proof}

\subsection{Vector Field Commutator Estimate}

In this section, we establish a version of Lemma 2.17 in \cite{alazard2014cauchy} for H\"older spaces. As in Appendix \ref{holderonDN}, our estimates will not be tame, but still useful toward studying the structure of $V$.

\begin{prop}\label{com}
For any $t \in I$, $s \geq 0$, and $\eps > -s$, 
\begin{align*}
\|\left((\D_t + V \cdot \nabla )T_p - T_p(\D_t + T_V \cdot \nabla)\right)&u(t)\|_{W^{s, \infty}} \\
\lesssim& \ M_0^m(p(t)) \|V(t)\|_{W^{1, \infty}} \|u(t)\|_{B^{s + m}_{\infty ,1}} \\
&+ M_0^m(p(t))\|V(t)\|_{W^{s + \eps, \infty}}\|u(t)\|_{B_{\infty, 1}^{1 + m - \eps}} \\
&+ M_0^m(\D_t p(t) + V \cdot \nabla p(t)) \|u(t)\|_{W^{s + m,\infty}}.
\end{align*}
\end{prop}

\begin{proof}
We follow \cite{alazard2014cauchy}, but exchange Sobolev norms for H\"older norms.

\subsubsection*{Step 1}

We study the special case when $m = 0, p = p(t, x)$. We first expand a term from the left hand side of our estimate using the product rule,
\begin{equation}\label{com2}
V \cdot \nabla T_pu = V \cdot T_{\nabla p}u + V T_p \cdot \nabla u.
\end{equation}

We can express the first term $V \cdot T_{\nabla p}u$ as $T_{V \cdot \nabla p}u$ up to two error terms:
\begin{align*}
T_{V \cdot \nabla p}u &= \sum_{\lambda = 0} S_{\leq \lambda/8} (V \cdot \nabla p)S_\lambda u = V \cdot \sum_{\lambda = 0} (S_{\leq \lambda/8}\nabla p)S_\lambda u + \sum_{\lambda = 0} \left([S_{\leq \lambda/8}, V]\cdot \nabla p\right)S_\lambda u \\
&= V \cdot T_{\nabla p}u + \sum_{\lambda = 0} \left([S_{\leq \lambda/8}\nabla, V]p\right)S_\lambda u - S_{\leq \lambda/8}((\nabla V)p)S_\lambda u.
\end{align*}
To bound the second error term, write
\begin{align*}
\|S_{\leq \lambda/8} ((\nabla V) p)S_\lambda u \|_{W^{s, \infty}} &\lesssim \lambda^{s}\|\nabla V\|_{L^\infty}\|p\|_{L^\infty}\|S_\lambda u\|_{L^\infty} \lesssim \|V\|_{W^{1,\infty}}\|p\|_{L^\infty}\|S_\lambda u\|_{W^{s, \infty}}.
\end{align*}
After summing in $\lambda$, this is bounded by the Besov norm on the right hand side of our desired estimate. 

For the commutator error term $([S_{\leq \lambda/8}\nabla, V]p)S_\lambda u$, first note we may replace $V$ with $S_{\leq \lambda/16}V$ by estimating the resulting two errors as follows, which sum in Besov norm:
\begin{align*}
\|(S_{\leq \lambda/8}\nabla((&S_{> \lambda/16}V) p))S_\lambda u\|_{W^{s, \infty}} \\
&\lesssim \lambda^{s + 1} \|(S_{> \lambda/16}V) p\|_{L^\infty} \|S_\lambda u\|_{L^\infty} \lesssim \|V\|_{W^{1, \infty}} \|p\|_{L^\infty} \|S_\lambda u\|_{W^{s, \infty}} \\
\|(S_{> \lambda/16}V)& (S_{\leq \lambda/8}\nabla p)S_\lambda u\|_{W^{s, \infty}}\\
&\lesssim \lambda^s\|S_{> \lambda/16}V\|_{L^\infty}\|(S_{\leq \lambda/8}\nabla p)S_\lambda u\|_{L^\infty} + \|S_{> \lambda/16}V\|_{W^{s, \infty}}\|(S_{\leq \lambda/8}\nabla p)S_\lambda u\|_{L^\infty} \\
&\lesssim \|V\|_{W^{1, \infty}}\| p\|_{L^\infty}\|S_\lambda u\|_{W^{s, \infty}} + \|V\|_{W^{s + \eps, \infty}}\| p\|_{L^\infty}\|S_\lambda u\|_{W^{1 - \eps, \infty}}.
\end{align*}
Here, for the second estimate we have used (\ref{holderproduct}). Then to bound the commutator term, note the symbol of $S_{\leq \lambda/8}\nabla$ may be expressed as
$$i\xi \varphi(8\xi/\lambda) = i\frac{\lambda}{8}\frac{8\xi}{\lambda} \varphi(8\xi/\lambda) = \lambda \tilde{\varphi}(\xi/\lambda)$$
so that its kernel may be expressed as $\lambda^{d + 1}m(\lambda x)$. Write, using the fundamental theorem of calculus,
\begin{align*}
[S_{\leq \lambda/8}\nabla, S_{\leq \lambda/16}V]p(x) &= \int \lambda^{d + 1}m(\lambda y) (S_{\leq \lambda/16}V(x - y) - S_{\leq \lambda/16}V(x)) p(x - y) \ dy \\
&= -\int_0^1 \int \lambda^{d + 1}m(\lambda y)y \cdot (\nabla S_{\leq \lambda/16}V)(x - ty)p(x - y) \ dydt.
\end{align*}
Since $\lambda^{d + 1}m(\lambda y)y$ is integrable independent of $\lambda$, we have by Minkowski,
$$\|[S_{\leq \lambda/8}\nabla, S_{\leq \lambda/16}V]p\|_{L^\infty} \lesssim \|V\|_{W^{1, \infty}} \|p\|_{L^\infty}$$
and hence
\begin{align}
\label{commutatorproof} \|\left([S_{\leq \lambda/8}\nabla,S_{\leq \lambda/16}V]p \right) S_\lambda u \|_{W^{s, \infty}} &\lesssim \lambda^s \|[S_{\leq \lambda/8}\nabla, S_{\leq \lambda/16}V]p\|_{L^\infty} \|S_\lambda u\|_{L^\infty} \\
&\lesssim \|V\|_{W^{1, \infty}} \|p\|_{L^\infty} \|S_\lambda u\|_{W^{s, \infty}} \nonumber
\end{align}
which sums in Besov norm. We conclude $V \cdot T_{\nabla p}u = T_{V \cdot \nabla p}u$ up to an allowed error.

Next we study the second term $V T_p \cdot \nabla u$ of (\ref{com2}). We will compare this with
\begin{align*}
T_p T_V \cdot \nabla u &= \sum_{\lambda = 0} (S_{\leq \lambda/8}p)S_\lambda\left(\sum_{\mu = 0} (S_{\leq \mu/8} V) \cdot S_{\mu}\nabla u\right) \\
&= \sum_{\lambda = 0} (S_{\leq \lambda/8}p)S_\lambda\left(\sum_{\mu = \lambda/8}^{8\lambda} (S_{\leq \mu/8} V) \cdot S_{\mu}\nabla u\right),
\end{align*}
noting that the $\mu$-indexed sum consists of only a a finite number of terms. Also note we may replace $S_{\leq \mu/8}V$ by $S_{\leq \lambda/16}V$ once we bound the following error (we use that $(S_{\leq \mu/8}- S_{\leq \lambda/16})V$ is one of seven possible projections at frequency $\lambda$, and let $\tilde{S}$ represent projections with expanded support):
\begin{align*}
\|(S_{\leq \lambda/8}p)S_\lambda(\sum_{\lambda/8 \leq \mu \leq 8\lambda} ((S_{\leq \mu/8}- S_{\leq \lambda/16})V) &\cdot S_{\mu}\nabla u)\|_{W^{s, \infty}} \\
&\lesssim \lambda^s \sum_{\lambda/8 \leq \mu \leq 8\lambda}  \|p\|_{L^\infty}\|\tilde{S}_\lambda V\|_{L^\infty}\lambda \|\tilde{S}_\lambda u\|_{L^\infty} \\
&\lesssim \|p\|_{L^\infty} \|V\|_{W^{1, \infty}}\|\tilde{S}_\lambda u\|_{W^{s, \infty}}
\end{align*}
which sums in Besov norm. Thus, we have up to this acceptable error, say $E$ (and using that $S_\lambda \tilde{S}_\lambda = S_\lambda$),
\begin{align*}
T_p T_V \cdot \nabla u &= \sum_{\lambda = 0} (S_{\leq \lambda/8}p)S_\lambda\left((S_{\leq \lambda/16} V)\cdot\sum_{\mu =\lambda/8}^{8\lambda}  S_{\mu}\nabla u\right) + E \\
&= \sum_{\lambda = 0} (S_{\leq \lambda/8}p)(S_{\leq \lambda/16} V)\cdot S_\lambda\nabla u + \sum_{\lambda = 0} (S_{\leq \lambda/8}p)[S_\lambda,S_{\leq \lambda/16} V]\cdot\tilde{S}_{\lambda}\nabla u + E.
\end{align*}
To bound the commutator error term, expand 
$$(S_{\leq \lambda/8}p)[S_\lambda,S_{\leq \lambda/16} V]\cdot\tilde{S}_{\lambda}\nabla u = (S_{\leq \lambda/8}p)[S_\lambda \nabla,S_{\leq \lambda/16} V]\cdot\tilde{S}_{\lambda} u - (S_{\leq \lambda/8}p)S_\lambda ((S_{\leq \lambda/16}\nabla V)\tilde{S}_{\lambda} u).$$
The second sub-term on the right is bounded by 
\begin{align*}
\|(S_{\leq \lambda/8}p)S_\lambda ((S_{\leq \lambda/16}\nabla V)\tilde{S}_{\lambda} u)\|_{W^{s, \infty}} &\lesssim \lambda^s\|p\|_{L^\infty}\|\nabla V\|_{L^\infty}\|\tilde{S}_{\lambda} u\|_{L^\infty} \\
&\lesssim  \|p\|_{L^\infty}\|V\|_{W^{1,\infty}}\|\tilde{S}_{\lambda} u\|_{W^{s, \infty}}
\end{align*}
which sums in Besov norm. The remaining commutator is bounded in a similar way as the previous commutator (\ref{commutatorproof}):
\begin{align*}
\|(S_{\leq \lambda/8}p)[S_\lambda \nabla,S_{\leq \lambda/16} V]\cdot\tilde{S}_{\lambda} u\|_{W^{s, \infty}} &\lesssim \lambda^s \|p\|_{L^\infty}\|[S_\lambda \nabla,S_{\leq \lambda/16} V]\cdot\tilde{S}_{\lambda} u\|_{L^\infty} \\
&\lesssim \lambda^s \|p\|_{L^\infty}\|V\|_{W^{1, \infty}}\|\tilde{S}_{\lambda} u\|_{L^\infty} \\
&\lesssim \|p\|_{L^\infty}\|V\|_{W^{1, \infty}}\|\tilde{S}_{\lambda} u\|_{W^{s, \infty}}
\end{align*}
which sums in Besov norm. We conclude
$$T_p T_V \cdot \nabla u = \sum_{\lambda = 0} (S_{\leq \lambda/8}p)(S_{\leq \lambda/16} V)\cdot S_\lambda\nabla u +E_1$$
where $E_1$ is an acceptable error.

We can conclude $V T_p \cdot \nabla u = T_p T_V \cdot \nabla u + E_2$ for an acceptable error $E_2$ once we expand
\begin{align*}
V T_p \cdot \nabla u &= V\sum_{\lambda = 0} (S_{\leq \lambda/8} p)S_\lambda \nabla u \\
&= \sum_{\lambda = 0}(S_{\leq \lambda/16}V) (S_{\leq \lambda/8} p)S_\lambda \nabla u + \sum_{\lambda = 0}(S_{> \lambda/16}V) (S_{\leq \lambda/8} p)S_\lambda \nabla u
\end{align*}
and bound, applying (\ref{holderproduct}) and summing in Besov norm,
\begin{align*}
\|&(S_{> \lambda/16}V) (S_{\leq \lambda/8} p)S_\lambda \nabla u\|_{W^{s, \infty}} \\
&\lesssim \lambda^s\|S_{> \lambda/16}V \|_{L^\infty} \|(S_{\leq \lambda/8} p)S_\lambda \nabla u\|_{L^\infty} + \|S_{> \lambda/16}V \|_{W^{s, \infty}} \|(S_{\leq \lambda/8} p)S_\lambda \nabla u\|_{L^\infty} \\
&\lesssim  \|V\|_{W^{1, \infty}}\| p\|_{L^\infty}\|S_\lambda u\|_{W^{s, \infty}}  + \|V\|_{W^{s + \eps, \infty}}\| p\|_{L^\infty}\|S_\lambda u\|_{W^{1 - \eps, \infty}}.
\end{align*}

We summarize our analysis of the right hand side of (\ref{com2}):
$$V \cdot \nabla T_pu = V \cdot T_{\nabla p}u + V T_p \cdot \nabla u =  T_{V \cdot \nabla p}u + T_p T_V \cdot \nabla u + E_3$$
where $E_3$ is an acceptable error. Then noting that
$$\D_t T_p u = T_{\D_t p} u + T_p \D_t u,$$
we have
$$(\D_t + V \cdot \nabla )T_pu  = T_{\D_t p + V \cdot \nabla p}u + T_p(\D_t + T_V \cdot \nabla) u + E_3.$$
Then by (\ref{ordernorm}),
$$\|T_{\D_t p + V \cdot \nabla p}u\|_{W^{s, \infty}} \lesssim \|\D_t p(t) + V \cdot \nabla p(t)\|_{L^\infty} \|u(t)\|_{W^{s, \infty}},$$
concluding the proof of the desired estimate for the case $m = 0, p = p(t, x)$.

Extending to general symbols follows Steps 2 and 3 of the proof of Lemma 2.17 in \cite{alazard2014cauchy}. For completeness we reproduce the steps here.

\subsubsection*{Step 2}

Consider the case when $p(t, x, \xi) = a(t, x)|\xi|^m h(\xi/|\xi|)$ with $h \in C^\infty(S^{d - 1})$. In this case, we have
\begin{align*}
(\D_t + V \cdot \nabla )T_p - T_p(\D_t + T_V \cdot \nabla) = &\left((\D_t + V \cdot \nabla )T_a - T_a(\D_t + T_V \cdot \nabla)\right)|D|^m h(D/|D|) \\
&+ T_a[|D|^m h(D/|D|), T_V] \cdot \nabla.
\end{align*}

For the second term, note by Theorem 6.1.4 in \cite{metivier2008differential},
\begin{align*}
\|T_a[|D|^m h(D/|D|), T_V] &\cdot \nabla u\|_{W^{s, \infty}}\\
&\lesssim \|a\|_{L^\infty}\|[|D|^m h(D/|D|), T_V] \cdot \nabla u\|_{W^{s, \infty}} \\
&\lesssim \|a\|_{L^\infty}\|V\|_{W^{1, \infty}} M_0^m\left(|\xi|^m h(\xi/|\xi|); [d/2] + 1 + 1\right)\| u\|_{W^{s + m, \infty}} \\
&\lesssim \|a\|_{L^\infty}\|V\|_{W^{1, \infty}} \|h\|_{W^{[d/2] + 2, \infty}}\| u\|_{W^{s + m, \infty}}
\end{align*}

For the first term, we can repeat the previous analysis, with $|D|^m h(D/|D|) u$ in the place of $u$, and $a$ in the place of $p$. In a typical error term, we have for instance
$$\|S_\lambda|D|^m h(D/|D|)u\|_{W^{s, \infty}} \lesssim \|h\|_{L^\infty}\|S_\lambda u\|_{W^{s + m, \infty}}$$
which still sums in Besov norm. The other terms are similar, so we conclude (using Sobolev embedding on $h$)
\begin{align*}
\|((\D_t + V \cdot \nabla )T_p - T_p(\D_t + T_V \cdot \nabla))u\|_{W^{s, \infty}} \lesssim \|h\|_{H^{d + 2}}(&\|a\|_{L^\infty} \|V\|_{W^{1, \infty}} \|u\|_{W^{s + m, \infty}} \\
&+ \|a\|_{L^\infty}\|V\|_{W^{s + \eps, \infty}}\|u\|_{W^{1 + m - \eps, \infty}} \\
&+ \|(\D_t  + V \cdot \nabla) a\|_{L^\infty} \|u\|_{H^{s + m}}).
\end{align*}

\subsubsection*{Step 3}

Finally, for general case, let $(h_\nu)_{\nu \in \N}$ be an orthonormal basis of $L^2(S^{d - 1})$ consisting of eigenfunctions of the Laplace-Beltrami operator $\Delta_\omega, \omega = \xi/|\xi|$. Then if $\Delta_\omega h_\nu = \lambda_\nu^2 h_\nu$, we have 
$$\lambda_\nu \sim \nu^{1/d}, \quad \|h\|_{H^{d + 2}(S^{d - 1})} \lesssim \lambda_\nu^{d + 2}.$$
We can write
$$p(t, x, \xi) = \sum_{\nu} a_\nu(t, x) |\xi|^m h_\nu(\xi), \quad a_\nu(t, x) = \int_{S^{d - 1}} p(t, x, \omega) \overline{h_\nu(\omega)} \ d\omega$$
and hence
\begin{align*}
a_\nu(t, x) &= \lambda^{-2(d + 2)}_\nu \int_{S^{d - 1}} \Delta_w^{d + 2} p(t, x, \omega) \overline{h_\nu(\omega)} \ d\omega \\
(\D_t + V \cdot \nabla)a_\nu(t, x) &= \lambda^{-2(d + 2)}_\nu \int_{S^{d - 1}} \Delta_w^{d + 2}(\D_t + V \cdot \nabla) p(t, x, \omega) \overline{h_\nu(\omega)} \ d\omega.
\end{align*}
Consequently,
\begin{align*}
\|a\|_{L^\infty} &\lesssim \lambda^{-2(d + 2)}_\nu M_0^m(p) \\
\|(\D_t + V \cdot \nabla)a\|_{L^\infty} &\lesssim \lambda^{-2(d + 2)}_\nu M_0^m((\D_t + V \cdot \nabla)p) 
\end{align*}
We conclude
\begin{align*}
\|((&\D_t + V \cdot \nabla )T_p - T_p(\D_t + T_V \cdot \nabla))u\|_{W^{s, \infty}} \\
\leq& \sum_\nu \|((\D_t + V \cdot \nabla )T_{a_\nu|\xi|^m h_\nu} - T_{a_\nu|\xi|^m h_\nu}(\D_t + T_V \cdot \nabla))u\|_{W^{s, \infty}} \\
\lesssim& \sum_\nu \|h\|_{H^{d + 2}}(\|a\|_{L^\infty} \|V\|_{W^{1, \infty}} \|u\|_{W^{s + m, \infty}} + \|a\|_{L^\infty}\|V\|_{W^{s + \eps, \infty}}\|u\|_{W^{1 + m - \eps, \infty}} \\
&+ \|(\D_t  + V \cdot \nabla) a\|_{L^\infty} \|u\|_{W^{s + m, \infty}}) \\
\lesssim& \sum_\nu \lambda_\nu^{-2(d + 2)} \lambda_\nu^{d + 2} (M_0^m(p) \|V\|_{W^{1, \infty}} \|u\|_{W^{s + m, \infty}} + M_0^m(p)\|V\|_{W^{s + \eps, \infty}}\|u\|_{W^{1 + m - \eps, \infty}} \\
&+ M_0^m((\D_t + V \cdot \nabla)p)\|u\|_{W^{s + m, \infty}}). \\
\end{align*}
We have the desired result by noting 
$$\sum_\nu\lambda_\nu^{-(d + 2)} \lesssim\sum_\nu \nu^{-1 - 2/d} < \infty.$$
\end{proof}

\section{Paradifferential Calculus}

For the reader's convenience, we provide an appendix of notation and estimates from Bony's paradifferential calculus. This is a subset of the appendix in \cite{alazard2014strichartz}.

\subsection{Notation}\label{paracalcnotation}

For $\rho = k + \sigma, \ k \in \N, \ \sigma \in (0, 1)$, denote by $W^{\rho, \infty}(\R^d)$ the space of functions whose derivatives up to order $k$ are bounded and uniformly H\"older continuous with exponent $\sigma$.

\begin{definition}
Given $\rho \in [0, 1]$ and $m \in \R$, let $\Gamma_\rho^m(\R^d)$ denote the space of locally bounded functions $a(x, \xi)$ on $\R^d \times (\R^d \backslash 0)$, which are $C^\infty$ functions of $\xi$ away from the origin and such that, for any $\alpha \in \N^d$ and any $\xi \neq 0$, the function $x \mapsto \D_\xi^\alpha a(x, \xi)$ is in $W^{\rho, \infty}(\R^d)$ and there exists a constant $C_\alpha$ such that on $\{|\xi| \geq \half\}$,
$$\|\D_\xi^\alpha a( \cdot, \xi)\|_{W^{\rho, \infty}(\R^d)} \leq C_\alpha (1 + |\xi|)^{m - |\alpha|}.$$
For $a \in \Gamma_\rho^m$, we define
$$M_\rho^m(a) = \sup_{|\alpha| \leq 1 + 2d + \rho} \sup_{|\xi| \geq 1/2} \|(1 + |\xi|)^{|\alpha| - m} \D_\xi^\alpha a(\cdot, \xi)\|_{W^{\rho, \infty}(\R^d)}.$$
\end{definition}

Given a symbol $a \in \Gamma_\rho^m(\R^d)$, define the (inhomogeneous) paradifferential operator $T_a$ by
$$\widehat{T_a u}(\xi) = (2\pi)^{-d}\int \chi(\xi - \eta, \eta)\widehat{a}(\xi - \eta, \eta) \psi(\eta) \widehat{u}(\eta) \, d\eta,$$
where $\widehat{a}(\theta, \xi)$ is the Fourier transform of $a$ with respect to the first variable, and $\chi$ and $\psi$ are two fixed $C^\infty$ functions satisfying, for small $0 < \eps_1 < \eps_2$,
\begin{equation*}
\begin{cases}
\psi(\eta) = 0 \quad \text{on } \{|\eta| \leq 1\} \\
\psi(\eta) = 1 \quad \text{on } \{|\eta| \geq 2\}, \\
\end{cases}
\end{equation*}
\begin{equation*}
\begin{cases}
\chi(\theta, \eta) = 1 \quad \text{on } \{|\theta| \leq \eps_1 |\eta|\} \\
\chi(\theta, \eta) = 0 \quad \text{on } \{|\theta| \geq \eps_2 |\eta|\}, \\
\end{cases}
\qquad |\D_\theta^\alpha \D_\eta^\beta \chi(\theta, \eta)| \leq c_{\alpha, \beta} (1 + |\eta|)^{-|\alpha| - |\beta|}.
\end{equation*}
The cutoff function $\chi$ can be chosen so that $T_a$ coincides with the usual definition  of a paraproduct (in terms of a Littlewood-Paley decomposition), where the symbol $a$ depends only on $x$.

\subsection{Symbolic calculus}

We shall use results from \cite{metivier2008differential} about operator norm estimates for the pseudodifferential symbolic calculus.

\begin{definition} 
Consider a dyadic decomposition of the identity:
$$I = S_0 + \sum_{\lambda = 1} S_\lambda.$$
If $s \in \R$, the Zygmund class $C_*^s(\R^d)$ is the space of tempered distributions $u$ such that
$$\|u\|_{C_*^s} := \sup_\lambda \lambda^s\|S_\lambda u\|_{L^\infty} < \infty.$$
\end{definition}

\begin{rem}
If $s > 0$ is not an integer, then $C_*^s(\R^d) = W^{s, \infty}(\R^d).$
\end{rem}

\begin{definition}
Let $m \in \R$. We say an operator $T$ is of order $m$ if for every $\mu \in \R$, it is bounded from $H^\mu$ to $H^{\mu - m}$ and from $C_*^\mu$ to $C_*^{\mu - m}$.
\end{definition}

The main features of the symbolic calculus for paradifferential operators are given by the following proposition:

\begin{prop}
Let $m \in \R$ and $\rho \in [0, 1]$.

i) If $a \in \Gamma_0^m(\R^d)$, then $T_a$ is of order $m$. Moreover, for all $\mu \in \R$,
\begin{equation}\label{ordernorm}
\|T_a\|_{H^\mu \rightarrow H^{\mu - m}} \lesssim M_0^m(a), \quad \|T_a\|_{C_*^\mu \rightarrow C_*^{\mu - m}} \lesssim M_0^m(a).
\end{equation}
ii) If $a \in \Gamma_\rho^m(\R^d)$ and $b \in \Gamma_\rho^{m'}(\R^d)$ then $T_a T_b - T_{ab}$ is of order $m + m' - \rho$. Moreover, for all $\mu \in \R$,
\begin{align}
\label{sobolevcommutator}
\|T_a T_b - T_{ab}\|_{H^\mu \rightarrow H^{\mu - m - m' + \rho}} \lesssim M_\rho^m(a) M_0^{m'}(b) + M_0^m(a) M_\rho^{m'}(b), \\
\label{holdercommutator}\|T_a T_b - T_{ab}\|_{C_*^\mu \rightarrow C_*^{\mu - m - m' + \rho}} \lesssim M_\rho^m(a) M_0^{m'}(b) + M_0^m(a) M_\rho^{m'}(b).
\end{align}
\end{prop}

We also need to consider paradifferential operators with negative regularity. As a consequence, we need to extend our previous definition.

\begin{definition}
For $m \in \R$ and $\rho < 0$, $\Gamma_\rho^m(\R^d)$ denotes the space of distributions $a(x, \xi)$ on $\R^d \times (\R^d \backslash 0)$ which are $C^\infty$ with respect to $\xi$ and such that, for all $\alpha \in \N^d$ and all $\xi \neq 0$, the function $x \mapsto \D_\xi^\alpha a(x, \xi)$ belongs to $C_*^\rho(\R^d)$ and there exists a constant $C_\alpha$ such that on $\{|\xi| \geq \half\}$,
$$\|\D_\xi^\alpha a(\cdot, \xi)\|_{C_*^\rho} \leq c_\alpha (1 + |\xi|)^{m - |\alpha|}.$$
For $a \in \Gamma_\rho^m$, we define
$$M_\rho^m(a) = \sup_{|\alpha| \leq \frac{3d}{2} + \rho + 1} \sup_{|\xi| \geq 1/2} \|(1 + |\xi|)^{|\alpha| - m} \D_\xi^\alpha a(\cdot, \xi)\|_{C_*^\rho(\R^d)}.$$
\end{definition}

We recall Proposition 2.12 in \cite{alazard2014cauchy} which is a generalization of (\ref{ordernorm}).

\begin{prop}
Let $\rho < 0, \ m \in \R$, and $a \in \dot{\Gamma}_\rho^m$. Then the operator $T_a$ is of order $m - \rho$:
\begin{equation}\label{negativeop}
\|T_a\|_{H^s \rightarrow H^{s - (m - \rho)}} \lesssim M_\rho^m(a), \quad \|T_a\|_{C_*^s \rightarrow C_*^{s - (m - \rho)}} \lesssim M_\rho^m(a).
\end{equation}
\end{prop}

\subsection{Paraproducts and product rules}

We recall here some properties of paraproducts, $T_a$ where $a(x, \xi) = a(x)$. A key feature is that one can define paraproducts for rough functions $a$ which do not belong to $L^\infty(\R^d)$ but merely $C_*^{-m}(\R^d)$ with $m > 0$.

\begin{definition}
Given two functions $a, b$ defined on $\R^d$, we define the remainder
$$R(a, u) = au - T_a u - T_u a.$$
\end{definition}

We record here various estimates about paraproducts (see Chapter 2 in \cite{bahouri2011fourier}).

\begin{prop}
i) Let $\alpha, \beta \in \R$. If $\alpha + \beta > 0$ then
\begin{align}
\|R(a, u)\|_{H^{\alpha + \beta - \frac{d}{2}}} &\lesssim \|a\|_{H^\alpha} \|u\|_{H^\beta} \\
\label{holderparaerror}\|R(a, u)\|_{C_*^{\alpha + \beta}} &\lesssim \|a\|_{C_*^\alpha} \|u\|_{C_*^\beta} \\
\|R(a, u)\|_{H^{\alpha + \beta}} &\lesssim \|a\|_{C_*^\alpha} \|u\|_{H^\beta}.
\end{align}
ii) Let $m > 0$ and $s \in \R$. Then
\begin{align}
\label{sobolevparaproduct}\|T_a u\|_{H^{s - m}} &\lesssim \|a\|_{C_*^{-m}} \|u\|_{H^s} \\
\label{holderparaproduct}\|T_a u\|_{C_*^{s - m}} &\lesssim \|a\|_{C_*^{-m}} \|u\|_{C_*^s} \\
\label{holderparaproduct0}\|T_a u\|_{C_*^s} &\lesssim \|a\|_{L^\infty} \|u\|_{C_*^s}.
\end{align}
\end{prop}

We have the following product estimates (for references and proofs, see \cite{alazard2014strichartz}):

\begin{prop}
i) Let $s \geq 0$. Then
\begin{align}
\|u_1 u_2\|_{H^s} &\lesssim \|u_1\|_{H^s} \|u_2\|_{L^\infty} + \|u_1\|_{L^\infty} \|u_2\|_{H^s} \\
\label{holderproduct}\|u_1 u_2\|_{C_*^s} &\lesssim \|u_1\|_{C_*^s} \|u_2\|_{L^\infty} + \|u_1\|_{L^\infty} \|u_2\|_{C_*^s}.
\end{align}
ii) Let $\beta > \alpha > 0$. Then
\begin{align}
\|u_1 u_2\|_{C_*^{-\alpha}} \lesssim \|u_1\|_{C_*^\beta} \|u_2\|_{C_*^{-\alpha}}.
\end{align}
iii) Let $s \geq 0$ and consider $F \in C^\infty(\C^N)$ such that $F(0) = 0$. Then there exists a non-decreasing function $\FF: \R_+ \rightarrow \R_+$ such that
\begin{equation}\label{smoothholder}
\|F(U)\|_{C_*^s} \leq \FF(\|U\|_{L^\infty}) \|U\|_{C_*^s}
\end{equation}
for any $U \in C_*^s(\R^d)^N.$
\end{prop}

We also have the following composition estimates:

\begin{prop}\label{chainprop1}
Let $0 < \alpha \leq 1$ and $f \in C^\alpha(\R^d)$. Let $\nabla g \in L^\infty(\R^d \rightarrow \R^d)$. Then
$$\|f(g(x))\|_{C^\alpha(\R^d)} \leq \|f\|_{C^\alpha}\|\nabla g\|_{L^\infty}^\alpha.$$
\end{prop}
\begin{proof}
Compute
$$\frac{|f(g(x)) - f(g(y))|}{|x - y|^\alpha} = \frac{|f(g(x)) - f(g(y))|}{|g(x) - g(y)|^\alpha}\frac{|g(x) - g(y)|^\alpha}{|x - y|^\alpha}$$
Then 
$$\sup_{x \neq y \in \R^d} \frac{|f(g(x)) - f(g(y))|}{|x - y|^\alpha} \leq \sup_{x \neq y} \frac{|f(g(x)) - f(g(y))|}{|g(x) - g(y)|^\alpha} \left(\sup_{x \neq y}\frac{|g(x) - g(y)|}{|x - y|}\right)^\alpha.$$

\end{proof}

\begin{prop}\label{chainprop2}
Let $0 < \alpha \leq 1$ and $a(x, \zeta) \in L^\infty(\R^n \times \R^m)$, smooth in $\zeta$, and $a(\cdot, \zeta) \in C^\alpha(\R^n)$ uniformly in $\zeta$. Let $f \in C^\alpha(\R^n \rightarrow \R^m)$. Then
$$\|a(x, f(x))\|_{C^\alpha_x(\R^n)} \lesssim \sup_y \|a(\cdot, y)\|_{C^\alpha(\R^n)} + \|\nabla_y a\|_{L^\infty_{x, y}(\R^n \times \R^m)} \|f\|_{C^\alpha(\R^n \rightarrow \R^m)}.$$
\end{prop}

\begin{proof}
Compute
$$\frac{a(x, f(x)) - a(y, f(y))}{|x - y|^\alpha} = \frac{a(x, f(x)) - a(y, f(x))}{|x - y|^\alpha} + \frac{a(y, f(x)) - a(y, f(y))}{f(x) - f(y)}\frac{f(x) - f(y)}{|x - y|^\alpha}.$$
Then take the suprema of both sides in $x \neq y \in \R^n$ as before.
\end{proof}

\bibliography{allbib}

\begin{thebibliography}{ABZ14b}

\bibitem[ABZ11a]{alazard2011water}
Thomas Alazard, Nicolas Burq, and Claude Zuily.
\newblock On the water-wave equations with surface tension.
\newblock {\em Duke Mathematical Journal}, 158(3):413--499, 2011.

\bibitem[ABZ11b]{alazard2011strichartz}
Thomas Alazard, Nicolas Burq, and Claude Zuily.
\newblock Strichartz estimates for water waves.
\newblock {\em Ann. Sci. Ec. Norm. Sup{\'e}r.(4)}, 44(5):855--903, 2011.

\bibitem[ABZ13]{alazard2013water}
Thomas Alazard, Nicolas Burq, and Claude Zuily.
\newblock The water-wave equations: from {Z}akharov to {E}uler.
\newblock In {\em Studies in phase space analysis with applications to PDEs},
  pages 1--20. Springer, 2013.

\bibitem[ABZ14a]{alazard2014cauchy}
Thomas Alazard, Nicolas Burq, and Claude Zuily.
\newblock On the {C}auchy problem for gravity water waves.
\newblock {\em Inventiones mathematicae}, 198(1):71--163, 2014.

\bibitem[ABZ14b]{alazard2014strichartz}
Thomas Alazard, Nicolas Burq, and Claude Zuily.
\newblock Strichartz estimates and the {C}auchy problem for the gravity water
  waves equations.
\newblock {\em Memoirs of the AMS, to appear}, 2014.

\bibitem[BC99]{bahouri1999equations}
Hajer Bahouri and Jean-Yves Chemin.
\newblock {\'E}quations d'ondes quasilin{\'e}aires et estimations de
  {S}trichartz.
\newblock {\em American Journal of Mathematics}, 121(6):1337--1377, 1999.

\bibitem[BCD11]{bahouri2011fourier}
Hajer Bahouri, Jean-Yves Chemin, and Rapha{\"e}l Danchin.
\newblock {\em Fourier analysis and nonlinear partial differential equations},
  volume 343.
\newblock Springer Science \& Business Media, 2011.

\bibitem[CHS10]{christianson2010strichartz}
Hans Christianson, Vera~Mikyoung Hur, and Gigliola Staffilani.
\newblock Strichartz estimates for the water-wave problem with surface tension.
\newblock {\em Communications in Partial Differential Equations},
  35(12):2195--2252, 2010.

\bibitem[CS93]{craig1993numerical}
Walter Craig and Catherine Sulem.
\newblock Numerical simulation of gravity waves.
\newblock {\em Journal of Computational Physics}, 108(1):73--83, 1993.

\bibitem[DPN15]{de2015paradifferential}
Thibault De~Poyferr{\'e} and Quang-Huy Nguyen.
\newblock A paradifferential reduction for the gravity-capillary waves system
  at low regularity and applications.
\newblock {\em arXiv preprint arXiv:1508.00326}, 2015.

\bibitem[dPN16]{de2016strichartz}
Thibault de~Poyferre and Quang-Huy Nguyen.
\newblock Strichartz estimates and local existence for the gravity-capillary
  waves with non-{L}ipschitz initial velocity.
\newblock {\em Journal of Differential Equations}, 261(1):396--438, 2016.

\bibitem[Ebi87]{ebin1987equations}
David~G Ebin.
\newblock The equations of motion of a perfect fluid with free boundary are not
  well posed.
\newblock {\em Communications in Partial Differential Equations},
  12(10):1175--1201, 1987.

\bibitem[HGIT17]{harrop2017finite}
Benjamin Harrop-Griffiths, Mihaela Ifrim, and Daniel Tataru.
\newblock Finite depth gravity water waves in holomorphic coordinates.
\newblock {\em Annals of PDE}, 3(1):4, 2017.

\bibitem[HIT16]{hunter2016two}
John~K Hunter, Mihaela Ifrim, and Daniel Tataru.
\newblock Two dimensional water waves in holomorphic coordinates.
\newblock {\em Communications in Mathematical Physics}, 346(2):483--552, 2016.

\bibitem[KT05]{koch2005dispersive}
Herbert Koch and Daniel Tataru.
\newblock Dispersive estimates for principally normal pseudodifferential
  operators.
\newblock {\em Communications on pure and applied mathematics}, 58(2):217--284,
  2005.

\bibitem[Lan05]{lannes2005well}
David Lannes.
\newblock Well-posedness of the water-waves equations.
\newblock {\em Journal of the American Mathematical Society}, 18(3):605--654,
  2005.

\bibitem[Lan13]{lannes2013water}
David Lannes.
\newblock The water waves problem.
\newblock {\em Mathematical surveys and monographs}, 188, 2013.

\bibitem[M{\'e}t08]{metivier2008differential}
Guy M{\'e}tivier.
\newblock Para-differential calculus and applications to the {C}auchy problem
  for nonlinear systems, {E}nnio de {G}iorgi {M}ath.
\newblock {\em Res. Center Publ., Edizione della Normale}, 2008.

\bibitem[MMT08]{marzuola2008wave}
Jeremy Marzuola, Jason Metcalfe, and Daniel Tataru.
\newblock Wave packet parametrices for evolutions governed by {PDO}'s with
  rough symbols.
\newblock {\em Proceedings of the American Mathematical Society},
  136(2):597--604, 2008.

\bibitem[Ngu15]{nguyen2015sharp}
Quang-Huy Nguyen.
\newblock Sharp {S}trichartz estimates for water waves systems.
\newblock {\em arXiv preprint arXiv:1512.02359}, 2015.

\bibitem[Ngu17]{nguyen2017sharp}
Huy~Quang Nguyen.
\newblock A sharp cauchy theory for the 2{D} gravity-capillary waves.
\newblock In {\em Annales de l'Institut Henri Poincare (C) Non Linear
  Analysis}. Elsevier, 2017.

\bibitem[ST05]{smith2005sharp}
Hart~F Smith and Daniel Tataru.
\newblock Sharp local well-posedness results for the nonlinear wave equation.
\newblock {\em Annals of Mathematics}, pages 291--366, 2005.

\bibitem[Tat00]{tataru2000strichartz}
Daniel Tataru.
\newblock Strichartz estimates for operators with nonsmooth coefficients and
  the nonlinear wave equation.
\newblock {\em American Journal of Mathematics}, 122(2):349--376, 2000.

\bibitem[Tat01]{tataru2001strichartz}
Daniel Tataru.
\newblock Strichartz estimates for second order hyperbolic operators with
  nonsmooth coefficients, {II}.
\newblock {\em American Journal of Mathematics}, 123(3):385--423, 2001.

\bibitem[Tat02]{tataru2002strichartz}
Daniel Tataru.
\newblock Strichartz estimates for second order hyperbolic operators with
  nonsmooth coefficients {III}.
\newblock {\em Journal of the American Mathematical Society}, 15(2):419--442,
  2002.

\bibitem[Tat04]{tataru2004phase}
Daniel Tataru.
\newblock Phase space transforms and microlocal analysis.
\newblock {\em Phase space analysis of partial differential equations},
  2:505--524, 2004.

\bibitem[Tay91]{taylor1991pseudodifferential}
Michael~E Taylor.
\newblock Pseudodifferential operators and linear {PDE}.
\newblock In {\em Pseudodifferential Operators and Nonlinear PDE}. Springer,
  1991.

\bibitem[Wu97]{wu1997well}
Sijue Wu.
\newblock Well-posedness in {S}obolev spaces of the full water wave problem in
  2-{D}.
\newblock {\em Inventiones mathematicae}, 130(1):39--72, 1997.

\bibitem[Wu99]{wu1999well}
Sijue Wu.
\newblock Well-posedness in {S}obolev spaces of the full water wave problem in
  3-{D}.
\newblock {\em Journal of the American Mathematical Society}, 12(2):445--495,
  1999.

\bibitem[Zak68]{zakharov1968stability}
Vladimir~E Zakharov.
\newblock Stability of periodic waves of finite amplitude on the surface of a
  deep fluid.
\newblock {\em Journal of Applied Mechanics and Technical Physics},
  9(2):190--194, 1968.

\end{thebibliography}
\bibliographystyle{alpha}

\end{document}